\documentclass[twoside, 11pt]{article}
\usepackage[top=.8
in, bottom=1in, left=1in, right=1in]{geometry}

\usepackage{enumerate,float,mathtools}

\usepackage{color, section, amsthm, textcase, setspace, amssymb, lineno, 
amsmath, amssymb, amsfonts, latexsym, fancyhdr, longtable, ulem}

\usepackage{tikz,tikz-3dplot}
\usetikzlibrary{decorations.markings}
\usetikzlibrary{arrows.meta}

\newtheorem{theorem}{Theorem}[section]
\newtheorem{lemma}[theorem]{Lemma}
\newtheorem{corollary}[theorem]{Corollary}

\theoremstyle{definition}

\newtheorem{example}[theorem]{Example}
\newtheorem{remark}[theorem]{Remark}
\newtheorem{case}{Case}

\newcommand{\Complex}{\mathbb C} 

\newcommand{\mf}{\mathfrak}

\newcommand{\ul}{\underline}
\newcommand{\Cn}{C_{\leq n}}
\newcommand{\dd}{\hspace{.1cm}|\hspace{.1cm}}
\newcommand{\ind}{{\rm ind \hspace{.1cm}}}
\newcommand{\rk}{{\rm rk \hspace{.1cm}}}
\newcommand{\A}{{\rm A}}
\newcommand{\B}{{\rm B}}
\newcommand{\C}{{\rm C}}
\newcommand{\D}{{\rm D}}
\newcommand{\E}{{\rm E}}
\newcommand{\F}{{\rm F}}
\newcommand{\G}{{\rm G}}
\newcommand{\X}{{\rm X}}
\newcommand{\ad}{{\rm ad \hspace{.05cm}}}
\newcommand{\lc}{\left\lceil}
\newcommand{\rc}{\right\rceil}
\newcommand{\lf}{\left\lfloor}
\newcommand{\rf}{\right\rfloor}

\newcommand{\dk}{\rm DK}

\newcommand{\sright}[1]{\sigma(\overrightarrow{#1})}
\newcommand{\sleft}[1]{\sigma(\overleftarrow{#1})}

\newcommand{\ceil}[1]{\left \lceil #1 \right \rceil }

\newcommand{\mir}{\rm mir}
\newcommand{\tright}[1]{\tau(\overrightarrow{#1})}
\newcommand{\tleft}[1]{\tau(\overleftarrow{#1})}

\begin{document}

\title{\bf The unbroken spectra of Frobenius seaweeds}

\author{Alex Cameron$^*$, Vincent E. Coll, Jr.$^{**}$, Matthew Hyatt$^{\dagger}$, and Colton Magnant$^{\dagger\dagger}$ }
\maketitle

\noindent
\textit{$^*$Department of Mathematics, Muhlenberg College, Allentown, PA, USA: alexcameron@muhlenberg.edu }\\
\textit{$^{**}$Department of Mathematics, Lehigh University, Bethlehem, PA, USA:  vec208@lehigh.edu}\\
\textit{$^{\dagger}$FactSet Research Systems, New York, NY, USA:  matthewdhyatt@gmail.com }\\
\textit{$^{\dagger\dagger}$
UPS of America, Inc. Atlanta, GA, USA. cmagnant@ups.com}

\begin{abstract}
\noindent
We show that if $\mathfrak{g}$ is a Frobenius seaweed, then the spectrum of the adjoint of a principal element consists of an unbroken set of integers whose multiplicities have a symmetric distribution.  Our methods are combinatorial.  
\end{abstract}




\section{Introduction}\label{sec:intro}

\textit{Notation: } All Lie algebras will be finite dimensional over $\Complex$, and the Lie multiplication will be denoted by $ [-,-]$.

\bigskip
The \textit{index} of a Lie algebra is an important algebraic invariant and, for \textit{seaweed algebras}, is bounded by the algebra's rank: $ \ind \mf{g} \leq \rk \mf{g}$, (see \textbf{\cite{dk}}). More formally, the index of a Lie algebra $\mf{g}$ is given by 

\[\ind \mf{g}=\min_{F\in \mf{g^*}} \dim  (\ker (B_F)),\]

\noindent where $F$ is a linear form on $\mf{g}$, and $B_F$ is the associated skew-symmetric bilinear \textit{Kirillov form}, defined by $B_F(x,y)=F([x,y])$ for $x,y\in \mf{g}$.  On a given $\mf{g}$, index-realizing functionals are called \textit{regular} and exist in profusion, being dense in both the Zariski and Euclidean topologies of $\mf{g}^*$.

Of particular interest are Lie algebras which have index zero. Such algebras are called \textit{Frobenius}
and have been studied extensively from the point of view of invariant theory \textbf{\cite{Ooms1}} and are of special interest in deformation and quantum group theory stemming from their connection with the classical Yang-Baxter equation (see \textbf{\cite{G1}} and \textbf{\cite{G2}}).
A regular functional $F$ on a Frobenius Lie algebra $\mathfrak{g}$ is called a \textit{Frobenius functional}; equivalently, $B_F(-,-)$ is non-degenerate.  Suppose $B_F(-,-)$ is non-degenerate and let $[F]$ be the matrix of $B_F(-,-)$ relative to some basis 
$\{x_1,\dots,x_n  \}$ of $\mathfrak{g}$.  In \textbf{\cite{BD}}, Belavin and Drinfel'd showed that   

\[
\sum_{i,j}[F]^{-1}_{ij}x_i\wedge x_j
\]

\noindent
is the infinitesimal of a \textit{Universal Deformation Formula} (UDF) based on $\mathfrak{g}$.  A UDF based on $\mathfrak{g}$ can be used to deform the universal enveloping algebra of $\mathfrak{g}$ and also the function space of any Lie group which contains $\mathfrak{g}$ in its Lie algebra of derivations.
Despite the existence proof of Belavin and Drinfel'd, 
only two UDF's are known:  the exponential and quasi-exponential.  These are based, respectively, on the abelian \textbf{\cite{UDF1}} and non-abelian \textbf{\cite{UDF2}} Lie algebras of dimension two (see also \textbf{\cite{Twist}}).

A Frobenius functional can be algorithmically produced as a by-product of the Kostant Cascade (see \textbf{\cite{Joseph5} and \cite{Kostant}}).  
If $F$ is a Frobenius functional on $\mathfrak{g}$, then the natural map $\mathfrak{g} \rightarrow \mathfrak{g}^*$ defined by $x \mapsto  F([x,-])$ is an isomorphism.  The image of $F$ under the inverse of this map is called a \textit{principal element} of $\mathfrak{g}$ and will be denoted $\widehat{F}$.  It is the unique element of $\mathfrak{g}$ such that 

$$
F\circ \ad \widehat{F}= F([\widehat{F},-]) = F.  
$$

As a consequence of Proposition 3.1 in \textbf{\cite{Ooms2}}, Ooms established that the spectrum of the adjoint of a principal element of a Frobenius Lie algebra is independent of the principal element chosen to compute it (see also \textbf{\cite{G2}}, Theorem 3).  Generally, the eigenvalues of ad$\widehat{F}$ can take on virtually any value (see \textbf{\cite{DIATTA}} for examples). But, in their formal study of principal elements \textbf{\cite{G3}}, Gerstenhaber and Giaquinto showed that if $\mathfrak{g}$ is a Frobenius seaweed subalgebra of $A_{n-1}=\mathfrak{sl}(n)$, then the spectrum of the adjoint of a principal element of $\mathfrak{g}$ consists entirely of integers.\footnote{Joseph used different methods to strongly extend this integrality result to all seaweed subalgebras of semisimple Lie algebras (see \textbf{\cite{Joseph2}} and \textbf{\cite{Joseph3}}).}  Subsequently, the last three of the current authors showed that this spectrum must actually be an \textit{unbroken} sequence of integers centered at one half \textbf{\cite{Coll typea}}.\footnote{The paper of Coll et al. appears as a follow-up to the \textit{Lett. in Math. Physics} article 
by Gerstenhaber and Giaquinto \textbf{\cite{G3}}, where they claim that the eigenvalues of the adjoint representation of a semisimple principal element of a Frobenius seaweed subalgebra of $\mathfrak{sl}(n)$ consists of an unbroken sequence of integers. However, M. Dufflo, in a private communication to those authors, indicated that their proof contained an error.} Moreover, the dimensions of the associated eigenspaces are shown to have a symmetric distribution.  

The goal of this paper is to establish the following theorem, which asserts that the above-described unbroken symmetric spectral phenomena for type A is exhibited in all seaweed algebras.

\begin{theorem}\label{thm:main}
If $\mathfrak{g}$ is a Frobenius seaweed and $\widehat{F}$ is a principal element of $\mathfrak{g}$, then the spectrum of $\ad \widehat{F}$ consists of an unbroken set of integers  
centered at one-half.  Moreover, the dimensions of the associated eigenspaces form a symmetric distribution. 
\end{theorem}

\noindent


\noindent
\begin{remark}
In the prequel to this article \textbf{\cite{Coll typea}}, the type-A unbroken symmetric spectrum result is established using a combinatorial argument based on the graph-theoretic meander construction of Dergachev and Kirillov \textbf{\cite{dk}}.  The combinatorial arguments  heavily leverage the results of \textbf{\cite{Coll typea}}, but the inductions here are predicated on the basis-independent ``orbit meander" construction of Joseph \textbf{\cite{Joseph2}}.
\end{remark}

\begin{remark} To establish Theorem \ref{thm:main},
we first combinatorially establish that the spectrum is symmetric about one-half. We hasten to add that Ooms had previously (1980) established the symmetry result for all Frobenius Lie algebras using a more algebraic approach \textbf{\cite{Ooms2}}, cf.  \textbf{\cite{G3}}.
\end{remark}


\begin{remark}
In  \textbf{\cite{DIATTA}} Diatta and Manga show that any Frobenius Lie algebra can be embedded into $\mathfrak{sl}(n)$ for some $n$.  They suggest that it would be interesting if one could find an obstruction to embedding the algebra as a seaweed.  The unbroken spectrum provides such an obstruction.
\end{remark}




\section{Seaweeds}\label{seaweeds}

Let $\mf{g}$ be a simple Lie algebra equipped with a triangular decomposition 
\begin{eqnarray} \label{triangular}
\mf{g}=\mf{u_+}\oplus\mf{h}\oplus\mf{u_-},
\end{eqnarray}
where $\mf{h}$ is a Cartan subalgebra of $\mf{g}$. Let $\Delta$ be its root system 
where $\Delta_{+}$ are the \textit{positive roots} on $\mf{u_+}$ and 
$\Delta_-$ are the \textit{negative roots} roots
on $\mf{u_-}$, and let $\Pi$ denote the set of \textit{simple roots} of $\mf{g}$. 
Given $\beta\in\Delta$,
let $\mf{g}_{\beta}$ denote its corresponding root space,
and let $x_\beta$ denote the element of weight $\beta$ in a 
Chevalley basis of $\mf{g}$. Given a subset 
$\pi_1\subseteq \Pi$, let $\mf{p}_{\pi_1}$ denote the parabolic subalgebra of
$\mf{g}$ generated by all $\mf{g}_{\beta}$ such that $-\beta\in\Pi$ or $\beta\in\pi_1$.
Such a parabolic subalgebra is called \textit{standard}  with respect to the Borel subalgebra 
$\mf{u_-}\oplus\mf{h}$, and it is known that every parabolic subalgebra is conjugate to exactly
one standard parabolic subalgebra.

Formation of a seaweed subalgebra of $\mathfrak{g}$ requires two weakly opposite parabolic
subalgebras, i.e., two parabolic subalgebras $\mf{p}$ and $\mf{p'}$ such that 
$\mf{p} + \mf{p'}= \mf{g}$. In this case,  $\mf{p}\cap\mf{p'}$ is called a \textit{seaweed}, or in the nomenclature of Joseph \textbf{\cite{Joseph2}}, a  \textit{biparabolic} subalgebra of $\mf{g}$.  Given a subset $\pi_2\subseteq \Pi$,
let $\mf{p}_{\pi_2}^-$ denote the parabolic subalgebra of
$\mf{g}$ generated by all $\mf{g}_{\beta}$ such that $\beta\in\Pi$ 
or $-\beta\in\pi_2$.
Given two subsets $\pi_1,\pi_2\subseteq\Pi$, we have 
$\mf{p}_{\pi_1}+\mf{p}_{\pi_2}=\mf{g}$.  We now define the seaweed

\begin{eqnarray*}
\mf{p}(\pi_1\dd \pi_2)=\mf{p}_{\pi_1}\cap\mf{p}_{\pi_2}^-,
 \end{eqnarray*}

\noindent
which is said to be \textit{standard} with respect to the triangular decomposition in (\ref{triangular}).
Any seaweed is conjugate to a standard one, so it suffices to work
with standard seaweeds only. Note that an arbitrary seaweed may be conjugate to more than one standard seaweed (see \textbf{\cite{Panyushev1}}). 

We will often assume that $\pi_1\cup\pi_2=\Pi$,
for if not then $\mf{p}(\pi_1\dd \pi_2)$ can be expressed
as a direct sum of seaweeds.
Additionally, we use superscripts and subscripts 
to specify the type and rank of the containing simple Lie algebra 
$\mf{g}$. For example $\mf{p}^{\C}_{n}(\pi_1\dd \pi_2)$ is a 
seaweed subalgebra of $C_n=\mathfrak{sp}(2n)$, the symplectic Lie algebra of rank $n$.

It will be convenient to visualize the simple roots of a seaweed 
by constructing a graph, which we call 
a ``split Dynkin diagram".  
Suppose $\mf{g}$ has rank $n$, and let $\Pi=\{\alpha_n,\dots ,\alpha_1\}$,
where $\alpha_1$ is the exceptional root for types B, C, and D. 
Draw two horizontal lines of $n$ vertices, say $v_n^+,\dots ,v_1^+$ on top
and $v_n^-,\dots ,v_1^-$ on the bottom.
Color $v_i^+$ black if $\alpha_i\in\pi_1$, color $v_i^-$ black if $\alpha_i\in\pi_2$,
and color all other vertices white. Furthermore, if $\alpha_i,\alpha_j\in\pi_1$ are
not orthogonal, connect $v_i^+$ and $v_j^+$ with an edge in the 
standard way used in Dynkin diagrams.
Do the same for bottom vertices according to the roots in $\pi_2$. A \textit{maximally connected component} of a split Dynkin diagram is defined in the obvious way, and such a component is of type B, C, or D if it contains the exceptional root $\alpha_1$; otherwise the component is of type A.  See Examples \ref{AExample} - \ref{DExample} for what will, res become our type-A, type-C, and type-D running examples, respectively.    

\begin{example}\label{AExample}
Define the seaweed 
$\mf{p}_{9}^\A(\Upsilon_1 \dd \Upsilon_2)$ by the following sets:
$$\Upsilon_1 = 
\{\alpha_9, \alpha_7, \alpha_6, \alpha_4, \alpha_3, \alpha_2, \alpha_1\} ~ \text{ and } ~ \Upsilon_2 = \{\alpha_9,\alpha_8,\alpha_7,\alpha_5,
\alpha_4,\alpha_3,\alpha_2,\alpha_1\}.$$  
See Figure \ref{Asdynkin} for the split Dynkin diagram of $\mf{p}_{9}^\A(\Upsilon_1 \dd \Upsilon_2)$.  
\end{example}

\begin{figure}[H]
\[\begin{tikzpicture}
[decoration={markings,mark=at position 0.6 with 
{\arrow{angle 90}{>}}}]

\draw (1,1) node[draw,circle,fill=black,minimum size=5pt,inner sep=0pt] (1+) {};
\draw (2,1) node[draw,circle,fill=white,minimum size=5pt,inner sep=0pt] (2+) {};
\draw (3,1) node[draw,circle,fill=black,minimum size=5pt,inner sep=0pt] (3+) {};
\draw (4,1) node[draw,circle,fill=black,minimum size=5pt,inner sep=0pt] (4+) {};
\draw (5,1) node[draw,circle,fill=white,minimum size=5pt,inner sep=0pt] (5+) {};
\draw (6,1) node[draw,circle,fill=black,minimum size=5pt,inner sep=0pt] (6+) {};
\draw (7,1) node[draw,circle,fill=black,minimum size=5pt,inner sep=0pt] (7+) {};
\draw (8,1) node[draw,circle,fill=black,minimum size=5pt,inner sep=0pt] (8+) {};
\draw (9,1) node[draw,circle,fill=black,minimum size=5pt,inner sep=0pt] (9+) {};

\draw (1,0) node[draw,circle,fill=black,minimum size=5pt,inner sep=0pt] (1-) {};
\draw (2,0) node[draw,circle,fill=black,minimum size=5pt,inner sep=0pt] (2-) {};
\draw (3,0) node[draw,circle,fill=black,minimum size=5pt,inner sep=0pt] (3-) {};
\draw (4,0) node[draw,circle,fill=white,minimum size=5pt,inner sep=0pt] (4-) {};
\draw (5,0) node[draw,circle,fill=black,minimum size=5pt,inner sep=0pt] (5-) {};
\draw (6,0) node[draw,circle,fill=black,minimum size=5pt,inner sep=0pt] (6-) {};
\draw (7,0) node[draw,circle,fill=black,minimum size=5pt,inner sep=0pt] (7-) {};
\draw (8,0) node[draw,circle,fill=black,minimum size=5pt,inner sep=0pt] (8-) {};
\draw (9,0) node[draw,circle,fill=black,minimum size=5pt,inner sep=0pt] (9-) {};

\node at (1,.5) {$\alpha_9$};
\node at (2,.5) {$\alpha_8$};
\node at (3,.5) {$\alpha_7$};
\node at (4,.5) {$\alpha_6$};
\node at (5,.5) {$\alpha_5$};
\node at (6,.5) {$\alpha_4$};
\node at (7,.5) {$\alpha_3$};
\node at (8,.5) {$\alpha_2$};
\node at (9,.5) {$\alpha_1$};

\draw (1-) to (3-);
\draw (5-) to (9-);
\draw (3+) to (4+);
\draw (6+) to (9+);

;\end{tikzpicture}\]
\caption{The split Dynkin diagram of $\mf{p}_9^\A(
\Upsilon_1 \dd \Upsilon_2)$}
\label{Asdynkin}
\end{figure}
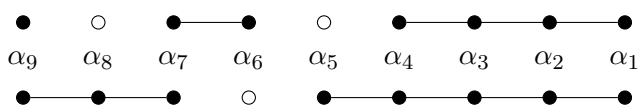

\begin{example}\label{BExample}
Define the seaweed 
$\mf{p}_{8}^\B(\Pi_1 \dd \Pi_2)$ by the following sets:
$$\Pi_1 = 
\{\alpha_8,\alpha_7,\alpha_6,\alpha_3, \alpha_2, \alpha_1\} ~ \text{ and } ~ \Pi_2 = \{\alpha_8,\alpha_7,\alpha_5,
\alpha_4,\alpha_3,\alpha_2\}.$$  
See Figure \ref{sBdynkin} for the split Dynkin diagram of $\mf{p}_{8}^\B(\Pi_1 \dd \Pi_2)$.  
\end{example}

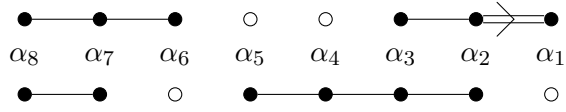
\begin{figure}[H]
\[\begin{tikzpicture}
[decoration={markings,mark=at position 0.6 with 
{\arrow{angle 90}{>}}}]

\draw (1,1) node[draw,circle,fill=black,minimum size=5pt,inner sep=0pt] (1+) {};
\draw (2,1) node[draw,circle,fill=black,minimum size=5pt,inner sep=0pt] (2+) {};
\draw (3,1) node[draw,circle,fill=black,minimum size=5pt,inner sep=0pt] (3+) {};
\draw (4,1) node[draw,circle,fill=white,minimum size=5pt,inner sep=0pt] (4+) {};
\draw (5,1) node[draw,circle,fill=white,minimum size=5pt,inner sep=0pt] (5+) {};
\draw (6,1) node[draw,circle,fill=black,minimum size=5pt,inner sep=0pt] (6+) {};
\draw (7,1) node[draw,circle,fill=black,minimum size=5pt,inner sep=0pt] (7+) {};
\draw (8,1) node[draw,circle,fill=black,minimum size=5pt,inner sep=0pt] (8+) {};

\draw (1,0) node[draw,circle,fill=black,minimum size=5pt,inner sep=0pt] (1-) {};
\draw (2,0) node[draw,circle,fill=black,minimum size=5pt,inner sep=0pt] (2-) {};
\draw (3,0) node[draw,circle,fill=white,minimum size=5pt,inner sep=0pt] (3-) {};
\draw (4,0) node[draw,circle,fill=black,minimum size=5pt,inner sep=0pt] (4-) {};
\draw (5,0) node[draw,circle,fill=black,minimum size=5pt,inner sep=0pt] (5-) {};
\draw (6,0) node[draw,circle,fill=black,minimum size=5pt,inner sep=0pt] (6-) {};
\draw (7,0) node[draw,circle,fill=black,minimum size=5pt,inner sep=0pt] (7-) {};
\draw (8,0) node[draw,circle,fill=white,minimum size=5pt,inner sep=0pt] (8-) {};

\node at (1,.5) {$\alpha_8$};
\node at (2,.5) {$\alpha_7$};
\node at (3,.5) {$\alpha_6$};
\node at (4,.5) {$\alpha_5$};
\node at (5,.5) {$\alpha_4$};
\node at (6,.5) {$\alpha_3$};
\node at (7,.5) {$\alpha_2$};
\node at (8,.5) {$\alpha_1$};

\draw (1-) to (2-);
\draw (4-) to (7-);
\draw (1+) to (3+);
\draw (6+) to (7+);
\draw [double distance=.8mm,postaction={decorate}] (7+) to (8+);

;\end{tikzpicture}\]
\caption{The split Dynkin diagram of $\mf{p}_8^\B(
\Pi_1 \dd \Pi_2)$}
\label{sBdynkin}
\end{figure}

\begin{example}\label{CExample}
Define the seaweed 
$\mf{p}_{8}^\C(\Pi_1 \dd \Pi_2)$ by the following sets:
$$\Pi_1 = 
\{\alpha_8,\alpha_7,\alpha_6,\alpha_3, \alpha_2, \alpha_1\} ~ \text{ and } ~ \Pi_2 = \{\alpha_8,\alpha_7,\alpha_5,
\alpha_4,\alpha_3,\alpha_2\}.$$  
See Figure \ref{sdynkin} for the split Dynkin diagram of $\mf{p}_{8}^\C(\Pi_1 \dd \Pi_2)$.  
\end{example}

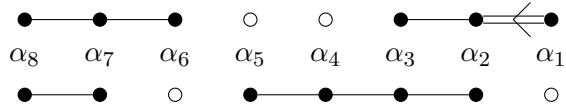
\begin{figure}[H]
\[\begin{tikzpicture}
[decoration={markings,mark=at position 0.6 with 
{\arrow{angle 90}{>}}}]

\draw (1,1) node[draw,circle,fill=black,minimum size=5pt,inner sep=0pt] (1+) {};
\draw (2,1) node[draw,circle,fill=black,minimum size=5pt,inner sep=0pt] (2+) {};
\draw (3,1) node[draw,circle,fill=black,minimum size=5pt,inner sep=0pt] (3+) {};
\draw (4,1) node[draw,circle,fill=white,minimum size=5pt,inner sep=0pt] (4+) {};
\draw (5,1) node[draw,circle,fill=white,minimum size=5pt,inner sep=0pt] (5+) {};
\draw (6,1) node[draw,circle,fill=black,minimum size=5pt,inner sep=0pt] (6+) {};
\draw (7,1) node[draw,circle,fill=black,minimum size=5pt,inner sep=0pt] (7+) {};
\draw (8,1) node[draw,circle,fill=black,minimum size=5pt,inner sep=0pt] (8+) {};

\draw (1,0) node[draw,circle,fill=black,minimum size=5pt,inner sep=0pt] (1-) {};
\draw (2,0) node[draw,circle,fill=black,minimum size=5pt,inner sep=0pt] (2-) {};
\draw (3,0) node[draw,circle,fill=white,minimum size=5pt,inner sep=0pt] (3-) {};
\draw (4,0) node[draw,circle,fill=black,minimum size=5pt,inner sep=0pt] (4-) {};
\draw (5,0) node[draw,circle,fill=black,minimum size=5pt,inner sep=0pt] (5-) {};
\draw (6,0) node[draw,circle,fill=black,minimum size=5pt,inner sep=0pt] (6-) {};
\draw (7,0) node[draw,circle,fill=black,minimum size=5pt,inner sep=0pt] (7-) {};
\draw (8,0) node[draw,circle,fill=white,minimum size=5pt,inner sep=0pt] (8-) {};

\node at (1,.5) {$\alpha_8$};
\node at (2,.5) {$\alpha_7$};
\node at (3,.5) {$\alpha_6$};
\node at (4,.5) {$\alpha_5$};
\node at (5,.5) {$\alpha_4$};
\node at (6,.5) {$\alpha_3$};
\node at (7,.5) {$\alpha_2$};
\node at (8,.5) {$\alpha_1$};

\draw (1-) to (2-);
\draw (4-) to (7-);
\draw (1+) to (3+);
\draw (6+) to (7+);
\draw [double distance=.8mm,postaction={decorate}] (8+) to (7+);

;\end{tikzpicture}\]
\caption{The split Dynkin diagram of $\mf{p}_8^\C(
\Pi_1 \dd \Pi_2)$}
\label{sdynkin}
\end{figure}


\begin{example}\label{DExample}
Define the seaweed $\mf{p}_{14}^\D(\Psi_1 \dd \Psi_2)$ by the following sets:
\begin{center}$\Psi_1 = 
\{\alpha_{14}, \alpha_{13}, \alpha_{12}, \alpha_{11}, \alpha_{10}, \alpha_9, \alpha_8, \alpha_7, \alpha_5, \alpha_4, \alpha_3, \alpha_2, \alpha_1\},$\\ and\\
$\Psi_2 = \{\alpha_{14}, \alpha_{13}, \alpha_{12}, \alpha_{11}, \alpha_9, \alpha_8, \alpha_7, \alpha_6, \alpha_5, \alpha_4, \alpha_3, \alpha_2 \}.$  
\end{center}

\bigskip
\noindent
See Figure \ref{Dsdynkin} for the split Dynkin diagram of $\mf{p}_{14}^\D(\Psi_1 \dd \Psi_2)$.  
\end{example}
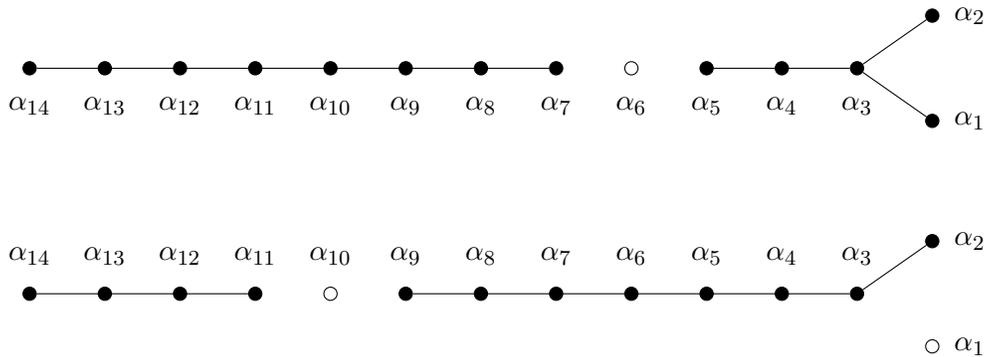
\begin{figure}[H]
\[\begin{tikzpicture}
[decoration={markings,mark=at position 0.6 with 
{\arrow{angle 90}{>}}}]

\draw (1,-3) node[draw,circle,fill=black,minimum size=5pt,inner sep=0pt] (1+) {};
\draw (2,-3) node[draw,circle,fill=black,minimum size=5pt,inner sep=0pt] (2+) {};
\draw (3,-3) node[draw,circle,fill=black,minimum size=5pt,inner sep=0pt] (3+) {};
\draw (4,-3) node[draw,circle,fill=black,minimum size=5pt,inner sep=0pt] (4+) {};
\draw (5,-3) node[draw,circle,fill=white,minimum size=5pt,inner sep=0pt] (5+) {};
\draw (6,-3) node[draw,circle,fill=black,minimum size=5pt,inner sep=0pt] (6+) {};
\draw (7,-3) node[draw,circle,fill=black,minimum size=5pt,inner sep=0pt] (7+) {};
\draw (8,-3) node[draw,circle,fill=black,minimum size=5pt,inner sep=0pt] (8+) {};
\draw (9,-3) node[draw,circle,fill=black,minimum size=5pt,inner sep=0pt] (9+) {};
\draw (10,-3) node[draw,circle,fill=black,minimum size=5pt,inner sep=0pt] (10+) {};
\draw (11,-3) node[draw,circle,fill=black,minimum size=5pt,inner sep=0pt] (11+) {};
\draw (12,-3) node[draw,circle,fill=black,minimum size=5pt,inner sep=0pt] (12+) {};
\draw (13,-2.3) node[draw,circle,fill=black,minimum size=5pt,inner sep=0pt] (13+) {};
\draw (13,-3.7) node[draw,circle,fill=white,minimum size=5pt,inner sep=0pt] (14+) {};

\draw (1,0) node[draw,circle,fill=black,minimum size=5pt,inner sep=0pt] (1-) {};
\draw (2,0) node[draw,circle,fill=black,minimum size=5pt,inner sep=0pt] (2-) {};
\draw (3,0) node[draw,circle,fill=black,minimum size=5pt,inner sep=0pt] (3-) {};
\draw (4,0) node[draw,circle,fill=black,minimum size=5pt,inner sep=0pt] (4-) {};
\draw (5,0) node[draw,circle,fill=black,minimum size=5pt,inner sep=0pt] (5-) {};
\draw (6,0) node[draw,circle,fill=black,minimum size=5pt,inner sep=0pt] (6-) {};
\draw (7,0) node[draw,circle,fill=black,minimum size=5pt,inner sep=0pt] (7-) {};
\draw (8,0) node[draw,circle,fill=black,minimum size=5pt,inner sep=0pt] (8-) {};
\draw (9,0) node[draw,circle,fill=white,minimum size=5pt,inner sep=0pt] (9-) {};
\draw (10,0) node[draw,circle,fill=black,minimum size=5pt,inner sep=0pt] (10-) {};
\draw (11,0) node[draw,circle,fill=black,minimum size=5pt,inner sep=0pt] (11-) {};
\draw (12,0) node[draw,circle,fill=black,minimum size=5pt,inner sep=0pt] (12-) {};
\draw (13,.7) node[draw,circle,fill=black,minimum size=5pt,inner sep=0pt] (13-) {};
\draw (13,-.7) node[draw,circle,fill=black,minimum size=5pt,inner sep=0pt] (14-) {};

\draw (1-) to (8-);
\draw (10-) to (12-);
\draw (12-) to (13-);
\draw (12-) to (14-);
\draw (1+) to (4+);
\draw (6+) to (12+);
\draw (12+) to (13+);

\node at (1,-2.5) {$\alpha_{14}$};
\node at (2,-2.5) {$\alpha_{13}$};
\node at (3,-2.5) {$\alpha_{12}$};
\node at (4,-2.5) {$\alpha_{11}$};
\node at (5,-2.5) {$\alpha_{10}$};
\node at (6,-2.5) {$\alpha_9$};
\node at (7,-2.5) {$\alpha_8$};
\node at (8,-2.5) {$\alpha_7$};
\node at (9,-2.5) {$\alpha_6$};
\node at (10,-2.5) {$\alpha_5$};
\node at (11,-2.5) {$\alpha_4$};
\node at (12,-2.5) {$\alpha_3$};
\node at (13.5,-2.3) {$\alpha_2$};
\node at (13.5,-3.7) {$\alpha_1$};

\node at (1,-.5) {$\alpha_{14}$};
\node at (2,-.5) {$\alpha_{13}$};
\node at (3,-.5) {$\alpha_{12}$};
\node at (4,-.5) {$\alpha_{11}$};
\node at (5,-.5) {$\alpha_{10}$};
\node at (6,-.5) {$\alpha_9$};
\node at (7,-.5) {$\alpha_8$};
\node at (8,-.5) {$\alpha_7$};
\node at (9,-.5) {$\alpha_6$};
\node at (10,-.5) {$\alpha_5$};
\node at (11,-.5) {$\alpha_4$};
\node at (12,-.5) {$\alpha_3$};
\node at (13.5,.7) {$\alpha_2$};
\node at (13.5,-.7) {$\alpha_1$};

;\end{tikzpicture}\]
\caption{The split Dynkin diagram for $\mf{p}_{14}^\D(\Psi_1 \dd \Psi_2)$}
\label{Dsdynkin}
\end{figure}

Given a seaweed $\mf{p}(\pi_1\dd\pi_2)$ of a simple 
Lie algebra $\mf{g}$, let $W$ denote the Weyl group of its
root system $\Delta$, generated by the reflections $s_\alpha$ 
such that $\alpha\in\Pi$. For $j=1,2$ define $W_{\pi_j}$ 
to be the subgroup of $W$ generated 
by $s_\alpha$ such that $\alpha\in\pi_j$. Let $w_j$
denote the unique longest (in the usual Coxeter sense) 
element of $W_{\pi_j}$,
and define an action
\[i_j\alpha=
\begin{cases}-w_j\alpha, & \text{ for all }\alpha\in\pi_j;\\
\alpha, & \text{ for all }\alpha\in\Pi\setminus\pi_j.
\end{cases}\]
Note that $i_j$ is an involution. 
For components of types B, C, and $D_k$ with $k$ even, the longest element $w_j = -id$.  However, if $k$ is odd,
\[w_j\alpha_i=
\begin{cases}
\alpha_2, & \text{ if  }i = 1;\\
\alpha_1, & \text{ if  }i = 2;\\
\alpha_i, & \text{ if }i \geq 3.
\end{cases}\]

To visualize the action of $i_j$, we append dashed edges
to the split Dynkin diagram of a seaweed. Specifically,
we draw a dashed edge from $v_i^+$ to $v_j^+$
if $i_1\alpha_i=\alpha_j$, and we draw a dashed  
edge from $v_i^-$ to $v_j^-$ if $i_2\alpha_i=\alpha_j$.
For simplicity, we will omit drawing the looped edge in the case
that $i_j\alpha_i=\alpha_i$.
We call the resulting graph the \textit{orbit meander}
of the associated seaweed.
See Figures \ref{fig:Aorbit meander} - \ref{fig:Dorbit meander}.

\begin{figure}[H]
\[\begin{tikzpicture}
[decoration={markings,mark=at position 0.6 with 
{\arrow{angle 90}{>}}}]

\draw (1,1) node[draw,circle,fill=black,minimum size=5pt,inner sep=0pt] (1+) {};
\draw (2,1) node[draw,circle,fill=white,minimum size=5pt,inner sep=0pt] (2+) {};
\draw (3,1) node[draw,circle,fill=black,minimum size=5pt,inner sep=0pt] (3+) {};
\draw (4,1) node[draw,circle,fill=black,minimum size=5pt,inner sep=0pt] (4+) {};
\draw (5,1) node[draw,circle,fill=white,minimum size=5pt,inner sep=0pt] (5+) {};
\draw (6,1) node[draw,circle,fill=black,minimum size=5pt,inner sep=0pt] (6+) {};
\draw (7,1) node[draw,circle,fill=black,minimum size=5pt,inner sep=0pt] (7+) {};
\draw (8,1) node[draw,circle,fill=black,minimum size=5pt,inner sep=0pt] (8+) {};
\draw (9,1) node[draw,circle,fill=black,minimum size=5pt,inner sep=0pt] (9+) {};

\draw (1,0) node[draw,circle,fill=black,minimum size=5pt,inner sep=0pt] (1-) {};
\draw (2,0) node[draw,circle,fill=black,minimum size=5pt,inner sep=0pt] (2-) {};
\draw (3,0) node[draw,circle,fill=black,minimum size=5pt,inner sep=0pt] (3-) {};
\draw (4,0) node[draw,circle,fill=white,minimum size=5pt,inner sep=0pt] (4-) {};
\draw (5,0) node[draw,circle,fill=black,minimum size=5pt,inner sep=0pt] (5-) {};
\draw (6,0) node[draw,circle,fill=black,minimum size=5pt,inner sep=0pt] (6-) {};
\draw (7,0) node[draw,circle,fill=black,minimum size=5pt,inner sep=0pt] (7-) {};
\draw (8,0) node[draw,circle,fill=black,minimum size=5pt,inner sep=0pt] (8-) {};
\draw (9,0) node[draw,circle,fill=black,minimum size=5pt,inner sep=0pt] (9-) {};


\draw (1-) to (3-);
\draw (5-) to (9-);
\draw (3+) to (4+);
\draw (6+) to (9+);

\draw [dashed] (3+) to [bend left=60] (4+);
\draw [dashed] (6+) to [bend left=60] (9+);
\draw [dashed] (7+) to [bend left=60] (8+);
\draw [dashed] (1-) to [bend right=60] (3-);
\draw [dashed] (5-) to [bend right=60] (9-);
\draw [dashed] (6-) to [bend right=60] (8-);

;\end{tikzpicture}\]
\caption{The orbit meander of $\mf{p}_9^\A(
\Upsilon_1 \dd \Upsilon_2)$}
\label{fig:Aorbit meander}
\end{figure}
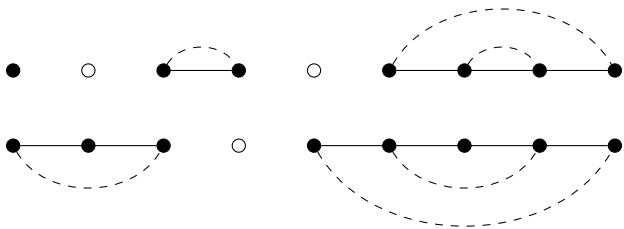

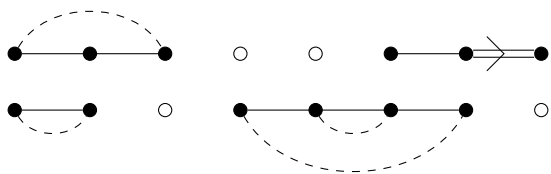
\begin{figure}[H]
\[\begin{tikzpicture}
[decoration={markings,mark=at position 0.6 with 
{\arrow{angle 90}{>}}}]

\draw (1,.75) node[draw,circle,fill=black,minimum size=5pt,inner sep=0pt] (1+) {};
\draw (2,.75) node[draw,circle,fill=black,minimum size=5pt,inner sep=0pt] (2+) {};
\draw (3,.75) node[draw,circle,fill=black,minimum size=5pt,inner sep=0pt] (3+) {};
\draw (4,.75) node[draw,circle,fill=white,minimum size=5pt,inner sep=0pt] (4+) {};
\draw (5,.75) node[draw,circle,fill=white,minimum size=5pt,inner sep=0pt] (5+) {};
\draw (6,.75) node[draw,circle,fill=black,minimum size=5pt,inner sep=0pt] (6+) {};
\draw (7,.75) node[draw,circle,fill=black,minimum size=5pt,inner sep=0pt] (7+) {};
\draw (8,.75) node[draw,circle,fill=black,minimum size=5pt,inner sep=0pt] (8+) {};

\draw (1,0) node[draw,circle,fill=black,minimum size=5pt,inner sep=0pt] (1-) {};
\draw (2,0) node[draw,circle,fill=black,minimum size=5pt,inner sep=0pt] (2-) {};
\draw (3,0) node[draw,circle,fill=white,minimum size=5pt,inner sep=0pt] (3-) {};
\draw (4,0) node[draw,circle,fill=black,minimum size=5pt,inner sep=0pt] (4-) {};
\draw (5,0) node[draw,circle,fill=black,minimum size=5pt,inner sep=0pt] (5-) {};
\draw (6,0) node[draw,circle,fill=black,minimum size=5pt,inner sep=0pt] (6-) {};
\draw (7,0) node[draw,circle,fill=black,minimum size=5pt,inner sep=0pt] (7-) {};
\draw (8,0) node[draw,circle,fill=white,minimum size=5pt,inner sep=0pt] (8-) {};

\draw (1-) to (2-);
\draw (4-) to (7-);
\draw (1+) to (3+);
\draw (6+) to (7+);
\draw [double distance=.8mm,postaction={decorate}] (7+) to (8+);

\draw [dashed] (1+) to [bend left=60] (3+);
\draw [dashed] (1-) to [bend right=60] (2-);
\draw [dashed] (4-) to [bend right=60] (7-);
\draw [dashed] (5-) to [bend right=60] (6-);

;\end{tikzpicture}\]
\caption{The orbit meander of $\mf{p}_8^\B(
\Pi_1
\dd \Pi_2)$}
\label{fig:Borbit meander}
\end{figure}

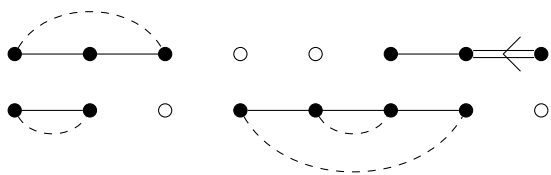
\begin{figure}[H]
\[\begin{tikzpicture}
[decoration={markings,mark=at position 0.6 with 
{\arrow{angle 90}{>}}}]

\draw (1,.75) node[draw,circle,fill=black,minimum size=5pt,inner sep=0pt] (1+) {};
\draw (2,.75) node[draw,circle,fill=black,minimum size=5pt,inner sep=0pt] (2+) {};
\draw (3,.75) node[draw,circle,fill=black,minimum size=5pt,inner sep=0pt] (3+) {};
\draw (4,.75) node[draw,circle,fill=white,minimum size=5pt,inner sep=0pt] (4+) {};
\draw (5,.75) node[draw,circle,fill=white,minimum size=5pt,inner sep=0pt] (5+) {};
\draw (6,.75) node[draw,circle,fill=black,minimum size=5pt,inner sep=0pt] (6+) {};
\draw (7,.75) node[draw,circle,fill=black,minimum size=5pt,inner sep=0pt] (7+) {};
\draw (8,.75) node[draw,circle,fill=black,minimum size=5pt,inner sep=0pt] (8+) {};

\draw (1,0) node[draw,circle,fill=black,minimum size=5pt,inner sep=0pt] (1-) {};
\draw (2,0) node[draw,circle,fill=black,minimum size=5pt,inner sep=0pt] (2-) {};
\draw (3,0) node[draw,circle,fill=white,minimum size=5pt,inner sep=0pt] (3-) {};
\draw (4,0) node[draw,circle,fill=black,minimum size=5pt,inner sep=0pt] (4-) {};
\draw (5,0) node[draw,circle,fill=black,minimum size=5pt,inner sep=0pt] (5-) {};
\draw (6,0) node[draw,circle,fill=black,minimum size=5pt,inner sep=0pt] (6-) {};
\draw (7,0) node[draw,circle,fill=black,minimum size=5pt,inner sep=0pt] (7-) {};
\draw (8,0) node[draw,circle,fill=white,minimum size=5pt,inner sep=0pt] (8-) {};

\draw (1-) to (2-);
\draw (4-) to (7-);
\draw (1+) to (3+);
\draw (6+) to (7+);
\draw [double distance=.8mm,postaction={decorate}] (8+) to (7+);

\draw [dashed] (1+) to [bend left=60] (3+);
\draw [dashed] (1-) to [bend right=60] (2-);
\draw [dashed] (4-) to [bend right=60] (7-);
\draw [dashed] (5-) to [bend right=60] (6-);

;\end{tikzpicture}\]
\caption{The orbit meander of $\mf{p}_8^\C(
\Pi_1
\dd \Pi_2)$}
\label{fig:orbit meander}
\end{figure}

\begin{figure}[H]
\[\begin{tikzpicture}
[decoration={markings,mark=at position 0.6 with 
{\arrow{angle 90}{>}}}]

\draw (1,-3) node[draw,circle,fill=black,minimum size=5pt,inner sep=0pt] (1+) {};
\draw (2,-3) node[draw,circle,fill=black,minimum size=5pt,inner sep=0pt] (2+) {};
\draw (3,-3) node[draw,circle,fill=black,minimum size=5pt,inner sep=0pt] (3+) {};
\draw (4,-3) node[draw,circle,fill=black,minimum size=5pt,inner sep=0pt] (4+) {};
\draw (5,-3) node[draw,circle,fill=white,minimum size=5pt,inner sep=0pt] (5+) {};
\draw (6,-3) node[draw,circle,fill=black,minimum size=5pt,inner sep=0pt] (6+) {};
\draw (7,-3) node[draw,circle,fill=black,minimum size=5pt,inner sep=0pt] (7+) {};
\draw (8,-3) node[draw,circle,fill=black,minimum size=5pt,inner sep=0pt] (8+) {};
\draw (9,-3) node[draw,circle,fill=black,minimum size=5pt,inner sep=0pt] (9+) {};
\draw (10,-3) node[draw,circle,fill=black,minimum size=5pt,inner sep=0pt] (10+) {};
\draw (11,-3) node[draw,circle,fill=black,minimum size=5pt,inner sep=0pt] (11+) {};
\draw (12,-3) node[draw,circle,fill=black,minimum size=5pt,inner sep=0pt] (12+) {};
\draw (13,-2.3) node[draw,circle,fill=black,minimum size=5pt,inner sep=0pt] (13+) {};
\draw (13,-3.7) node[draw,circle,fill=white,minimum size=5pt,inner sep=0pt] (14+) {};

\draw (1,0) node[draw,circle,fill=black,minimum size=5pt,inner sep=0pt] (1-) {};
\draw (2,0) node[draw,circle,fill=black,minimum size=5pt,inner sep=0pt] (2-) {};
\draw (3,0) node[draw,circle,fill=black,minimum size=5pt,inner sep=0pt] (3-) {};
\draw (4,0) node[draw,circle,fill=black,minimum size=5pt,inner sep=0pt] (4-) {};
\draw (5,0) node[draw,circle,fill=black,minimum size=5pt,inner sep=0pt] (5-) {};
\draw (6,0) node[draw,circle,fill=black,minimum size=5pt,inner sep=0pt] (6-) {};
\draw (7,0) node[draw,circle,fill=black,minimum size=5pt,inner sep=0pt] (7-) {};
\draw (8,0) node[draw,circle,fill=black,minimum size=5pt,inner sep=0pt] (8-) {};
\draw (9,0) node[draw,circle,fill=white,minimum size=5pt,inner sep=0pt] (9-) {};
\draw (10,0) node[draw,circle,fill=black,minimum size=5pt,inner sep=0pt] (10-) {};
\draw (11,0) node[draw,circle,fill=black,minimum size=5pt,inner sep=0pt] (11-) {};
\draw (12,0) node[draw,circle,fill=black,minimum size=5pt,inner sep=0pt] (12-) {};
\draw (13,.7) node[draw,circle,fill=black,minimum size=5pt,inner sep=0pt] (13-) {};
\draw (13,-.7) node[draw,circle,fill=black,minimum size=5pt,inner sep=0pt] (14-) {};

\draw (1-) to (8-);
\draw (10-) to (12-);
\draw (12-) to (13-);
\draw (12-) to (14-);
\draw (1+) to (4+);
\draw (6+) to (12+);
\draw (12+) to (13+);



\draw [dashed] (1+) to [bend right=60] (4+);
\draw [dashed] (2+) to [bend right=60] (3+);
\draw [dashed] (6+) to [bend right=80] (13+);
\draw [dashed] (7+) to [bend right=60] (12+);
\draw [dashed] (8+) to [bend right=60] (11+);
\draw [dashed] (9+) to [bend right=60] (10+);

\draw [dashed] (1-) to [bend left=60] (8-);
\draw [dashed] (2-) to [bend left=60] (7-);
\draw [dashed] (3-) to [bend left=60] (6-);
\draw [dashed] (4-) to [bend left=60] (5-);
\draw [dashed] (13-) to [bend left=60] (14-);

;\end{tikzpicture}\]
\caption{The orbit meander of $\mf{p}_{14}^\D(\Psi_1 \dd \Psi_2)$}
\label{fig:Dorbit meander}
\end{figure}
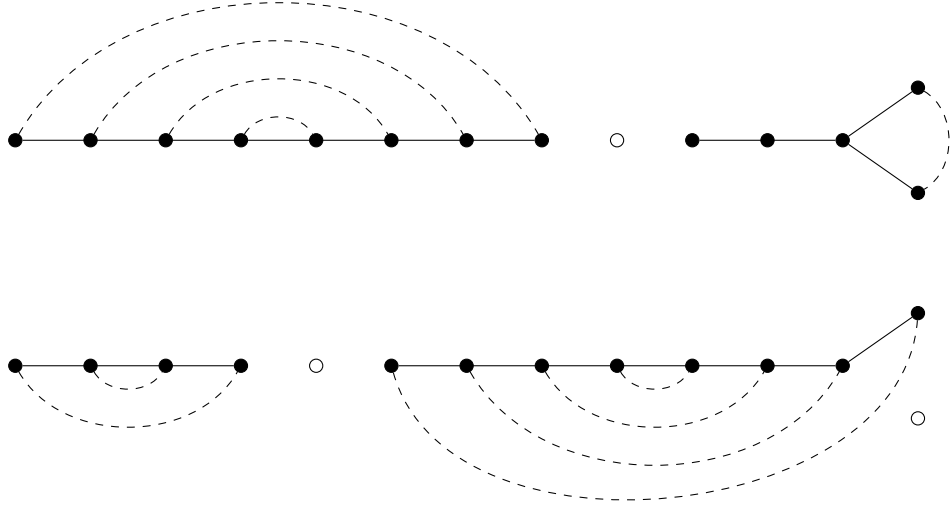

It turns out that the index of $\mf{p}(\pi_1\dd\pi_2)$
is governed by the orbits of the cyclic group $<i_1i_2>$
acting on $\Pi$.

\begin{theorem}[Joseph \textbf{\cite{Joseph4}}, Lemma 4.2]\label{thm:frobenius}
Given subsets $\pi_1,\pi_2\subseteq\Pi$ such that $\pi_1\cup\pi_2=\Pi$, 
let $\pi_{\cup}=\Pi\setminus (\pi_1\cap\pi_2)$.
The seaweed $\mf{p}(\pi_1\dd\pi_2)$ is Frobenius if and only if
every $<i_1i_2>$ orbit contains 
exactly one element from $\pi_{\cup}$.
\end{theorem}

\begin{example}
We show that the seaweed in each of our running examples is Frobenius.  
The type-A seaweed $\mf{p}_{9}^\A(\Upsilon_1 \dd \Upsilon_2)$ is Frobenius as its $<i_1i_2>$ orbits are $\{\alpha_9, \alpha_6, \alpha_7\}$, $\{\alpha_8\}$, $\{\alpha_5,\alpha_4,\alpha_3,\alpha_2,\alpha_1\}$, and each orbit contains exactly one element
from $\pi_{\cup}=\{\alpha_8,\alpha_6,\alpha_5\}$.  
The type-B and type-C seaweeds are Frobenius as each seaweed's $<i_1i_2>$ orbits are
$\{\alpha_7,\alpha_6,\alpha_8\}$, $\{\alpha_5,\alpha_2\}$, $\{\alpha_4,\alpha_3\}$, $\{\alpha_1\}$, and each orbit contains exactly one element
from $\pi_{\cup}=\{\alpha_6,\alpha_5,\alpha_4,\alpha_1\}$.  
Similarly, the type-D seaweed $\mf{p}_{14}^\D(\Psi_1 \dd \Psi_2)$ is Frobenius as its $<i_1i_2>$ orbits are
$\{\alpha_6, \alpha_5\}$, $\{\alpha_{14}, \alpha_{11}, \alpha_{10}, \alpha_7, \alpha_4\}$, and $\{\alpha_{13}, \alpha_{12}, \alpha_9, \alpha_8, \alpha_3, \alpha_2, \alpha_1\}$, each of which contains exactly one element
from $\pi_{\cup}=\{\alpha_{10}, \alpha_6,\alpha_1\}$.
\end{example}

\begin{example} 
For an example of a seaweed that is not Frobenius, consider
the orbit meander of $\mf{p}_7^\A(\{\alpha_7,\alpha_6,\alpha_5,
\alpha_4,\alpha_3,\alpha_2\}
\dd \{\alpha_7,\alpha_6,\alpha_4,\alpha_3,\alpha_2,\alpha_1\})$ shown in Figure \ref{fig:orbit meander nonfrob} below.
The $<i_1i_2>$ orbits $\{\alpha_7,\alpha_3\}$ and 
$\{\alpha_6,\alpha_2\}$ contain no elements from
$\pi_{\cup} = \{\alpha_5, \alpha_1\}$, and the orbit $\{\alpha_5,\alpha_4,\alpha_1\}$ contains
two elements from $\pi_{\cup}$.
\end{example}

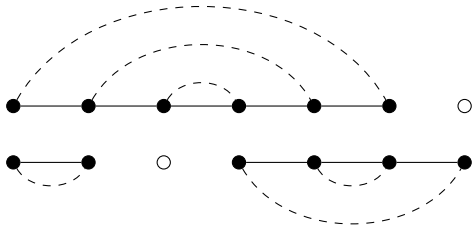
\begin{figure}[H]
\[\begin{tikzpicture}
[decoration={markings,mark=at position 0.6 with 
{\arrow{angle 90}{>}}}]

\draw (1,.75) node[draw,circle,fill=black,minimum size=5pt,inner sep=0pt] (1+) {};
\draw (2,.75) node[draw,circle,fill=black,minimum size=5pt,inner sep=0pt] (2+) {};
\draw (3,.75) node[draw,circle,fill=black,minimum size=5pt,inner sep=0pt] (3+) {};
\draw (4,.75) node[draw,circle,fill=black,minimum size=5pt,inner sep=0pt] (4+) {};
\draw (5,.75) node[draw,circle,fill=black,minimum size=5pt,inner sep=0pt] (5+) {};
\draw (6,.75) node[draw,circle,fill=black,minimum size=5pt,inner sep=0pt] (6+) {};
\draw (7,.75) node[draw,circle,fill=white,minimum size=5pt,inner sep=0pt] (7+) {};

\draw (1,0) node[draw,circle,fill=black,minimum size=5pt,inner sep=0pt] (1-) {};
\draw (2,0) node[draw,circle,fill=black,minimum size=5pt,inner sep=0pt] (2-) {};
\draw (3,0) node[draw,circle,fill=white,minimum size=5pt,inner sep=0pt] (3-) {};
\draw (4,0) node[draw,circle,fill=black,minimum size=5pt,inner sep=0pt] (4-) {};
\draw (5,0) node[draw,circle,fill=black,minimum size=5pt,inner sep=0pt] (5-) {};
\draw (6,0) node[draw,circle,fill=black,minimum size=5pt,inner sep=0pt] (6-) {};
\draw (7,0) node[draw,circle,fill=black,minimum size=5pt,inner sep=0pt] (7-) {};

\draw (1+) to (2+);
\draw (2+) to (3+);
\draw (3+) to (4+);
\draw (4+) to (5+);
\draw (5+) to (6+);
\draw (1-) to (2-);
\draw (4-) to (5-);
\draw (5-) to (6-);
\draw (6-) to (7-);

\draw [dashed] (1+) to [bend left=60] (6+);
\draw [dashed] (2+) to [bend left=60] (5+);
\draw [dashed] (3+) to [bend left=60] (4+);
\draw [dashed] (1-) to [bend right=60] (2-);
\draw [dashed] (4-) to [bend right=60] (7-);
\draw [dashed] (5-) to [bend right=60] (6-);

;\end{tikzpicture}\]
\caption{The orbit meander of $\mf{p}_7^\A(\{\alpha_7,\alpha_6,\alpha_5,
\alpha_4,\alpha_3,\alpha_2\}
\dd \{\alpha_7,\alpha_6,\alpha_4,\alpha_3,\alpha_2,\alpha_1\})$}
\label{fig:orbit meander nonfrob}
\end{figure}

\section{Principal Elements}\label{Principal Elements}

Given a Frobenius seaweed with Frobenius functional
$F$ and corresponding principal element $\widehat{F}$, the eigenvalues of
$\ad \widehat{F}$ are independent of which Frobenius functional is chosen \textbf{\cite{Ooms2}}.
We call these eigenvalues the \textit{spectrum} of the seaweed.
In this section, we describe an algorithm
for computing them.

Let $\mf{p}(\pi_1\dd\pi_2)$ be Frobenius and $F_{\pi_1,\pi_2}$ be an associated Frobenius functional with principal element $ \widehat{F}_{\pi_1,\pi_2}$.  (We will simply write $F$ and $\widehat{F}$ when the seaweed is understood.) 
Let $\sigma$ be a maximally connected component of either
$\pi_1$ or $\pi_2$, and for convenience let 
$\sigma=\{\alpha_k,\alpha_{k-1},\dots ,\alpha_1\}$ where $\alpha_1$ is the
exceptional root if $\sigma$ is of type B, C, or D.

Each eigenvalue of $\ad \widehat{F}$ can be expressed as a linear combination
of elements $\alpha_i(\widehat{F})$ where $\alpha_i$ is a simple root. We call such numbers \textit{simple eigenvalues}.
In many cases, the simple eigenvalues are determined.

\begin{lemma}[Joseph \textbf{\cite{Joseph}}, Section 5]
\label{eigenvalue table}
In Table \ref{tab:simple eigenvalue} below, the given value is 
$\alpha_i(\widehat{F})$ if $\sigma$ is a maximally connected component 
of $\pi_1$, and it is $-\alpha_i(\widehat{F})$ if 
$\sigma$ is a maximally connected component of $\pi_2$.
In either case it is assumed that $\alpha_i\in\sigma$.
\end{lemma}

\begin{table}[H]
\[\begin{tabular}{|l|l|l|}
\hline
Type & $\pm\alpha_i(\widehat{F})$ & $\pm\alpha_i(\widehat{F})$ \\
\hline
\hline
$A_k:k\geq 1$ & 
1, if $i_j\alpha_i=\alpha_i$
& \\
\hline
$B_{2k-1}:k\geq 2$ & $(-1)^{i-1}$, if $1\leq i\leq 2k-1$ & \\
\hline
$B_{2k}:k\geq 2$ & $(-1)^{i}$, if $2\leq i\leq 2k$ & 0, if $i=1$\\
\hline
$C_{k}:k\geq 2$ & $0$, if $2\leq i\leq k$ & 1, if $i=1$\\
\hline
$D_{2k}:k\geq 2$ & $(-1)^{i}$, if $3\leq i\leq 2k$ & 1, if $i=1,2$\\
\hline
$D_{2k+1}:k\geq 2$ & $(-1)^{i-1}$, if $3\leq i\leq 2k+1$ & \\
\hline
\end{tabular}\]
\caption{Values of $\pm\alpha_i(\widehat{F})$}
\label{tab:simple eigenvalue}
\end{table}

For the cases not covered by Table 1, 
the following lemma 
(which includes a corrected typo from \textbf{\cite{Joseph}})
can be applied.

\begin{lemma}[Joseph \textbf{\cite{Joseph}}, Section 5]
\label{eigenvalue orbit}
For each equation below, $\sigma$ is assumed to be a maximally 
connected component of $\pi_1$.
If $\sigma$ is a maximally connected component of $\pi_2$,
then replace $\alpha_i\mapsto -\alpha_i$ and $i_1\mapsto i_2$.
Then
\[\alpha_i(\widehat{F})+i_1\alpha_i(\widehat{F})=
\begin{dcases}
1, & \text{ if }\sigma \text{ is of type }A_k\text{ and }
(\alpha_i,i_1\alpha_i)<0; \\
0, & \text{ if }\sigma \text{ is of type }A_k\text{ and }
(\alpha_i,i_1\alpha_i)=0; \\
0, & \text{ if }\sigma \text{ is of type }D_{2k+1}\text{ and }
i=1.\\
\end{dcases}\]
\end{lemma}

\noindent
Here,  $(\alpha, \beta)$ denotes
the standard 
inner product on the Euclidean representation of the simple roots.  
Using Table \ref{tab:simple eigenvalue} and applying Lemma \ref{eigenvalue orbit}, we compute the simple eigenvalues for each running example.  See Figures \ref{fig:Asimple eigenvalues} - \ref{fig:Dsimple eigenvalues}, where each simple eigenvalue is noted above, below, or next to the appropriate vertex in the orbit meander for this seaweed.

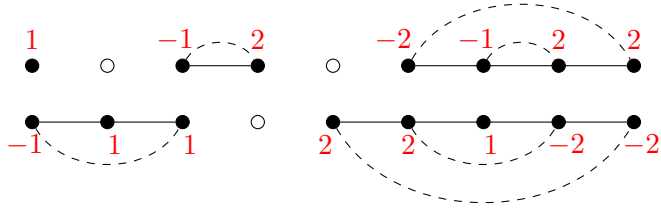
\begin{figure}[H]
\[\begin{tikzpicture}
[decoration={markings,mark=at position 0.6 with 
{\arrow{angle 90}{>}}}]

\draw (1,.75) node[draw,circle,fill=black,minimum size=5pt,inner sep=0pt] (1+) {};
\draw (2,.75) node[draw,circle,fill=white,minimum size=5pt,inner sep=0pt] (2+) {};
\draw (3,.75) node[draw,circle,fill=black,minimum size=5pt,inner sep=0pt] (3+) {};
\draw (4,.75) node[draw,circle,fill=black,minimum size=5pt,inner sep=0pt] (4+) {};
\draw (5,.75) node[draw,circle,fill=white,minimum size=5pt,inner sep=0pt] (5+) {};
\draw (6,.75) node[draw,circle,fill=black,minimum size=5pt,inner sep=0pt] (6+) {};
\draw (7,.75) node[draw,circle,fill=black,minimum size=5pt,inner sep=0pt] (7+) {};
\draw (8,.75) node[draw,circle,fill=black,minimum size=5pt,inner sep=0pt] (8+) {};
\draw (9,.75) node[draw,circle,fill=black,minimum size=5pt,inner sep=0pt] (9+) {};

\draw (1,0) node[draw,circle,fill=black,minimum size=5pt,inner sep=0pt] (1-) {};
\draw (2,0) node[draw,circle,fill=black,minimum size=5pt,inner sep=0pt] (2-) {};
\draw (3,0) node[draw,circle,fill=black,minimum size=5pt,inner sep=0pt] (3-) {};
\draw (4,0) node[draw,circle,fill=white,minimum size=5pt,inner sep=0pt] (4-) {};
\draw (5,0) node[draw,circle,fill=black,minimum size=5pt,inner sep=0pt] (5-) {};
\draw (6,0) node[draw,circle,fill=black,minimum size=5pt,inner sep=0pt] (6-) {};
\draw (7,0) node[draw,circle,fill=black,minimum size=5pt,inner sep=0pt] (7-) {};
\draw (8,0) node[draw,circle,fill=black,minimum size=5pt,inner sep=0pt] (8-) {};
\draw (9,0) node[draw,circle,fill=black,minimum size=5pt,inner sep=0pt] (9-) {};

\node at (9,1.1) [color=red] {{$2$}};
\node at (8,1.1) [color=red] {{$2$}};
\node at (6.9,1.1) [color=red] {{$-1$}};
\node at (5.8,1.1) [color=red] {{$-2$}};
\node at (9.1,-.3) [color=red] {{$-2$}};
\node at (8.1,-.3) [color=red] {{$-2$}};
\node at (7.1,-.3) [color=red] {{$1$}};
\node at (6,-.3) [color=red] {{$2$}};
\node at (4.9,-.3) [color=red] {{$2$}};
\node at (2.1,-.3) [color=red] {{$1$}};
\node at (3.1,-.3) [color=red] {{$1$}};
\node at (.9,-.3) [color=red] {{$-1$}};
\node at (1,1.1) [color=red] {{$1$}};
\node at (2.9,1.1) [color=red] {{$-1$}};
\node at (4,1.1) [color=red] {{$2$}};

\draw (1-) to (3-);
\draw (5-) to (9-);
\draw (3+) to (4+);
\draw (6+) to (9+);

\draw [dashed] (3+) to [bend left=60] (4+);
\draw [dashed] (6+) to [bend left=60] (9+);
\draw [dashed] (7+) to [bend left=60] (8+);
\draw [dashed] (1-) to [bend right=60] (3-);
\draw [dashed] (5-) to [bend right=60] (9-);
\draw [dashed] (6-) to [bend right=60] (8-);

;\end{tikzpicture}\]
\caption{The simple eigenvalues of $\mf{p}_9^\A(
\Upsilon_1 \dd \Upsilon_2)$}
\label{fig:Asimple eigenvalues}
\end{figure}

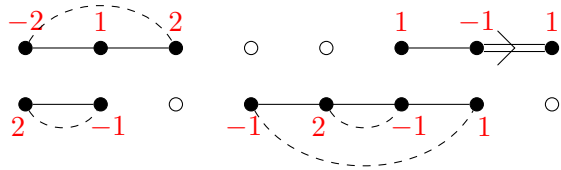
\begin{figure}[H]
\[\begin{tikzpicture}
[decoration={markings,mark=at position 0.6 with 
{\arrow{angle 90}{>}}}]

\draw (1,.75) node[draw,circle,fill=black,minimum size=5pt,inner sep=0pt] (1+) {};
\draw (2,.75) node[draw,circle,fill=black,minimum size=5pt,inner sep=0pt] (2+) {};
\draw (3,.75) node[draw,circle,fill=black,minimum size=5pt,inner sep=0pt] (3+) {};
\draw (4,.75) node[draw,circle,fill=white,minimum size=5pt,inner sep=0pt] (4+) {};
\draw (5,.75) node[draw,circle,fill=white,minimum size=5pt,inner sep=0pt] (5+) {};
\draw (6,.75) node[draw,circle,fill=black,minimum size=5pt,inner sep=0pt] (6+) {};
\draw (7,.75) node[draw,circle,fill=black,minimum size=5pt,inner sep=0pt] (7+) {};
\draw (8,.75) node[draw,circle,fill=black,minimum size=5pt,inner sep=0pt] (8+) {};

\draw (1,0) node[draw,circle,fill=black,minimum size=5pt,inner sep=0pt] (1-) {};
\draw (2,0) node[draw,circle,fill=black,minimum size=5pt,inner sep=0pt] (2-) {};
\draw (3,0) node[draw,circle,fill=white,minimum size=5pt,inner sep=0pt] (3-) {};
\draw (4,0) node[draw,circle,fill=black,minimum size=5pt,inner sep=0pt] (4-) {};
\draw (5,0) node[draw,circle,fill=black,minimum size=5pt,inner sep=0pt] (5-) {};
\draw (6,0) node[draw,circle,fill=black,minimum size=5pt,inner sep=0pt] (6-) {};
\draw (7,0) node[draw,circle,fill=black,minimum size=5pt,inner sep=0pt] (7-) {};
\draw (8,0) node[draw,circle,fill=white,minimum size=5pt,inner sep=0pt] (8-) {};

\node at (8,1.1) [color=red] {{$1$}};
\node at (7,1.1) [color=red] {{$-1$}};
\node at (6,1.1) [color=red] {{$1$}};
\node at (7.1,-.3) [color=red] {{$1$}};
\node at (6.1,-.3) [color=red] {{$-1$}};
\node at (4.9,-.3) [color=red] {{$2$}};
\node at (3.9,-.3) [color=red] {{$-1$}};
\node at (2.1,-.3) [color=red] {{$-1$}};
\node at (.9,-.3) [color=red] {{$2$}};
\node at (1,1.1) [color=red] {{$-2$}};
\node at (2,1.1) [color=red] {{$1$}};
\node at (3,1.1) [color=red] {{$2$}};

\draw (1-) to (2-);
\draw (4-) to (7-);
\draw (1+) to (3+);
\draw (6+) to (7+);
\draw [double distance=.8mm,postaction={decorate}] (7+) to (8+);

\draw [dashed] (1+) to [bend left=60] (3+);
\draw [dashed] (1-) to [bend right=60] (2-);
\draw [dashed] (4-) to [bend right=60] (7-);
\draw [dashed] (5-) to [bend right=60] (6-);

;\end{tikzpicture}\]
\caption{The simple eigenvalues of $\mf{p}_8^\B(
\Pi_1
\dd \Pi_2)$}
\label{fig:Bsimple eigenvalues}
\end{figure}

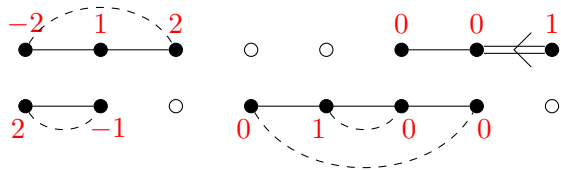
\begin{figure}[H]
\[\begin{tikzpicture}
[decoration={markings,mark=at position 0.6 with 
{\arrow{angle 90}{>}}}]

\draw (1,.75) node[draw,circle,fill=black,minimum size=5pt,inner sep=0pt] (1+) {};
\draw (2,.75) node[draw,circle,fill=black,minimum size=5pt,inner sep=0pt] (2+) {};
\draw (3,.75) node[draw,circle,fill=black,minimum size=5pt,inner sep=0pt] (3+) {};
\draw (4,.75) node[draw,circle,fill=white,minimum size=5pt,inner sep=0pt] (4+) {};
\draw (5,.75) node[draw,circle,fill=white,minimum size=5pt,inner sep=0pt] (5+) {};
\draw (6,.75) node[draw,circle,fill=black,minimum size=5pt,inner sep=0pt] (6+) {};
\draw (7,.75) node[draw,circle,fill=black,minimum size=5pt,inner sep=0pt] (7+) {};
\draw (8,.75) node[draw,circle,fill=black,minimum size=5pt,inner sep=0pt] (8+) {};

\draw (1,0) node[draw,circle,fill=black,minimum size=5pt,inner sep=0pt] (1-) {};
\draw (2,0) node[draw,circle,fill=black,minimum size=5pt,inner sep=0pt] (2-) {};
\draw (3,0) node[draw,circle,fill=white,minimum size=5pt,inner sep=0pt] (3-) {};
\draw (4,0) node[draw,circle,fill=black,minimum size=5pt,inner sep=0pt] (4-) {};
\draw (5,0) node[draw,circle,fill=black,minimum size=5pt,inner sep=0pt] (5-) {};
\draw (6,0) node[draw,circle,fill=black,minimum size=5pt,inner sep=0pt] (6-) {};
\draw (7,0) node[draw,circle,fill=black,minimum size=5pt,inner sep=0pt] (7-) {};
\draw (8,0) node[draw,circle,fill=white,minimum size=5pt,inner sep=0pt] (8-) {};

\node at (8,1.1) [color=red] {{$1$}};
\node at (7,1.1) [color=red] {{$0$}};
\node at (6,1.1) [color=red] {{$0$}};
\node at (7.1,-.3) [color=red] {{$0$}};
\node at (6.1,-.3) [color=red] {{$0$}};
\node at (4.9,-.3) [color=red] {{$1$}};
\node at (3.9,-.3) [color=red] {{$0$}};
\node at (2.1,-.3) [color=red] {{$-1$}};
\node at (.9,-.3) [color=red] {{$2$}};
\node at (1,1.1) [color=red] {{$-2$}};
\node at (2,1.1) [color=red] {{$1$}};
\node at (3,1.1) [color=red] {{$2$}};

\draw (1-) to (2-);
\draw (4-) to (7-);
\draw (1+) to (3+);
\draw (6+) to (7+);
\draw [double distance=.8mm,postaction={decorate}] (8+) to (7+);

\draw [dashed] (1+) to [bend left=60] (3+);
\draw [dashed] (1-) to [bend right=60] (2-);
\draw [dashed] (4-) to [bend right=60] (7-);
\draw [dashed] (5-) to [bend right=60] (6-);

;\end{tikzpicture}\]
\caption{The simple eigenvalues of $\mf{p}_8^\C(
\Pi_1
\dd \Pi_2)$}
\label{fig:simple eigenvalues}
\end{figure}

\begin{remark}
Modulo the arrow which emanates from the exceptional
root, observe that the components and Weyl action for $\mf{p}_8^\B(\Pi_1 \dd \Pi_2)$ and $\mf{p}_8^\C(\Pi_1 \dd \Pi_2)$ are identical. Even so, the simple eigenvalues of the type-B and type-C components differ.  Curiously, this affects the simple eigenvalues associated to the simple eigenvalues of same the type-A components of $\mf{p}_8^\B(\Pi_1 \dd \Pi_2)$ and $\mf{p}_8^\C(\Pi_1 \dd \Pi_2)$ leading to different spectra for these two seaweeds.
\end{remark}

\begin{figure}[H]
\[\begin{tikzpicture}
[decoration={markings,mark=at position 0.6 with 
{\arrow{angle 90}{>}}}]

\draw (1,-3) node[draw,circle,fill=black,minimum size=5pt,inner sep=0pt] (1+) {};
\draw (2,-3) node[draw,circle,fill=black,minimum size=5pt,inner sep=0pt] (2+) {};
\draw (3,-3) node[draw,circle,fill=black,minimum size=5pt,inner sep=0pt] (3+) {};
\draw (4,-3) node[draw,circle,fill=black,minimum size=5pt,inner sep=0pt] (4+) {};
\draw (5,-3) node[draw,circle,fill=white,minimum size=5pt,inner sep=0pt] (5+) {};
\draw (6,-3) node[draw,circle,fill=black,minimum size=5pt,inner sep=0pt] (6+) {};
\draw (7,-3) node[draw,circle,fill=black,minimum size=5pt,inner sep=0pt] (7+) {};
\draw (8,-3) node[draw,circle,fill=black,minimum size=5pt,inner sep=0pt] (8+) {};
\draw (9,-3) node[draw,circle,fill=black,minimum size=5pt,inner sep=0pt] (9+) {};
\draw (10,-3) node[draw,circle,fill=black,minimum size=5pt,inner sep=0pt] (10+) {};
\draw (11,-3) node[draw,circle,fill=black,minimum size=5pt,inner sep=0pt] (11+) {};
\draw (12,-3) node[draw,circle,fill=black,minimum size=5pt,inner sep=0pt] (12+) {};
\draw (13,-2.3) node[draw,circle,fill=black,minimum size=5pt,inner sep=0pt] (13+) {};
\draw (13,-3.7) node[draw,circle,fill=white,minimum size=5pt,inner sep=0pt] (14+) {};

\draw (1,0) node[draw,circle,fill=black,minimum size=5pt,inner sep=0pt] (1-) {};
\draw (2,0) node[draw,circle,fill=black,minimum size=5pt,inner sep=0pt] (2-) {};
\draw (3,0) node[draw,circle,fill=black,minimum size=5pt,inner sep=0pt] (3-) {};
\draw (4,0) node[draw,circle,fill=black,minimum size=5pt,inner sep=0pt] (4-) {};
\draw (5,0) node[draw,circle,fill=black,minimum size=5pt,inner sep=0pt] (5-) {};
\draw (6,0) node[draw,circle,fill=black,minimum size=5pt,inner sep=0pt] (6-) {};
\draw (7,0) node[draw,circle,fill=black,minimum size=5pt,inner sep=0pt] (7-) {};
\draw (8,0) node[draw,circle,fill=black,minimum size=5pt,inner sep=0pt] (8-) {};
\draw (9,0) node[draw,circle,fill=white,minimum size=5pt,inner sep=0pt] (9-) {};
\draw (10,0) node[draw,circle,fill=black,minimum size=5pt,inner sep=0pt] (10-) {};
\draw (11,0) node[draw,circle,fill=black,minimum size=5pt,inner sep=0pt] (11-) {};
\draw (12,0) node[draw,circle,fill=black,minimum size=5pt,inner sep=0pt] (12-) {};
\draw (13,.7) node[draw,circle,fill=black,minimum size=5pt,inner sep=0pt] (13-) {};
\draw (13,-.7) node[draw,circle,fill=black,minimum size=5pt,inner sep=0pt] (14-) {};

\draw (1-) to (8-);
\draw (10-) to (12-);
\draw (12-) to (13-);
\draw (12-) to (14-);
\draw (1+) to (4+);
\draw (6+) to (12+);
\draw (12+) to (13+);

\node at (1,-2.5) [color=red] {$1$};
\node at (2,-2.5) [color=red] {$-1$};
\node at (3,-2.5) [color=red] {$2$};
\node at (4,-2.5) [color=red] {$-1$};
\node at (6,-2.5) [color=red] {$-2$};
\node at (7,-2.5) [color=red] {$1$};
\node at (8,-2.5) [color=red] {$-1$};
\node at (9,-2.5) [color=red] {$2$};
\node at (10,-2.5) [color=red] {$-1$};
\node at (11,-2.5) [color=red] {$1$};
\node at (12,-2.5) [color=red] {$-1$};
\node at (13.5,-2.3) [color=red] {$2$};

\node at (1,-.5) [color=red] {$-1$};
\node at (2,-.5) [color=red] {$1$};
\node at (3,-.5) [color=red] {$-2$};
\node at (4,-.5) [color=red] {$1$};
\node at (5,-.5) [color=red] {$0$};
\node at (6,-.5) [color=red] {$2$};
\node at (7,-.5) [color=red] {$-1$};
\node at (8,-.5) [color=red] {$1$};
\node at (10,-.5) [color=red] {$1$};
\node at (11,-.5) [color=red] {$-1$};
\node at (12,-.5) [color=red] {$1$};
\node at (13.5,.7) [color=red] {$-2$};
\node at (13.5,-.7) [color=red] {$2$};

\draw [dashed] (1+) to [bend right=60] (4+);
\draw [dashed] (2+) to [bend right=60] (3+);
\draw [dashed] (6+) to [bend right=80] (13+);
\draw [dashed] (7+) to [bend right=60] (12+);
\draw [dashed] (8+) to [bend right=60] (11+);
\draw [dashed] (9+) to [bend right=60] (10+);

\draw [dashed] (1-) to [bend left=60] (8-);
\draw [dashed] (2-) to [bend left=60] (7-);
\draw [dashed] (3-) to [bend left=60] (6-);
\draw [dashed] (4-) to [bend left=60] (5-);
\draw [dashed] (13-) to [bend left=60] (14-);

;\end{tikzpicture}\]
\caption{The simple eigenvalues of $\mf{p}_{14}^\D(\Psi_1 \dd \Psi_2)$}
\label{fig:Dsimple eigenvalues}
\end{figure}
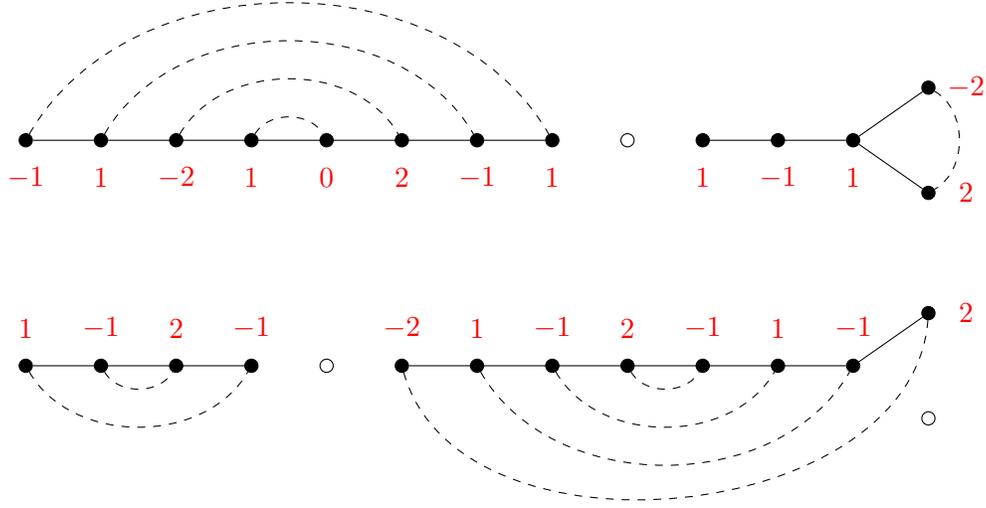

\begin{remark}
We will use the Greek letter $\sigma$ when referring to the simple roots of a maximally connected component of an orbit meander.  However, there are times when it will be more convenient to consider the set of eigenvalues associated to a set of consecutive vertices in an orbit meander.  
If $A$ is a set of consecutive vertices in an orbit meander, let $\mathcal{E}(A)$ be the eigenvalues associated to $A$.  That is, if $A = \{v_k, v_{k-1}, ... ,v_j\}$, let $\sigma = \{\alpha_k, \alpha_{k-1}, ..., \alpha_j \}$.  We make the following notational convention:
$$\mathcal{E}(A):=\mathcal{E}(\sigma).$$
\end{remark}

We have the following corollary of Lemma \ref{eigenvalue orbit} that will be used to prove symmetry and the unbroken property.  

\begin{theorem}\label{typeAsymmetricpositiveroot}
If $\sigma$ is a maximally connected component of type $A_k$, then $\sum_{i=1}^k \alpha_i (\widehat{F}) = 1$.  
\end{theorem}

\begin{proof}
If $k$ is odd, then $\alpha_i (\widehat{F}) = -\alpha_{k+1-i} (\widehat{F})$ for $i < \frac{k+1}{2}$, and $\alpha_{\frac{k+1}{2}} (\widehat{F}) = 1$.  
If $k$ is even, then $\alpha_i (\widehat{F}) = -\alpha_{k+1-i} (\widehat{F})$ for $i < \frac{k}{2}$, and $\alpha_{\frac{k}{2}} (\widehat{F}) + \alpha_{\frac{k}{2}+1} (\widehat{F}) = 1$.  
\end{proof}

If will be convenient to use the following notation.  
Let $\sigma$ be a maximally connected component of $\pi_1$, and let $\sigma=\{\alpha_k, \alpha_{k-1}, ..., \alpha_1\}$ be of type $A_k$.  The positive roots of $\sigma$ are of the form 
\[\alpha_j+\alpha_{j-1}+\dots +\alpha_i,\]
where $k\geq j\geq i\geq 1$.  
If $\alpha$ is a positive root with $j+i \neq k+1$, call $\overline{\alpha}$ its \textit{symmetric (positive) root}, where

\[\overline{\alpha}=
\begin{dcases}
\alpha_{i-1}+\alpha_{i-2} + \dots + \alpha_{k+1-j}, & \text{ if }k \geq j > i \geq \lc\frac{k}{2}\rc \geq 1; \\
\alpha_{i-1}+\alpha_{i-2} + \dots + \alpha_{k+1-j}, & \text{ if } \left|j - \lc\frac{k}{2}\rc\right| > \left|i - \lc\frac{k}{2}\rc\right|; \\
\alpha_{j+1}+ \alpha_{j+2} + \dots + \alpha_{k+1-i}, & \text{ if } \left|j - \lc\frac{k}{2}\rc\right| < \left|i - \lc\frac{k}{2}\rc\right|; \\
\alpha_{j+1}+ \alpha_{j+2} + \dots + \alpha_{k+1-i}, & \text{ if }k \geq \lc\frac{k}{2}\rc \geq j > i \geq 1. 
\end{dcases}\]

\noindent 
As a corollary of Theorem \ref{typeAsymmetricpositiveroot}, the symmetric roots $\alpha$ and $\overline{\alpha}$ satisfy the following relation:
\begin{eqnarray}\label{SymmPos}
\alpha(\widehat{F})+\overline{\alpha}(\widehat{F})=1
\end{eqnarray}
\noindent 
Symmetric roots satisfy the relation (\ref{SymmPos}), and since $\alpha(\widehat{F})=1$ when $j+i=k+1$, we say the positive root $\alpha$ has no symmetric root.  
We call $\overline{\alpha}(\widehat{F})$ the \textit{symmetric eigenvalue} of $\alpha(\widehat{F})$ or that $\overline{\alpha}(\widehat{F})$ is an eigenvalue \textit{symmetric to} $\alpha(\widehat{F})$.  

\begin{example}
Referring to, for example, Figure \ref{fig:simple eigenvalues}, consider the type-A component on the bottom $\sigma = \{\alpha_5, \alpha_4, \alpha_3, \alpha_2\}$.  The root symmetric to $\alpha = \alpha_5+\alpha_4+\alpha_3$ is $\overline{\alpha} = \alpha_2$, and the root symmetric to $\alpha = \alpha_5+\alpha_4$ is $\overline{\alpha} = \alpha_3+\alpha_2$.  The root $\alpha = \alpha_5 + \alpha_4 + \alpha_3 + \alpha_2$ has no symmetric root.
\end{example}

\section{Proof of Theorem \ref{thm:main} }

\subsection{Symmetry}

We will partition the multiset of eigenvalues according to the
maximally connected components $\sigma$ of $\pi_1$ and $\pi_2$. 
We will prove that each member of this partition is 
symmetric and unbroken.

If $\sigma$ is a Type-A maximally connected
component of $\pi_1$, then we say the multiset of eigenvalues 
from $\sigma$ is
\begin{eqnarray}\label{S1}
\mathcal{E}(\sigma)=\{\beta(\widehat{F})\mid\beta\in\mathbb{N}\sigma\cap\Delta_+\}
\cup\{0^{\ceil{|\sigma|/2}}\},
\end{eqnarray}
if $\sigma$ is of Type B, C or $\D_k$ with $k$ even, then
\begin{eqnarray}\label{S2}
\mathcal{E}(\sigma)=\{\beta(\widehat{F})\mid\beta\in\mathbb{N}\sigma\cap\Delta_+\}
\cup\{0^{|\sigma|}\},
\end{eqnarray}
and if $\sigma$ is of Type $\D_k$ with $k$ odd, then
\begin{eqnarray}\label{SD3}
\mathcal{E}(\sigma)=\{\beta(\widehat{F})\mid\beta\in\mathbb{N}\sigma\cap\Delta_+\}
\cup\{0^{|\sigma|-1}\},
\end{eqnarray}

\noindent 
where the exponent denotes the multiplicity, and 
$\mathbb{N}\sigma = \left\{\sum c_i \alpha_i ~|~ c_i > 0 ~ \text{ and } \alpha_i \in \sigma\right\}.$
\noindent
We proceed similarly if $\sigma$ is a maximally connected component of $\pi_2$ except that $\Delta_+$ is replaced by $\Delta_-$ in equations (\ref{S1}) - 
(\ref{SD3}).  
We demonstrate these computations for each running example.

\begin{example}\label{ASpectrum}
In the type-A running example $\mf{p}_{9}^\A(\Upsilon_1 \dd \Upsilon_2)$, on the top there are three type-A components:
$\sigma_1=\{\alpha_9\}$, 
$\sigma_2=\{\alpha_7, \alpha_6\}$, and
$\sigma_3=\{\alpha_4, \alpha_3, \alpha_2, \alpha_1\}$.  
There are two type-A components on the bottom:  $\sigma_4=\{\alpha_9, \alpha_8, \alpha_7\}$ and 
$\sigma_5=\{\alpha_5, \alpha_4, \alpha_3, \alpha_2, \alpha_1\}$.  
To compute $\mathcal{E}(\sigma_i)$, note,
for example, that the positive roots for the computation of $\mathcal{E}(\sigma_3)$ are elements of the set 
$$B_{\sigma_3}=\{\alpha_4, \alpha_3, \alpha_2, \alpha_1, \alpha_4+\alpha_3, \alpha_3+\alpha_2, \alpha_2+ \alpha_1, \alpha_4+\alpha_3+\alpha_2, \alpha_3+\alpha_2+\alpha_1, \alpha_4+\alpha_3+\alpha_2+\alpha_1  \}.$$ 
Applying each of $\beta\in B_{\sigma_3}$ to $\widehat{F}$ yields the multiset 
$$
\{-2, -1, 2, 2, -3, 1, 4, -1, 3, 1\} = \{-3,-2,-1,-1,1,1,2,2,3,4\}.$$  
Since $|\sigma_3|=4$, we have by equation (\ref{S1}) that
$$
\mathcal{E}(\sigma_3) = \{-3,-2,-1,-1,1,1,2,2,3,4\}\cup\{0,0\} = \{-3,-2,-1^2,0^2,1^2,2^2,3,4  \}.
$$
\noindent 
Similar computations yield
$$
\mathcal{E}(\sigma_1) = \{0,1\}, ~~
\mathcal{E}(\sigma_2) = \{-1,0,1,2\}, ~~
\mathcal{E}(\sigma_4) = \{-1,0^3,1^3,2\}, 
$$
$$
\mathcal{E}(\sigma_5) = \{-4,-3,-2^2,-1^2,0^3,1^3,2^2,3^2,4,5\}.
$$
\end{example}

\begin{example}\label{ex:Bspectrum}
In the type-B running example $\mf{p}_{8}^\B(\Pi_1 \dd \Pi_2)$, on the top there is a single type-A component $\sigma_1=\{\alpha_8, \alpha_7, \alpha_6\}$ and a single type-B component $\sigma_2=\{\alpha_3, \alpha_2, \alpha_1\}$. There are two type-A components on the bottom:  $\sigma_3=\{\alpha_8, \alpha_7\}$ and 
$\sigma_4=\{\alpha_5, \alpha_4, \alpha_3, \alpha_2\}$.  
For the type-A components, computations similar to those in Example \ref{ASpectrum} yield 
$$\mathcal{E}(\sigma_1) = \{-2,-1,0^2,1^2,2,3\}, ~~~~ \mathcal{E}(\sigma_3) = \{-1,0,1,2 \}, ~~\text{and}~~
\mathcal{E}(\sigma_4) = \{-1^2,0^4,1^4,2^2\}. $$
For the type-B component, the positive roots for the computation of 
$\mathcal{E}(\sigma_2)$ are elements of the set 
$$B_{\sigma_2}=\{\alpha_3, \alpha_2, \alpha_1, \alpha_3+\alpha_2, \alpha_2+\alpha_1, \alpha_3+\alpha_2+\alpha_1, \alpha_2+2\alpha_1, \alpha_3 + \alpha_2+2\alpha_1, \alpha_3+2\alpha_2+2\alpha_1\}.$$ 
\noindent
Applying each of $\beta\in B_{\sigma_2}$ to $\widehat{F}$ yields the multiset 
$$\{1,-1,1,0,0,1,1,2,1\} = \{-1,0,0,1,1,1,1,1,2\}.$$
Since $|\sigma_2| = 3$, we have by equation (\ref{S2}) that
$$
\mathcal{E}(\sigma_2) =
\{-1,0,0,1,1,1,1,1,2\}\cup \{0,0,0\} = 
\{-1,0^5,1^5,2\}.
$$

\end{example}

\begin{example}\label{ex:spectrum}
In the type-C running example $\mf{p}_{8}^\C(\Pi_1 \dd \Pi_2)$, on the top there is a single type-A component $\sigma_1=\{\alpha_8, \alpha_7, \alpha_6\}$ and a single type-C component $\sigma_2=\{\alpha_3, \alpha_2, \alpha_1\}$. There are two type-A components on the bottom:  $\sigma_3=\{\alpha_8, \alpha_7\}$ and 
$\sigma_4=\{\alpha_5, \alpha_4, \alpha_3, \alpha_2\}$.  
For the type-A components, computations similar to those in Example \ref{ASpectrum} yield 
$$\mathcal{E}(\sigma_1) = \{-2,-1,0^2,1^2,2,3\}, ~~~~ \mathcal{E}(\sigma_3) = \{-1,0,1,2 \}, ~~\text{and}~~
\mathcal{E}(\sigma_4) = \{0^6,1^6\}. $$
For the type-C component, the positive roots for the computation of 
$\mathcal{E}(\sigma_2)$ are elements of the set 
$$B_{\sigma_2}=\{\alpha_3, \alpha_2, \alpha_1, \alpha_3+\alpha_2, \alpha_2+\alpha_1, \alpha_3+\alpha_2+\alpha_1, 2\alpha_2+\alpha_1, \alpha_3 + 2\alpha_2+\alpha_1, 2\alpha_3+2\alpha_2+\alpha_1\}.$$ 
\noindent
Applying each of $\beta\in B_{\sigma_2}$ to $\widehat{F}$ yields the multiset 
$$\{0,0,1,0,1,1,1,1,1\} = \{0,0,0,1,1,1,1,1,1\}$$
Since $|\sigma_2| = 3$, we have by equation (\ref{S2}) that
$$
\mathcal{E}(\sigma_2) =
\{0,0,1,0,1,1,1,1,1\}\cup \{0,0,0\} = 
\{0^6,1^6\}.
$$

\end{example}

\begin{example}\label{ex:Dspectrum}
Finally, in the type-D running example, $\mf{p}_{14}^\D(\Psi_1 \dd \Psi_2)$, on the top there again is single type-A component $\sigma_1=\{\alpha_{14}, \alpha_{13}, \alpha_{12}, \alpha_{11}, \alpha_{10}, \alpha_9, \alpha_8, \alpha_7\}$ and there is a single type-D component $\sigma_2=\{\alpha_5, \alpha_4, \alpha_3, \alpha_2, \alpha_1\}$. 
There are also two type-A components on the bottom:  $\sigma_3=\{\alpha_{14}, \alpha_{13}, \alpha_{12}, \alpha_{11}\} \text{ and }
\sigma_4=\{\alpha_9, \alpha_8, \alpha_7, \alpha_6, \alpha_5, \alpha_4, \alpha_3, \alpha_2\}.$ 
For the type-A components, computations similar to those in Example \ref{ASpectrum}, yield the multisets:  
$$\mathcal{E}(\sigma_1) = \{-2^2, -1^7, 0^{11}, 1^{11}, 2^7, 3^2\}, ~~~~ \mathcal{E}(\sigma_3) = \{-1^2, 0^4, 1^4, 2^2\}, ~~\text{and}$$
$$\mathcal{E}(\sigma_4) = \{-2^2, -1^7, 0^{11}, 1^{11}, 2^7, 3^2\}. $$

For the type-D component, the positive roots for the computation of $\mathcal{E}(\sigma_2)$ are elements of the set 
$$\hspace*{-.5cm}B_{\sigma_2}=\left\{
\begin{array}{c}
\alpha_5, \alpha_4, \alpha_3, \alpha_2, \alpha_1, 
\alpha_5+\alpha_4, \alpha_4+\alpha_3, \alpha_3+\alpha_2, \alpha_3+\alpha_1, \\
\alpha_5+\alpha_4+\alpha_3, \alpha_4+\alpha_3+\alpha_2, \alpha_4+\alpha_3+\alpha_1, \alpha_3+\alpha_2+\alpha_1, \\
\alpha_5+\alpha_4+\alpha_3+\alpha_2, \alpha_5+\alpha_4+\alpha_3+\alpha_1, \alpha_4+\alpha_3+\alpha_2+\alpha_1, \alpha_4+2\alpha_3+\alpha_2+\alpha_1, \\
\alpha_5+\alpha_4+\alpha_3+\alpha_2+\alpha_1, \alpha_5+\alpha_4+2\alpha_3+\alpha_2+\alpha_1, \alpha_5+2\alpha_4+2\alpha_3+\alpha_2+\alpha_1
\end{array}
\right\}.$$
Applying each of $\beta\in B_{\sigma_2}$ to $\widehat{F}$ yields the multiset 
$$\{-2^2, -1^3, 0^3, 1^7, 2^3, 3^2\}.$$  
\noindent 
Since, $|\sigma_2|=5$, we have by equation (\ref{SD3}) that
$$\mathcal{E}(\sigma_2) = \{-2^2, -1^3, 0^3, 1^7, 2^3, 3^2\}\cup\{0^4\} = \{-2^2, -1^3, 0^7, 1^7, 2^3, 3^2\}.$$
\end{example}

\begin{lemma}\label{Union}
If $\mf{p}(\pi_1\dd \pi_2)$ is a Frobenius seaweed, then the
multisets of eigenvalues contributed by each maximally
connected component form a multiset partition of the spectrum
of the seaweed.
\end{lemma}

We apply Lemma \ref{Union} 
to obtain the spectrum for each running example.  We take the union of the sets $\mathcal{E}(\sigma_i)$ found in Examples \ref{ASpectrum} - \ref{ex:Dspectrum}.  
This data is consolidated in Tables \ref{tab:Aeigenvalue} - \ref{tab:Deigenvalue}.  

\begin{table}[H]
\[\begin{tabular}{|l||l|l|l|l|l|l|l|l|l|l|}
\hline
Eigenvalue& -4 & -3 & -2 & -1 & 0 & 1 & 2 & 3 & 4 & 5 \\
\hline
Multiplicity & 1 & 2 & 3 & 6 & 10 & 10 & 6 & 3 & 2 & 1 \\
\hline
\end{tabular}\]
\caption{Spectrum of $\mf{p}_{9}^\A(\Upsilon_1 \dd \Upsilon_2)$ with multiplicities}
\label{tab:Aeigenvalue}
\end{table}

\begin{table}[H]
\[\begin{tabular}{|l||l|l|l|l|l|l|}
\hline
Eigenvalue& -2 & -1 & 0 & 1 & 2 & 3 \\
\hline
Multiplicity & 1 & 5 & 12 & 12 & 5 & 1 \\
\hline
\end{tabular}\]
\caption{Spectrum of $\mf{p}_{8}^\B(\Pi_1 \dd \Pi_2)$ with multiplicities}
\label{tab:Beigenvalue}
\end{table}

\begin{table}[H]
\[\begin{tabular}{|l||l|l|l|l|l|l|}
\hline
Eigenvalue& -2 & -1 & 0 & 1 & 2 & 3 \\
\hline
Multiplicity & 1 & 2 & 15 & 15 & 2 & 1 \\
\hline
\end{tabular}\]
\caption{Spectrum of $\mf{p}_{8}^\C(\Pi_1 \dd \Pi_2)$ with multiplicities}
\label{tab:eigenvalue}
\end{table}

\begin{table}[H]
\[\begin{tabular}{|l||l|l|l|l|l|l|}
\hline
Eigenvalue& -2 & -1 & 0 & 1 & 2 & 3 \\
\hline
Multiplicity & 6 & 19 & 33 & 33 & 19 & 6 \\
\hline
\end{tabular}\]
\caption{Spectrum of $\mf{p}_{14}^\D(\Psi_1 \dd \Psi_2)$ with multiplicities}
\label{tab:Deigenvalue}
\end{table}

The following example illustrates a Frobenius type-D seaweed containing an even-sized type-D component. 

\begin{example}\label{DEven}
Define the seaweed $\mf{p}_{11}^\D(\Phi_1 \dd \Phi_2)$ by the following sets:
\begin{center}$\Phi_1 = \{\alpha_{11}, \alpha_{10}, \alpha_9, \alpha_8, \alpha_7, \alpha_6, \alpha_4, \alpha_3, \alpha_2, \alpha_1\},$\\ and\\
$\Phi_2 = \{\alpha_{10}, \alpha_9, \alpha_7, \alpha_6, \alpha_5, \alpha_4, \alpha_3, \alpha_2\}.$  
\end{center}
See Figure \ref{fig:DEvensimple eigenvalues} for the simple eigenvalues of $\mf{p}_{11}^\D(\Phi_1 \dd \Phi_2)$.  
\end{example}

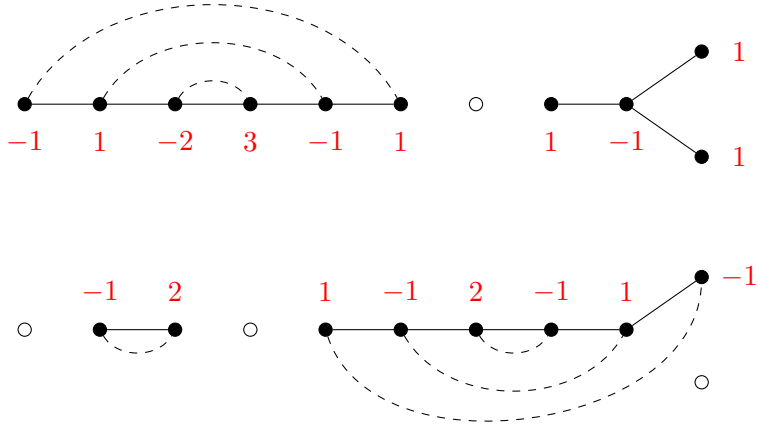
\begin{figure}[H]
\[\begin{tikzpicture}
[decoration={markings,mark=at position 0.6 with 
{\arrow{angle 90}{>}}}]

\draw (4,-3) node[draw,circle,fill=white,minimum size=5pt,inner sep=0pt] (4+) {};
\draw (5,-3) node[draw,circle,fill=black,minimum size=5pt,inner sep=0pt] (5+) {};
\draw (6,-3) node[draw,circle,fill=black,minimum size=5pt,inner sep=0pt] (6+) {};
\draw (7,-3) node[draw,circle,fill=white,minimum size=5pt,inner sep=0pt] (7+) {};
\draw (8,-3) node[draw,circle,fill=black,minimum size=5pt,inner sep=0pt] (8+) {};
\draw (9,-3) node[draw,circle,fill=black,minimum size=5pt,inner sep=0pt] (9+) {};
\draw (10,-3) node[draw,circle,fill=black,minimum size=5pt,inner sep=0pt] (10+) {};
\draw (11,-3) node[draw,circle,fill=black,minimum size=5pt,inner sep=0pt] (11+) {};
\draw (12,-3) node[draw,circle,fill=black,minimum size=5pt,inner sep=0pt] (12+) {};
\draw (13,-2.3) node[draw,circle,fill=black,minimum size=5pt,inner sep=0pt] (13+) {};
\draw (13,-3.7) node[draw,circle,fill=white,minimum size=5pt,inner sep=0pt] (14+) {};

\draw (4,0) node[draw,circle,fill=black,minimum size=5pt,inner sep=0pt] (4-) {};
\draw (5,0) node[draw,circle,fill=black,minimum size=5pt,inner sep=0pt] (5-) {};
\draw (6,0) node[draw,circle,fill=black,minimum size=5pt,inner sep=0pt] (6-) {};
\draw (7,0) node[draw,circle,fill=black,minimum size=5pt,inner sep=0pt] (7-) {};
\draw (8,0) node[draw,circle,fill=black,minimum size=5pt,inner sep=0pt] (8-) {};
\draw (9,0) node[draw,circle,fill=black,minimum size=5pt,inner sep=0pt] (9-) {};
\draw (10,0) node[draw,circle,fill=white,minimum size=5pt,inner sep=0pt] (10-) {};
\draw (11,0) node[draw,circle,fill=black,minimum size=5pt,inner sep=0pt] (11-) {};
\draw (12,0) node[draw,circle,fill=black,minimum size=5pt,inner sep=0pt] (12-) {};
\draw (13,.7) node[draw,circle,fill=black,minimum size=5pt,inner sep=0pt] (13-) {};
\draw (13,-.7) node[draw,circle,fill=black,minimum size=5pt,inner sep=0pt] (14-) {};

\draw (4-) to (9-);
\draw (11-) to (12-);
\draw (12-) to (13-);
\draw (12-) to (14-);
\draw (5+) to (6+);
\draw (8+) to (12+);
\draw (12+) to (13+);

\node at (5,-2.5) [color=red] {$-1$};
\node at (6,-2.5) [color=red] {$2$};
\node at (8,-2.5) [color=red] {$1$};
\node at (9,-2.5) [color=red] {$-1$};
\node at (10,-2.5) [color=red] {$2$};
\node at (11,-2.5) [color=red] {$-1$};
\node at (12,-2.5) [color=red] {$1$};
\node at (13.5,-2.3) [color=red] {$-1$};

\node at (4,-.5) [color=red] {$-1$};
\node at (5,-.5) [color=red] {$1$};
\node at (6,-.5) [color=red] {$-2$};
\node at (7,-.5) [color=red] {$3$};
\node at (8,-.5) [color=red] {$-1$};
\node at (9,-.5) [color=red] {$1$};
\node at (11,-.5) [color=red] {$1$};
\node at (12,-.5) [color=red] {$-1$};
\node at (13.5,.7) [color=red] {$1$};
\node at (13.5,-.7) [color=red] {$1$};

\draw [dashed] (5+) to [bend right=60] (6+);
\draw [dashed] (8+) to [bend right=80] (13+);
\draw [dashed] (9+) to [bend right=60] (12+);
\draw [dashed] (10+) to [bend right=60] (11+);

\draw [dashed] (4-) to [bend left=60] (9-);
\draw [dashed] (5-) to [bend left=60] (8-);
\draw [dashed] (6-) to [bend left=60] (7-);

;\end{tikzpicture}\]
\caption{The simple eigenvalues of \small{$\mf{p}_{11}^\D(\Phi_1 \dd \Phi_2)$}}
\label{fig:DEvensimple eigenvalues}
\end{figure}

\noindent
The eigenvalues for $\mf{p}_{11}^\D(\Phi_1 \dd \Phi_2)$ can be found using computations similar to those in Example \ref{ex:Dspectrum} and are given in Table \ref{tab:DEveneigenvalue}.  

\begin{table}[H]
\[\begin{tabular}{|l||l|l|l|l|l|l|}
\hline
Eigenvalue& -2 & -1 & 0 & 1 & 2 & 3 \\
\hline
Multiplicity & 2 & 9 & 23 & 23 & 9 & 2 \\
\hline
\end{tabular}\]
\caption{Spectrum of $\mf{p}_{11}^\D(\Phi_1 \dd \Phi_2)$ with multiplicities}
\label{tab:DEveneigenvalue}
\end{table}

Notice that in Tables \ref{tab:Aeigenvalue} - \ref{tab:DEveneigenvalue}, $\mathcal{E}(\sigma_i)$ is symmetric about one half for each $i$.  
Furthermore, the multiplicities form a symmetric distribution. 
As the following theorem shows,  these observations are not coincidences. 

\begin{theorem}\label{thm:symmetry}
Let $\mf{p}(\pi_1\dd \pi_2)$ be a Frobenius seaweed.
For each maximally connected component $\sigma$ of $\pi_1$ or $\pi_2$,
let $r_i$ be the multiplicity of the eigenvalue $i$ in $\mathcal{E}(\sigma)$. 
The sequence $(i)$ is symmetric about one-half.  Moreover, $r_{-i} = r_{i+1}$ for each eigenvalue $i$.
\end{theorem}

\begin{proof}
We prove this in the case
that $\sigma$ is a maximally connected component of $\pi_1$.
The case that $\sigma$ is a maximally connected component of $\pi_2$
is similar and is omitted.

\setcounter{case}{0}
\begin{case}\label{symmetryA}
Type A
\end{case}
Suppose $\sigma$ is of type $A_k$.  
Then $\mathcal{E}(\sigma)$ is comprised of elements of the form 
\[\alpha_j(\widehat{F})+\alpha_{j-1}(\widehat{F})+\dots +\alpha_i(\widehat{F}),\]
where $k\geq j\geq i\geq 1$ along with $\lc \frac{k}{2}\rc$ zeros per Equation (\ref{S1}).  

Each element $\alpha(\widehat{F})$ in $\mathcal{E}(\sigma)$ with $j+i \neq k+1$ has a symmetric eigenvalue $\overline{\alpha}$ with $\alpha(\widehat{F})+\overline{\alpha}(\widehat{F}) = 1$.  
Moreover, there are $\lc \frac{k}{2}\rc$ positive roots 
$\alpha_j + \alpha_{j-1} ... + \alpha_i$ with $k \geq j \geq i \geq 1$ and $j+i = k+1$.  These satisfy
\[\alpha_j(\widehat{F})+\alpha_{j-1}(\widehat{F})+\dots +\alpha_i(\widehat{F})=1\]
 and are in bijective correspondence with the zeros from Equation (\ref{S1}).  Therefore, the multiset of eigenvalues from $\sigma$ is symmetric about one-half.  

\begin{case}\label{symmetryB}
Type B
\end{case}
Suppose $\sigma$ is of type B, with odd cardinality.
For convenience, reorder the indices of the simple roots
so that $\sigma=\{\alpha_{2k-1},\alpha_{2k-2},\dots ,\alpha_{1}\}$ as in 
Lemma \ref{eigenvalue table}.
Then $\mathcal{E}(\sigma)$ is comprised of elements of the form
\[\alpha_j(\widehat{F})+\alpha_{j-1}(\widehat{F})+\dots +\alpha_i(\widehat{F}),\]
where $2k-1\geq j\geq i\geq 1$, or
\[\alpha_{j}(\widehat{F})+\dots +\alpha_{i+1}(\widehat{F})+2\alpha_{i}(\widehat{F})+\dots +2\alpha_{1}(\widehat{F}),\]
where $2k-1\geq j>i\geq 2$, along with $2k-1$ zeros per Equation (\ref{S2}).

We need only consider eigenvalues $-1, 0, 1,$ and $2$, and these eigenvalues have multiplicities given by 

\[r_{-1}=\sum_{i=1}^{k-1}i=\binom{k}{2},\]
\[r_{0}=\left(\sum_{i=1}^{k-1}2i\right)+(2k-1)+\left(\sum_{i=1}^{k-2}i\right)
=\dfrac{k(3k-1)}{2},\]
\[r_{1}=\left(\sum_{i=1}^{k}i\right)+\left(\sum_{i=1}^{k-1}2i\right)
=\dfrac{k(3k-1)}{2},\]
\[r_{2}=\sum_{i=1}^{k-1}i=\binom{k}{2}.\]

If $\sigma$ is of type B with even cardinality,
we reorder the indices of the simple roots
so that $\sigma=\{\alpha_{2k},\dots ,\alpha_{1}\}$ as in 
Lemma \ref{eigenvalue table}. 
Then $\mathcal{E}(\sigma)$ is comprised of elements of the form
\[\alpha_j(\widehat{F})+\alpha_{j-1}(\widehat{F})+\dots +\alpha_i(\widehat{F}),\]
where $2k\geq j\geq i\geq 1$, or

\[\alpha_{j}(\widehat{F})+\dots +\alpha_{i+1}(\widehat{F})+2\alpha_{i}(\widehat{F})+\dots +2\alpha_{1}(\widehat{F}),\]
where $2k\geq j>i\geq 2$, along with $2k$ zeros per Equation (\ref{S2}).
Then

\[r_{-1}=\sum_{i=1}^{k-1}i=\binom{k}{2},\]
\[r_{0}=\left(\sum_{i=1}^{k}2(i-1)+1\right)+(2k)+\left(\sum_{i=1}^{k-1}i\right)
=3\binom{k+1}{2},\]
\[r_{1}=\left(\sum_{i=1}^{k}i+1\right)+(k)+\left(\sum_{i=1}^{k-1}2i\right)
=3\binom{k+1}{2},\]
\[r_{2}=\sum_{i=1}^{k-1}i=\binom{k}{2}.\]

\begin{case}\label{symmetryC}
Type C
\end{case}
Suppose $\sigma$ is of type C.  We again 
reorder the indices of the simple roots
so $\sigma=\{\alpha_{k},\alpha_{k-1},\dots ,\alpha_{1}\}$ as in 
Lemma \ref{eigenvalue table}.
Then $\mathcal{E}(\sigma)$ is comprised of elements of the form
\[\alpha_j(\widehat{F})+\alpha_{j-1}(\widehat{F})+\dots +\alpha_i(\widehat{F}),\]
where $k\geq j\geq i\geq 1$, or
\[\alpha_j(\widehat{F})+\dots +\alpha_{i+1}(\widehat{F})+2\alpha_{i}(\widehat{F})+\dots +2\alpha_{2}(\widehat{F})
+\alpha_1(\widehat{F}),\]
where $k\geq j\geq i\geq 2$, along with $k$ zeros per Equation (\ref{S2}).

Counting the multiplicity of the eigenvalue 0, we see that 
there are $\binom{k-1}{2}$ eigenvalues of the form 
$\alpha_j(\widehat{F})+\dots +\alpha_i(\widehat{F})$ where $k\geq j>i\geq 2$, 
$k-1$ eigenvalues of the form $\alpha_i(\widehat{F})$ where $k\geq i\geq 2$,
and $k$ zeros per Equation (\ref{S2}) included in $\mathcal{E}(\sigma)$.
Thus $r_0=\binom{k-1}{2}+(k-1)+k=\binom{k+1}{2}$.

Counting the multiplicity of the eigenvalue 1, we see that 
there are $k$ eigenvalues of the form 
$\alpha_j(\widehat{F})+\dots +\alpha_1(\widehat{F})$ where $k\geq j\geq 1$, 
$k-1$ eigenvalues of the form 
$2\alpha_i(\widehat{F})+\dots 2\alpha_{2}(\widehat{F})+\alpha_1(\widehat{F})$ where $k\geq i\geq 2$, and $\binom{k-1}{2}$ eigenvalues of the form 
$\alpha_j(\widehat{F})+\dots +\alpha_{i+1}(\widehat{F})+2\alpha_{i}(\widehat{F})+\dots +2\alpha_{2}(\widehat{F})
+\alpha_1(\widehat{F})$ where $k\geq j>i\geq 2$.
It follows that $r_1=\binom{k+1}{2}$ as well.

\begin{case}\label{symmetryD}
Type D
\end{case}

Suppose $\sigma$ is of type D, with odd cardinality.
For convenience, reorder the indices of the simple roots
so that $\sigma=\{\alpha_{2k+1},\alpha_{2k},\dots ,\alpha_{1}\}$ as in 
Lemma \ref{eigenvalue table}.
Then $\mathcal{E}(\sigma)$ is comprised of elements of the form
\[\alpha_j(\widehat{F})+\alpha_{j-1}(\widehat{F})+\dots +\alpha_i(\widehat{F}),\]
where $2k+1\geq j\geq i\geq 1$  except the case with $j = 2$ and $i = 1$, or
\[\alpha_j(\widehat{F})+\alpha_{j-1}(\widehat{F})+\dots + \alpha_{i+2}(\widehat{F})+\alpha_i(\widehat{F}),\]
where $2k+1 \geq j\geq i = 1$, or
\[\alpha_{j}(\widehat{F})+\dots +\alpha_{i+1}(\widehat{F})+2\alpha_{i}(\widehat{F})+\dots +2\alpha_{3}(\widehat{F})+\alpha_{2}(\widehat{F})+\alpha_{1}(\widehat{F}),\]
where $2k+1\geq j>i\geq 2$, along with $2k$ zeros per Equation (\ref{SD3}).

\noindent 
If $\alpha_2(\widehat{F})=-2$, then 

\[r_{-2}=\sum_{i=1}^{k}1=k,\]
\[r_{-1}=\sum_{i=1}^k 1+\sum_{i=1}^{k-1}i=\dfrac{k(k+1)}{2},\]
\[r_{0}=\left(\sum_{i=1}^{k-1}2i\right)+2k+\left(\sum_{i=1}^{k-1}i\right)
=\dfrac{k(3k+1)}{2},\]
\[r_{1}=\left(\sum_{i=0}^{k-1}2i+1\right)+\left(\sum_{i=1}^{k}i\right)
=\dfrac{k(3k+1)}{2},\]
\[r_{2}=\sum_{i=1}^k 1+\sum_{i=1}^{k-1}i=\dfrac{k(k+1)}{2},\]
\[r_{3}=\sum_{i=1}^{k}1=k.\]
If
$\alpha_2(\widehat{F})=-1$, then 
\[r_{-1}=\sum_{i=1}^k i=\dfrac{k(k+1)}{2},\]
\[r_{0}=\left(\sum_{i=0}^{k-1}2i+1\right)+2k+\sum_{i=1}^{k-1}i
=\dfrac{3k(k+1)}{2},\]
\[r_{1}=\sum_{i=1}^{k}i+\left(\sum_{i=1}^{k}2i\right)
=\dfrac{3k(k+1)}{2},\]
\[r_{2}=\sum_{i=1}^k i=\dfrac{k(k+1)}{2}.\]

Suppose $\sigma$ is of type D, with even cardinality.
For convenience, reorder the indices of the simple roots
so that $\sigma=\{\alpha_{2k},\dots ,\alpha_{1}\}$ as in 
Lemma \ref{eigenvalue table}. 
Then $\mathcal{E}(\sigma)$ is comprised of elements of the form
\[\alpha_j(\widehat{F})+\alpha_{j-1}(\widehat{F})+\dots +\alpha_i(\widehat{F}),\]
where $2k\geq j\geq i\geq 1$  except the case with $j = 2$ and $i = 1$, or
\[\alpha_j(\widehat{F})+\alpha_{j-1}(\widehat{F})+\dots + \alpha_{i+2}(\widehat{F})+\alpha_i(\widehat{F}),\]
where $2k \geq j\geq i = 1$, or
\[\alpha_{j}(\widehat{F})+\dots +\alpha_{i+1}(\widehat{F})+2\alpha_{i}(\widehat{F})+\dots +2\alpha_{3}(\widehat{F})+\alpha_{2}(\widehat{F})+\alpha_{1}(\widehat{F}),\]
where $2k\geq j>i\geq 2$, along with $2k$ zeros per Equation (\ref{S2}).
We then have
\[r_{-1}=\sum_{i=1}^{k-1}i=\dfrac{k(k-1)}{2},\]
\[r_{0}=\left(\sum_{i=1}^{k-1}2i\right)+2k+\left(\sum_{i=1}^{k-1}i\right)
=\dfrac{k(3k+1)}{2},\]
\[r_{1}=\left(\sum_{i=1}^{k}i\right)+\left(\sum_{i=0}^{k-1}2i+1\right)
=\dfrac{k(3k+1)}{2},\]
\[r_{2}=\sum_{i=1}^{k-1}i=\dfrac{k(k-1)}{2}.\]

\end{proof} 

We have the following corollary that will be used to prove the unbroken property.  

\begin{corollary}
Let $\mf{p}(\pi_1\dd \pi_2)$ be a Frobenius seaweed, and let $\sigma$ be a maximally connected component of type A, B, C, or D.  If $\mathcal{E}(\sigma)$ is an unbroken multiset, then $\mathcal{E}(\sigma) \cup [-\mathcal{E}(\sigma)]$ is an unbroken multiset.  Moreover, if $x$ is symmetric to any eigenvalue in $\mathcal{E}(\sigma)$, then $\mathcal{E}(\sigma) \cup \{x\}$ is an unbroken multiset.  
\end{corollary}

\subsection{Unbroken}

A key concept for showing the spectrum is unbroken is a ``U-turn" in an orbit of the orbit meander.  
By U-turn, we mean an application of Lemma \ref{eigenvalue orbit} 
in the case that $(\alpha_i,i_1\alpha_i)<0$.  Since this case applies only to components of type A, if $(\alpha_i,i_1\alpha_i)<0$, then $\alpha_i$ and $i_1\alpha_i$ must be adjacent. It follows that the type-A component must have an even number of vertices. 
To see this, note that in the orbit meander, $\alpha_i$ and $i_1\alpha_i$ are connected by a dashed edge.  A component of type A with an odd number of vertices does not have any adjacent vertices connected by a dashed edge since the middle vertex is isolated. For example, the orbit meander in Figure \ref{fig:simple eigenvalues} has two U-turns: one in the orbit 
$\{\alpha_7,\alpha_6,\alpha_8\}$ and one in the orbit $\{\alpha_4,\alpha_3\}$.  We will find it convenient to break U-turns into two types of U-turns: \textit{right U-turns} and \textit{left U-turns}.  

Now, arrange all orbits so that the first entry is a fixed point in $\pi_1\cap\pi_2$.
If a U-turn involves a dashed edge
from $v_i^-$ to $v_{i-1}^-$, or a dashed edge from $v_i^+$ to $v_{i+1}^+$, we call this a right U-turn.  
Similarly, if a U-turn involves a dashed edge
from $v_i^-$ to $v_{i+1}^-$, or a dashed edge from $v_i^+$ to $v_{i-1}^+$, we call this a left U-turn.  See Figure \ref{fig:uturns}; the right U-turn is represented by a red dashed arc, and the left U-turns are blue.

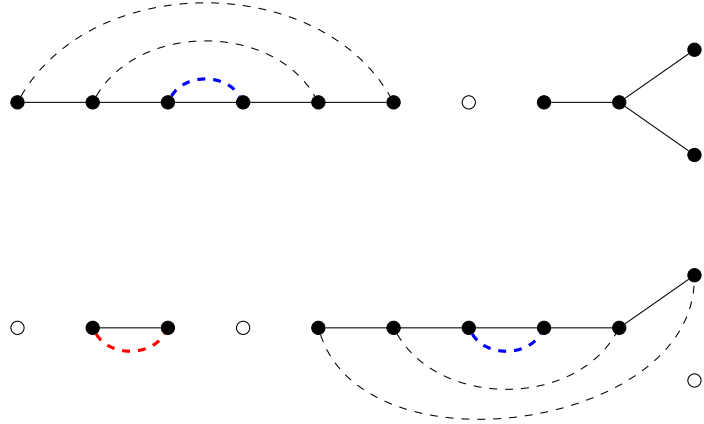
\begin{figure}[H]
\[\begin{tikzpicture}
[decoration={markings,mark=at position 0.6 with 
{\arrow{angle 90}{>}}}]

\draw (4,-3) node[draw,circle,fill=white,minimum size=5pt,inner sep=0pt] (4+) {};
\draw (5,-3) node[draw,circle,fill=black,minimum size=5pt,inner sep=0pt] (5+) {};
\draw (6,-3) node[draw,circle,fill=black,minimum size=5pt,inner sep=0pt] (6+) {};
\draw (7,-3) node[draw,circle,fill=white,minimum size=5pt,inner sep=0pt] (7+) {};
\draw (8,-3) node[draw,circle,fill=black,minimum size=5pt,inner sep=0pt] (8+) {};
\draw (9,-3) node[draw,circle,fill=black,minimum size=5pt,inner sep=0pt] (9+) {};
\draw (10,-3) node[draw,circle,fill=black,minimum size=5pt,inner sep=0pt] (10+) {};
\draw (11,-3) node[draw,circle,fill=black,minimum size=5pt,inner sep=0pt] (11+) {};
\draw (12,-3) node[draw,circle,fill=black,minimum size=5pt,inner sep=0pt] (12+) {};
\draw (13,-2.3) node[draw,circle,fill=black,minimum size=5pt,inner sep=0pt] (13+) {};
\draw (13,-3.7) node[draw,circle,fill=white,minimum size=5pt,inner sep=0pt] (14+) {};

\draw (4,0) node[draw,circle,fill=black,minimum size=5pt,inner sep=0pt] (4-) {};
\draw (5,0) node[draw,circle,fill=black,minimum size=5pt,inner sep=0pt] (5-) {};
\draw (6,0) node[draw,circle,fill=black,minimum size=5pt,inner sep=0pt] (6-) {};
\draw (7,0) node[draw,circle,fill=black,minimum size=5pt,inner sep=0pt] (7-) {};
\draw (8,0) node[draw,circle,fill=black,minimum size=5pt,inner sep=0pt] (8-) {};
\draw (9,0) node[draw,circle,fill=black,minimum size=5pt,inner sep=0pt] (9-) {};
\draw (10,0) node[draw,circle,fill=white,minimum size=5pt,inner sep=0pt] (10-) {};
\draw (11,0) node[draw,circle,fill=black,minimum size=5pt,inner sep=0pt] (11-) {};
\draw (12,0) node[draw,circle,fill=black,minimum size=5pt,inner sep=0pt] (12-) {};
\draw (13,.7) node[draw,circle,fill=black,minimum size=5pt,inner sep=0pt] (13-) {};
\draw (13,-.7) node[draw,circle,fill=black,minimum size=5pt,inner sep=0pt] (14-) {};

\draw (4-) to (9-);
\draw (11-) to (12-);
\draw (12-) to (13-);
\draw (12-) to (14-);
\draw (5+) to (6+);
\draw (8+) to (12+);
\draw (12+) to (13+);

\draw [dashed, color=red, line width = 1.2pt] (5+) to [bend right=60] (6+);
\draw [dashed] (8+) to [bend right=80] (13+);
\draw [dashed] (9+) to [bend right=60] (12+);
\draw [dashed, color=blue, line width = 1.2pt] (10+) to [bend right=60] (11+);

\draw [dashed] (4-) to [bend left=60] (9-);
\draw [dashed] (5-) to [bend left=60] (8-);
\draw [dashed, color=blue, line width = 1.2pt] (6-) to [bend left=60] (7-);

;\end{tikzpicture}\]
\caption{The orbit meander of $\mf{p}_{11}^\D(\Phi_1 \dd \Phi_2)$ with U-turns highlighted}
\label{fig:uturns}
\end{figure}

 The orbit meander for $\mf{p}_{11}^\D(\Phi_1 \dd \Phi_2)$ has four orbits:
$$\mathcal{O}_1 = \{\alpha_2, \alpha_7, \alpha_{10}, \alpha_9, \alpha_8\}, \quad 
\mathcal{O}_2 = \{\alpha_3, \alpha_6, \alpha_{11}\}, \quad 
\mathcal{O}_3 = \{\alpha_4, \alpha_5\}, ~~ \text{and} ~~ 
\mathcal{O}_4 = \{\alpha_1\}.$$
\noindent
Note that $\mathcal{O}_2$ and $\mathcal{O}_4$ contain no U-turns, $\mathcal{O}_1$ has two U-turns (a right and a left) and $\mathcal{O}_3$ has a single left U-turn.
As it turns out, an orbit cannot have more than two U-turns.  
Moreover, if an orbit has two U-turns, one must be a right U-turn and one must be a left U-turn.  
We record this in the following Lemma.

\begin{lemma}[U-turn Lemma]\label{lem:Dsimple}
An orbit in an orbit meander contains at most two U-turns.  Moreover, if an orbit contains two U-turns, one must be a right U-turn and one must be a left U-turn.
\end{lemma}

\begin{proof}

The orbit meander for a Frobenius type-A seaweed has exactly two maximally connected
components of even size, so an orbit can have at most two U-turns. For a seaweed of type B, C, or D, it is possible to 
have more than two maximally connected components of even size.
Indeed, start from a fixed point contained in $\pi_1\cap\pi_2$.
Without loss of generality, we proceed with the proof in the case that the first U-turn is a right U-turn.

The orbit may make a second U-turn. However, it cannot make a second
right U-turn, as the orbit would self-intersect 
and never terminate at an element of $\pi_\cup$. Therefore, a second U-turn must be to the left. But now the orbit has
previously traced edges on both the left and right. Additional U-turns
are not allowed.  
\end{proof}





Simple eigenvalues can increase in absolute value by one only through a U-turn.  This observation together with the U-turn Lemma gives us the following critical result.

\begin{theorem}\label{DSimple}
Let $\mf{g}$ be a Frobenius seaweed of type A or type D.  If $\widehat{F}$ is a principal element of $\mf{g}$ and $\alpha_i(\widehat{F})$ is a simple eigenvalue, then $$\alpha_i(\widehat{F}) \in \{-2,-1,0,1,2,3\}.$$
In particular, the simple eigenvalues are bounded in absolute value by three.
\end{theorem}

\begin{remark}
Theorem \ref{DSimple} combined with the results of \textbf{\cite{Coll typea}} gives us the following complete table of possible simple eigenvalues for each classical case.  

\begin{table}[H]
\[\begin{tabular}{|c|c|}
\hline
Component type & Possible values of $\alpha_i(\widehat{F})$  \\
\hline
\hline
A & 
$-2, -1, 0, 1, 2, 3$ \\
\hline
B & $-1, 0, 1$  \\
\hline
C & $0,1$\\
\hline
D & $-2,-1,0,1,2,3$\\
\hline
\end{tabular}\]
\caption{Values of $\alpha_i(\widehat{F})$}
\label{tab:Dsimple eigenvalue}
\end{table}
\end{remark}




To ease discourse, we note the following bases for the subsequent induction before  detailing the winding-up moves.
These bases are established in \textbf{\cite{DIndex}} and follow from a general combinatorial formula for the index of a type-D seaweed. 

\begin{theorem}\label{typeDIndBases}
The following seaweeds are the bases for induction.   
\begin{enumerate}[\textup(i\textup)]
\item $D_{2k}$ with $k \geq 2$ is
$\mf{p}_{2k}^D(\{\alpha_{2k}, \alpha_{2k-1}, ...,\alpha_1\}
\dd \emptyset)$.  
\item $D_{2k+1}$ with $k \geq 1$  $\mf{p}_{2k+1}^D(\{\alpha_{2k+1}, \alpha_{2k}, ...,\alpha_1\} \dd \{\alpha_2\})$, and
\item $D_{2k+1}$ with $k \geq 1$ $\mf{p}_{2k+1}^D(\{\alpha_{2k+1}, \alpha_{2k}, ...,\alpha_1\} \dd \{\alpha_3, \alpha_2\})$.  
\end{enumerate}
\end{theorem}

To prove the unbroken property, we will find it convenient to express these seaweeds in terms of sequences of flags defining them.   Let $C_{\leq n}$ denote the set of sequences of positive integers whose sum 
is less than or equal to $n$, and call each integer in the string a \textit{part}.
Let $\mathcal{P}(X)$ denote the power set of a set $X$.
Given $\ul{a}=(a_1,a_2,\dots ,a_m)\in \Cn$, define a bijection
$\varphi:\Cn\rightarrow \mathcal{P}(\Pi)$ by
\[\varphi(\ul{a})=\{\alpha_{n+1-a_1},\alpha_{n+1-(a_1+a_2)},
\dots ,\alpha_{n+1-(a_1+a_2+\dots +a_m)}\}.\]
Then define \[\mf{p}_n(\ul{a} \dd \ul{b})=
\mf{p}\left(\Pi\setminus\varphi(\ul{a}) \dd \Pi\setminus\varphi(\ul{b})\right),\]
and let $\mathcal{M}_n(\ul{a} \dd \ul{b})$ denote the orbit meander of 
$\mf{p}_n(\ul{a} \dd \ul{b})$.

\begin{example}
For example,
$$\mf{p}_8^\C(\{\alpha_6,\alpha_5,\alpha_4,\alpha_3\}
\dd \{\alpha_8,\alpha_7,\alpha_6,\alpha_4,
\alpha_3,\alpha_2,\alpha_1\})=
\mf{p}^\C_8((1,1,5,1)\dd (4)).$$
\end{example}

We now write the bases for induction from Theorem \ref{typeDIndBases} using sequences of integers.  

\begin{theorem}\label{typeDSpectrumFrobParabolicBlocks}
The following seaweeds are the bases for the induction.   
\begin{enumerate}[\textup(i\textup)$^{\prime}$]
\item $\mf{p}_q^\D\left( 1^q \dd \emptyset \right)$ with $q$ even,
\item $\mf{p}_q^\D\left( 1^{q-2}, 2 \dd \emptyset \right)$ with $q$ odd,
\item $\mf{p}_q^\D\left( 1^{q-3}, 3 \dd \emptyset \right)$ with $q$ odd.
\end{enumerate}
\end{theorem}

\begin{example}\label{Dbases}
The following Figure \ref{typedbase} illustrates the inductive bases of Theorem \ref{typeDSpectrumFrobParabolicBlocks}.  
\end{example}

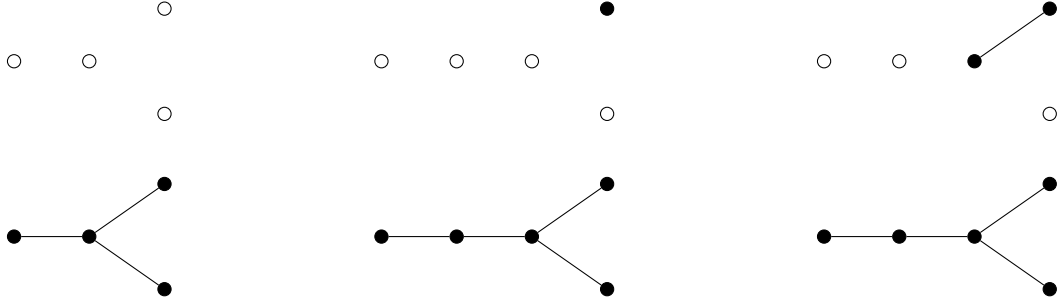
\begin{figure}[H]
$$
\begin{tikzpicture}
[decoration={markings,mark=at position 0.6 with 
{\arrow{angle 90}{>}}}]

\draw (2,.75) node[draw,circle,fill=white,minimum size=5pt,inner sep=0pt] (2+) {};
\draw (3,.75) node[draw,circle,fill=white,minimum size=5pt,inner sep=0pt] (3+) {};
\draw (4,1.45) node[draw,circle,fill=white,minimum size=5pt,inner sep=0pt] (4+) {};
\draw (4,.05) node[draw,circle,fill=white,minimum size=5pt,inner sep=0pt] (5+) {};



;\end{tikzpicture}
\hspace{7em}
\begin{tikzpicture}
[decoration={markings,mark=at position 0.6 with 
{\arrow{angle 90}{>}}}]

\draw (1,.75) node[draw,circle,fill=white,minimum size=5pt,inner sep=0pt] (1+) {};
\draw (2,.75) node[draw,circle,fill=white,minimum size=5pt,inner sep=0pt] (2+) {};
\draw (3,.75) node[draw,circle,fill=white,minimum size=5pt,inner sep=0pt] (3+) {};
\draw (4,1.45) node[draw,circle,fill=black,minimum size=5pt,inner sep=0pt] (4+) {};
\draw (4,.05) node[draw,circle,fill=white,minimum size=5pt,inner sep=0pt] (5+) {};



;\end{tikzpicture}
\hspace{7em}
\begin{tikzpicture}
[decoration={markings,mark=at position 0.6 with 
{\arrow{angle 90}{>}}}]

\draw (1,.75) node[draw,circle,fill=white,minimum size=5pt,inner sep=0pt] (1+) {};
\draw (2,.75) node[draw,circle,fill=white,minimum size=5pt,inner sep=0pt] (2+) {};
\draw (3,.75) node[draw,circle,fill=black,minimum size=5pt,inner sep=0pt] (3+) {};
\draw (4,1.45) node[draw,circle,fill=black,minimum size=5pt,inner sep=0pt] (4+) {};
\draw (4,.05) node[draw,circle,fill=white,minimum size=5pt,inner sep=0pt] (5+) {};

\draw (3+) to (4+);


;\end{tikzpicture}
$$

$$
\begin{tikzpicture}
[decoration={markings,mark=at position 0.6 with 
{\arrow{angle 90}{>}}}]

\draw (2,.75) node[draw,circle,fill=black,minimum size=5pt,inner sep=0pt] (2+) {};
\draw (3,.75) node[draw,circle,fill=black,minimum size=5pt,inner sep=0pt] (3+) {};
\draw (4,1.45) node[draw,circle,fill=black,minimum size=5pt,inner sep=0pt] (4+) {};
\draw (4,.05) node[draw,circle,fill=black,minimum size=5pt,inner sep=0pt] (5+) {};

\draw (2+) to (3+);
\draw (3+) to (4+);
\draw (3+) to (5+);


;\end{tikzpicture}
\hspace{7em}
\begin{tikzpicture}
[decoration={markings,mark=at position 0.6 with 
{\arrow{angle 90}{>}}}]

\draw (1,.75) node[draw,circle,fill=black,minimum size=5pt,inner sep=0pt] (1+) {};
\draw (2,.75) node[draw,circle,fill=black,minimum size=5pt,inner sep=0pt] (2+) {};
\draw (3,.75) node[draw,circle,fill=black,minimum size=5pt,inner sep=0pt] (3+) {};
\draw (4,1.45) node[draw,circle,fill=black,minimum size=5pt,inner sep=0pt] (4+) {};
\draw (4,.05) node[draw,circle,fill=black,minimum size=5pt,inner sep=0pt] (5+) {};

\draw (1+) to (2+);
\draw (2+) to (3+);
\draw (3+) to (4+);
\draw (3+) to (5+);


;\end{tikzpicture}
\hspace{7em}
\begin{tikzpicture}
[decoration={markings,mark=at position 0.6 with 
{\arrow{angle 90}{>}}}]

\draw (1,.75) node[draw,circle,fill=black,minimum size=5pt,inner sep=0pt] (1+) {};
\draw (2,.75) node[draw,circle,fill=black,minimum size=5pt,inner sep=0pt] (2+) {};
\draw (3,.75) node[draw,circle,fill=black,minimum size=5pt,inner sep=0pt] (3+) {};
\draw (4,1.45) node[draw,circle,fill=black,minimum size=5pt,inner sep=0pt] (4+) {};
\draw (4,.05) node[draw,circle,fill=black,minimum size=5pt,inner sep=0pt] (5+) {};

\draw (1+) to (2+);
\draw (2+) to (3+);
\draw (3+) to (4+);
\draw (3+) to (5+);


;\end{tikzpicture}
$$

\caption{Type-D inductive bases with $q=4$ and $q=5$}
\label{typedbase}
\end{figure}

If $a_1+\dots +a_m=n$, then each part $a_i$ corresponds to a 
maximally connected component $\sigma$ of cardinality $|\sigma|=a_i-1$,
all of which are of type A.
If $a_1+\dots +a_m=r<n$, then each part $a_i$ corresponds to a 
type-A maximally connected component $\sigma$ of cardinality $|\sigma|=a_i-1$,
and there is one additional maximally connected component of cardinality $n-r$,
which is of type B, C, or D if $n-r>1$.

The following ``Winding-up" lemma can be used to develop any Frobenius orbit meander of any size or configuration.   It can be regarded as the inverse graph-theoretic rendering of Panyushev's well-known reduction \textbf{\cite{Panyushev1}}.

\begin{lemma}[Coll et al. \textbf{\cite{Meanders3}}, Lemma 4.2]\label{Expansion}
Let $\mathcal{M}_n(\ul{c} \dd \ul{d})$ be any type-B or type-C Frobenius
orbit meander, and without loss
of generality, assume that $\sum c_i=q+\sum d_i=n$.
Then $\mathcal{M}_n(\ul{c} \dd \ul{d})$ is the result of a sequence of the 
following moves starting from $\mathcal{M}_q(1^q \dd \emptyset)$.
Starting with an orbit meander 
$\mathcal{M}=\mathcal{M}_n\left(a_{1},a_{2}, \dots ,a_{m}\dd b_{1},b_{2}, \dots ,b_{t}\right)$, 
create an orbit meander $\mathcal{M}'$ by one of of the following:

\begin{enumerate}
\item {\bf Block Creation:}
$\displaystyle
\mathcal{M}'=\mathcal{M}_{n+a_1}(2a_{1},a_{2}, \dots ,a_{m}\dd a_1, b_{1},b_{2}, \dots ,b_{t})$,
\item {\bf Rotation Expansion:}
$\displaystyle
\mathcal{M}'=\mathcal{M}_{n+a_1-b_1}(2a_{1}-b_1,a_{2},a_3, \dots ,a_{m}\dd a_1, b_{2},b_{3}, \dots ,b_{t})$, provided that $a_1>b_1$,
\item {\bf Pure Expansion:}
$\displaystyle
\mathcal{M}'=\mathcal{M}_{n+a_2}(a_{1}+2a_2,a_{3},a_4, \dots ,a_{m}\dd a_2,b_{1},b_{2}, \dots ,b_{t})$,
\item {\bf Flip-Up:}
$\displaystyle
\mathcal{M}'=\mathcal{M}_{n}(b_{1},b_{2}, \dots ,b_{t}\dd a_{1},a_{2}, \dots ,a_{m})$.
\end{enumerate}
Similarly, any type-D Frobenius orbit meander is the result of a sequence of the same moves but starting from one of the bases given in Theorem \ref{typeDSpectrumFrobParabolicBlocks}.  
\end{lemma}

\begin{remark}
While we previously found it convenient to order the vertices of an orbit meander from right to left, to ease notation in the following proof, we will find it convenient to relabel the vertices going from left to right.  That is, if an orbit meander is defined by vertices
$\{v_n^+, v_{n-1}^+, ..., v_1^+\}$ and 
$\{v_n^-, v_{n-1}^-, ..., v_1^-\}$, redefine the orbit meander by 

$$\{v_n^+, v_{n-1}^+, ..., v_1^+\} \mapsto \{w_1^+, w_{2}^+, ..., w_n^+\}$$
$$\{v_n^-, v_{n-1}^-, ..., v_1^-\} \mapsto \{w_1^-, w_{2}^-, ..., w_n^-\}$$
\noindent 
This is done so that the induction is done on the block whose leftmost vertex is labeled $w_1^+$ rather than $v_n^+$.  
\end{remark}

\begin{theorem}\label{thm:unbroken}
If $\mathcal{M}(\pi_1\dd\pi_2)=\mathcal{M}(\ul{a}\dd\ul{b})$ is any Frobenius orbit meander, 
then
$\mathcal{E}(\sigma)$ is unbroken for every maximally connected component $\sigma$. 
\end{theorem}

\begin{proof} 
The proof is by  
induction on the number of Winding-up moves from Lemma \ref{Expansion} and
that $\mathcal{E}(\sigma)$ is unbroken for every maximally connected component $\sigma$.
Since $\mathcal{E}(\sigma)$ always contains $0$, Theorem \ref{thm:unbroken} implies the unbroken property of
Theorem \ref{thm:main}.

The base of the induction in type-A is 
$\mathcal{M}_1(1 \dd \emptyset)$, which has unbroken spectrum $\{0,1\}$.  

The base of the induction for either type B or type C is an orbit meander $\mathcal{M}_q(1^q\dd \emptyset)$
 where $q$ is a positive integer.
There is one maximally connected component $\sigma$, of type B or type C, and by Theorem \ref{thm:symmetry}, $\mathcal{E}(\sigma)$ is unbroken either way.

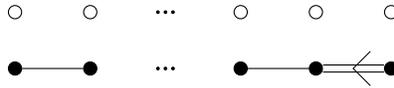
\begin{figure}[H]
\[\begin{tikzpicture}
[decoration={markings,mark=at position 0.6 with 
{\arrow{angle 90}{>}}}]

\draw (3,-.75) node[draw,circle,fill=black,minimum size=5pt,inner sep=0pt] (3+) {};
\draw (4,-.75) node[draw,circle,fill=black,minimum size=5pt,inner sep=0pt] (4+) {};
\draw (4.9,-.75) node[draw,circle,fill=black,minimum size=1pt,inner sep=0pt] (4.9+) {};
\draw (5,-.75) node[draw,circle,fill=black,minimum size=1pt,inner sep=0pt] (5+) {};
\draw (5.1,-.75) node[draw,circle,fill=black,minimum size=1pt,inner sep=0pt] (5.1+) {};
\draw (6,-.75) node[draw,circle,fill=black,minimum size=5pt,inner sep=0pt] (6+) {};
\draw (7,-.75) node[draw,circle,fill=black,minimum size=5pt,inner sep=0pt] (7+) {};
\draw (8,-.75) node[draw,circle,fill=black,minimum size=5pt,inner sep=0pt] (8+) {};

\draw (3,0) node[draw,circle,fill=white,minimum size=5pt,inner sep=0pt] (3-) {};
\draw (4,0) node[draw,circle,fill=white,minimum size=5pt,inner sep=0pt] (4-) {};
\draw (4.9,0) node[draw,circle,fill=black,minimum size=1pt,inner sep=0pt] (4.9-) {};
\draw (5,0) node[draw,circle,fill=black,minimum size=1pt,inner sep=0pt] (5-) {};
\draw (5.1,0) node[draw,circle,fill=black,minimum size=1pt,inner sep=0pt] (5.1-) {};
\draw (6,0) node[draw,circle,fill=white,minimum size=5pt,inner sep=0pt] (6-) {};
\draw (7,0) node[draw,circle,fill=white,minimum size=5pt,inner sep=0pt] (7-) {};
\draw (8,0) node[draw,circle,fill=white,minimum size=5pt,inner sep=0pt] (8-) {};

\draw (3+) to (4+);
\draw (6+) to (7+);
\draw [double distance=.8mm,postaction={decorate}] (8+) to (7+);

;\end{tikzpicture}\]
\caption{The Frobenius orbit meander $\mathcal{M}_q^\C(1^q\dd \emptyset)$}
\label{typeCbase}
\end{figure}

The base of the induction for type D is one of the orbit meanders from Example \ref{Dbases}.  The first two cases follow immediately from Theorem \ref{thm:symmetry}.  For the last, observe the orbit meander consists of a type-A component and a type-D component.  The eigenvalues contributed by the type-A component are
$\{-1, 0, 1, 2\}$.  The eigenvalues contributed by the type-D component are unbroken and contain $0$ by Theorem \ref{thm:symmetry}.  Their union is therefore unbroken.

Let $\mathcal{M} = \mathcal{M}(\ul{a}\dd\ul{b})$ be an orbit meander of classical type.  Suppose $\mathcal{E}(\sigma)$ is unbroken for each maximally connected component $\sigma$ in $\mathcal{M}$.  Let $\mathcal{M}'$ be the orbit meander resulting from applying one of the Winding-up moves from Lemma \ref{Expansion} to $\mathcal{M}$.  
For the inductive step, we need not consider the Flip-up move since this merely replaces $\mathcal{M}$ with an inverted isomorphic copy and consequently has no effect on the eigenvalue calculations. 
So, there are three cases we need to consider: block creation, rotation expansion, and pure expansion.  
We show $\mathcal{E}(\sigma)$ is unbroken for each maximally connected component $\sigma$ in $\mathcal{M}'$ that is not in $\mathcal{M}$.  
The following sets of vertices will assist in the computation of the eigenvalues.  
Note that the ordering of the sets in columns three and four of the following table are set up so that the vertices in the sets are in the order of the inducted-upon orbit meander.  

\newpage  

\begin{table}[H]
\begin{center}
\scalebox{0.83}{
\begin{tabular}{|c|c|c|c|}
\hline
\textbf{Move} & \textbf{Base} $\mathcal{M}$ & \textbf{New} $\mathcal{M}'$ & \textbf{New $\sigma$'s} \\
\hline
\hline
$\begin{array}{cc}
     \textbf{Block}  \\
     \textbf{Creation} 
\end{array}$ 
& 
$\begin{array}{l}
     A = \{w_1^+, w_{2}^+, ..., w_{a_1-1}^+\}  
\end{array}$ 
&
$\begin{array}{llll}
     B' = \{w_1^+, w_2^+, ..., w_{a_1-1}^+\}  \\
     C' = \{w_{a_1}^+\}  \\
     A' = \{w_{a_1+1}^+, w_{a_1+2}^+, ..., w_{2a_1-1}^+\} \\
     D' = \{w_1^-, w_2^-, ..., w_{a_1-1}^-\}  
\end{array}$  
&
$\begin{array}{ll}
     \sigma_1 = B' \cup C' \cup A'  \\ 
     \sigma_2 = D'  
\end{array}$ \\ 
\hline
$\begin{array}{cc}
     \textbf{Rotation}  \\
     \textbf{Expansion} 
\end{array}$ 
& 
$\begin{array}{lll}
     A = \{w_1^+, w_2^+, ..., w_{b_1-1}^+\}  \\ 
     B = \{w_{b_1}^+, w_{b_1+1}^+, ..., w_{a_1-1}^+\} \\
     C = \{w_1^-, w_2^-, ..., w_{b_1-1}^-\} 
\end{array}$
&
$\begin{array}{llll}
     A' = \{w_1^+, w_2^+, ..., w_{a_1-b_1}^+\}  \\
     C' = \{w_{a_1-b_1+1}^+, w_{a_1-b_1+2}^+, ..., w_{a_1-1}^+\}  \\
     B' = \{w_{a_1}^+, w_{a_1+1}^+, ..., w_{2a_1-b_1-1}^+\} \\
     D'= \{w_1^-, w_2^-, ...,w_{a_1-1}^-\} 
\end{array}$  
&
$\begin{array}{ll}
     \sigma_1 = A' \cup C' \cup B'  \\
     \sigma_2 = D'
\end{array}$ \\
\hline
$\begin{array}{cc}
     \textbf{Pure}  \\
     \textbf{Expansion} 
\end{array}$ & 
$\begin{array}{ll}
     A = \{w_1^+, w_2^+, ..., w_{a_1-1}^+\}  \\ 
     B = \{w_{a_1+1}^+, w_{a_1+2}^+, ..., w_{a_1+a_2-1}^+\}
\end{array}$
&
$\begin{array}{lll}
     B' = \{w_1^+, w_2^+, ..., w_{a_2-1}^+\} \\
     E' = \{w_{a_2}^+\} \\
     A' = \{w_{a_2+1}^+, w_{a_2+2}^+, ..., w_{a_1+a_2-1}\}  \\
     F' = \{w_{a_1+a_2}^+\} \\
     C' = \{w_{a_1+a_2+1}^+, w_{a_1+a_2+2}^+, ..., w_{a_1+2a_2-1}^+\} \\
     D' = \{w_1^-, w_2^-, ..., w_{a_2-1}^-\}
\end{array}$  
&
$\begin{array}{ll}
     \sigma_1 = B' \cup E' \cup A' \cup F' \cup C'  \\
     \sigma_2 = D' 
\end{array}$ \\
\hline
\end{tabular}
}
\caption{Sets in the Induction Proof}
\label{tab:induction}
\end{center}
\end{table}

Gray vertices in the orbit meanders associated to each Winding up move are vertices not impacted by the induction.  Furthermore, for brevity, the orbit meanders do not extend beyond the relevant top blocks.  

\setcounter{case}{0}
\begin{case}\label{unbroken1}
Block Creation:
\end{case}

\begin{figure}[H]
\[\begin{tikzpicture}
[decoration={markings,mark=at position 0.6 with 
{\arrow{angle 90}{>}}}]

\draw (.5,1) -- (4.5,1);
\draw (.5,.25) -- (4.5,.25);
\draw (.5,.25) -- (.5,1);
\draw (4.5,.25) -- (4.5,1);

\draw (1,.75) node[draw,circle,fill=black,minimum size=5pt,inner sep=0pt] (1+) {};
\draw (2,.75) node[draw,circle,fill=black,minimum size=5pt,inner sep=0pt] (2+) {};
\draw (3,.75) node[draw,circle,fill=black,minimum size=5pt,inner sep=0pt] (3+) {};
\draw (4,.75) node[draw,circle,fill=black,minimum size=5pt,inner sep=0pt] (4+) {};
\draw (5,.75) node[draw,circle,fill=white,minimum size=5pt,inner sep=0pt] (5+) {};

\draw (1,0) node[draw,circle,fill=gray,minimum size=5pt,inner sep=0pt] (1-) {};
\draw (2,0) node[draw,circle,fill=gray,minimum size=5pt,inner sep=0pt] (2-) {};
\draw (3,0) node[draw,circle,fill=gray,minimum size=5pt,inner sep=0pt] (3-) {};
\draw (4,0) node[draw,circle,fill=gray,minimum size=5pt,inner sep=0pt] (4-) {};
\draw (5,0) node[draw,circle,fill=gray,minimum size=5pt,inner sep=0pt] (5-) {};

\draw (1+) to (2+);
\draw (2+) to (3+);
\draw (3+) to (4+);

\draw [dashed] (1+) to [bend left=60] (4+);
\draw [dashed] (2+) to [bend left=60] (3+);

\node at (.5,1.25) {$A$};
\node at (1,.5) {$w_1^+$};
\node at (2,.5) {$w_2^+$};
\node at (3,.5) {$w_3^+$};
\node at (4,.5) {$w_4^+$};

;\end{tikzpicture}
\hspace{1cm}
\begin{tikzpicture}
[decoration={markings,mark=at position 0.6 with 
{\arrow{angle 90}{>}}}]

\draw (.5,1.5) -- (4.5,1.5);
\draw (.5,.75) -- (4.5,.75);
\draw (.5,.75) -- (.5,1.5);
\draw (4.5,.75) -- (4.5,1.5);

\draw (5.5,1.5) -- (9.5,1.5);
\draw (5.5,.75) -- (9.5,.75);
\draw (5.5,.75) -- (5.5,1.5);
\draw (9.5,.75) -- (9.5,1.5);

\draw (4.6,1.5) -- (5.4,1.5);
\draw (4.6,.75) -- (5.4,.75);
\draw (4.6,.75) -- (4.6,1.5);
\draw (5.4,.75) -- (5.4,1.5);

\draw (.5,-.25) -- (4.5,-.25);
\draw (.5,.6) -- (4.5,.6);
\draw (.5,.6) -- (.5,-.25);
\draw (4.5,.6) -- (4.5,-.25);

\draw (1,1.25) node[draw,circle,fill=black,minimum size=5pt,inner sep=0pt] (1+) {};
\draw (2,1.25) node[draw,circle,fill=black,minimum size=5pt,inner sep=0pt] (2+) {};
\draw (3,1.25) node[draw,circle,fill=black,minimum size=5pt,inner sep=0pt] (3+) {};
\draw (4,1.25) node[draw,circle,fill=black,minimum size=5pt,inner sep=0pt] (4+) {};
\draw (5,1.25) node[draw,circle,fill=black,minimum size=5pt,inner sep=0pt] (5+) {};
\draw (6,1.25) node[draw,circle,fill=black,minimum size=5pt,inner sep=0pt] (6+) {};
\draw (7,1.25) node[draw,circle,fill=black,minimum size=5pt,inner sep=0pt] (7+) {};
\draw (8,1.25) node[draw,circle,fill=black,minimum size=5pt,inner sep=0pt] (8+) {};
\draw (9,1.25) node[draw,circle,fill=black,minimum size=5pt,inner sep=0pt] (9+) {};
\draw (10,1.25) node[draw,circle,fill=white,minimum size=5pt,inner sep=0pt] (10+) {};

\draw (1,0) node[draw,circle,fill=black,minimum size=5pt,inner sep=0pt] (1-) {};
\draw (2,0) node[draw,circle,fill=black,minimum size=5pt,inner sep=0pt] (2-) {};
\draw (3,0) node[draw,circle,fill=black,minimum size=5pt,inner sep=0pt] (3-) {};
\draw (4,0) node[draw,circle,fill=black,minimum size=5pt,inner sep=0pt] (4-) {};
\draw (5,0) node[draw,circle,fill=white,minimum size=5pt,inner sep=0pt] (5-) {};
\draw (6,0) node[draw,circle,fill=gray,minimum size=5pt,inner sep=0pt] (6-) {};
\draw (7,0) node[draw,circle,fill=gray,minimum size=5pt,inner sep=0pt] (7-) {};
\draw (8,0) node[draw,circle,fill=gray,minimum size=5pt,inner sep=0pt] (8-) {};
\draw (9,0) node[draw,circle,fill=gray,minimum size=5pt,inner sep=0pt] (9-) {};
\draw (10,0) node[draw,circle,fill=gray,minimum size=5pt,inner sep=0pt] (10-) {};

\draw (1+) to (2+);
\draw (2+) to (3+);
\draw (3+) to (4+);
\draw (4+) to (5+);
\draw (5+) to (6+);
\draw (6+) to (7+);
\draw (7+) to (8+);
\draw (8+) to (9+);
\draw (1-) to (2-);
\draw (2-) to (3-);
\draw (3-) to (4-);

\draw [dashed] (1+) to [bend left=60] (9+);
\draw [dashed] (2+) to [bend left=60] (8+);
\draw [dashed] (3+) to [bend left=60] (7+);
\draw [dashed] (4+) to [bend left=90] (6+);
\draw [dashed] (1-) to [bend right=60] (4-);
\draw [dashed] (2-) to [bend right=60] (3-);

\node at (.5,-.5) {$D'$};
\node at (.5,1.75) {$B'$};
\node at (5,1.75) {$C'$};
\node at (9.5,1.75) {$A'$};
\node at (1,1) {$w_1^+$};
\node at (2,1) {$w_2^+$};
\node at (3,1) {$w_3^+$};
\node at (4,1) {$w_4^+$};
\node at (5,1) {$w_5^+$};
\node at (6,1) {$w_6^+$};
\node at (7,1) {$w_7^+$};
\node at (8,1) {$w_8^+$};
\node at (9,1) {$w_9^+$};
\node at (1,.35) {$w_1^-$};
\node at (2,.35) {$w_2^-$};
\node at (3,.35) {$w_3^-$};
\node at (4,.35) {$w_4^-$};

;\end{tikzpicture}
\]
\caption{Block Creation applied to $\mathcal{M}$ with $a_1 = 5$ (left) to obtain $\mathcal{M}'$ (right)}
\label{fig:InductionBlockCreation}
\end{figure}
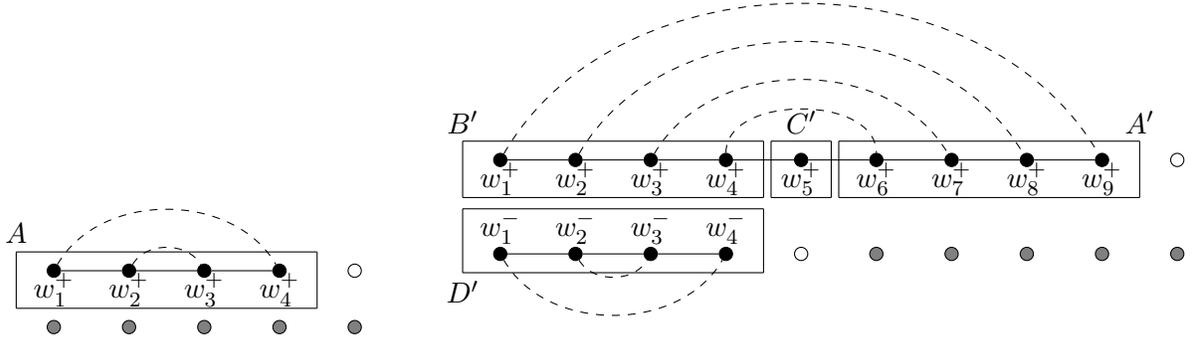

We have  

\noindent 
\begin{eqnarray}
\label{ind1.1} \mathcal{E}(A') &=& \mathcal{E}(A),\\
\label{ind1.2} \mathcal{E}(B') &=& -\mathcal{E}(A), \\
\label{ind1.3} \mathcal{E}(D') &=& \mathcal{E}(A).
\end{eqnarray} 
Note that $\mathcal{E}(A)$ is unbroken by induction.  So, by Equation (\ref{ind1.3}), $\mathcal{E}(\sigma_2)$ is unbroken.  

Any number in 
$\mathcal{E}(A' \cup C')$
is in either $\mathcal{E}(A')$ or $\mathcal{E}(A')+1$, so by equation (\ref{ind1.1}), $\mathcal{E}(A' \cup C')$ is unbroken.  By symmetry and equation (\ref{ind1.2}), 
$\mathcal{E}(B' \cup C')$ is unbroken.  Hence
$\mathcal{E}(B' \cup C') \cup \mathcal{E}(A' \cup C')$ is unbroken.  
Any remaining eigenvalue in $\mathcal{E}(\sigma_1)$ is either one or symmetric to an eigenvalue in $\mathcal{E}(B' \cup C') \cup \mathcal{E}(A' \cup C')$,  
so $\mathcal{E}(\sigma_1)$ is unbroken.

\begin{case}\label{unbroken2}
Rotation Expansion:
\end{case}

\begin{figure}[H]
\[\begin{tikzpicture}
[decoration={markings,mark=at position 0.6 with 
{\arrow{angle 90}{>}}}]

\draw (.5,1.5) -- (2.5,1.5);
\draw (.5,.75) -- (2.5,.75);
\draw (.5,.75) -- (.5,1.5);
\draw (2.5,.75) -- (2.5,1.5);

\draw (2.6,1.5) -- (4.5,1.5);
\draw (2.6,.75) -- (4.5,.75);
\draw (2.6,.75) -- (2.6,1.5);
\draw (4.5,.75) -- (4.5,1.5);

\draw (.5,-.25) -- (2.5,-.25);
\draw (.5,.6) -- (2.5,.6);
\draw (.5,.6) -- (.5,-.25);
\draw (2.5,.6) -- (2.5,-.25);

\draw (1,1.25) node[draw,circle,fill=black,minimum size=5pt,inner sep=0pt] (1+) {};
\draw (2,1.25) node[draw,circle,fill=black,minimum size=5pt,inner sep=0pt] (2+) {};
\draw (3,1.25) node[draw,circle,fill=black,minimum size=5pt,inner sep=0pt] (3+) {};
\draw (4,1.25) node[draw,circle,fill=black,minimum size=5pt,inner sep=0pt] (4+) {};
\draw (5,1.25) node[draw,circle,fill=white,minimum size=5pt,inner sep=0pt] (5+) {};

\draw (1,0) node[draw,circle,fill=black,minimum size=5pt,inner sep=0pt] (1-) {};
\draw (2,0) node[draw,circle,fill=black,minimum size=5pt,inner sep=0pt] (2-) {};
\draw (3,0) node[draw,circle,fill=white,minimum size=5pt,inner sep=0pt] (3-) {};
\draw (4,0) node[draw,circle,fill=gray,minimum size=5pt,inner sep=0pt] (4-) {};
\draw (5,0) node[draw,circle,fill=gray,minimum size=5pt,inner sep=0pt] (5-) {};

\draw (1+) to (2+);
\draw (2+) to (3+);
\draw (3+) to (4+);

\draw [dashed] (1+) to [bend left=60] (4+);
\draw [dashed] (2+) to [bend left=60] (3+);
\draw [dashed] (1-) to [bend right=60] (2-);

\node at (.5,1.75) {$A$};
\node at (4.5,1.75) {$B$};
\node at (.5,-.5) {$C$};
\node at (1,1) {$w_1^+$};
\node at (2,1) {$w_2^+$};
\node at (3,1) {$w_3^+$};
\node at (4,1) {$w_4^+$};
\node at (1,.35) {$w_1^-$};
\node at (2,.35) {$w_2^-$};

;\end{tikzpicture}
\hspace{1cm}
\begin{tikzpicture}
[decoration={markings,mark=at position 0.6 with 
{\arrow{angle 90}{>}}}]

\draw (.5,1.5) -- (2.5,1.5);
\draw (.5,.75) -- (2.5,.75);
\draw (.5,.75) -- (.5,1.5);
\draw (2.5,.75) -- (2.5,1.5);

\draw (4.5,1.5) -- (6.5,1.5);
\draw (4.5,.75) -- (6.5,.75);
\draw (4.5,.75) -- (4.5,1.5);
\draw (6.5,.75) -- (6.5,1.5);

\draw (2.6,1.5) -- (4.4,1.5);
\draw (2.6,.75) -- (4.4,.75);
\draw (2.6,.75) -- (2.6,1.5);
\draw (4.4,.75) -- (4.4,1.5);

\draw (.5,-.25) -- (4.5,-.25);
\draw (.5,.6) -- (4.5,.6);
\draw (.5,.6) -- (.5,-.25);
\draw (4.5,.6) -- (4.5,-.25);

\draw (1,1.25) node[draw,circle,fill=black,minimum size=5pt,inner sep=0pt] (1+) {};
\draw (2,1.25) node[draw,circle,fill=black,minimum size=5pt,inner sep=0pt] (2+) {};
\draw (3,1.25) node[draw,circle,fill=black,minimum size=5pt,inner sep=0pt] (3+) {};
\draw (4,1.25) node[draw,circle,fill=black,minimum size=5pt,inner sep=0pt] (4+) {};
\draw (5,1.25) node[draw,circle,fill=black,minimum size=5pt,inner sep=0pt] (5+) {};
\draw (6,1.25) node[draw,circle,fill=black,minimum size=5pt,inner sep=0pt] (6+) {};
\draw (7,1.25) node[draw,circle,fill=white,minimum size=5pt,inner sep=0pt] (7+) {};

\draw (1,0) node[draw,circle,fill=black,minimum size=5pt,inner sep=0pt] (1-) {};
\draw (2,0) node[draw,circle,fill=black,minimum size=5pt,inner sep=0pt] (2-) {};
\draw (3,0) node[draw,circle,fill=black,minimum size=5pt,inner sep=0pt] (3-) {};
\draw (4,0) node[draw,circle,fill=black,minimum size=5pt,inner sep=0pt] (4-) {};
\draw (5,0) node[draw,circle,fill=white,minimum size=5pt,inner sep=0pt] (5-) {};
\draw (6,0) node[draw,circle,fill=gray,minimum size=5pt,inner sep=0pt] (6-) {};
\draw (7,0) node[draw,circle,fill=gray,minimum size=5pt,inner sep=0pt] (7-) {};

\draw (1+) to (6+);
\draw (1-) to (4-);

\draw [dashed] (1+) to [bend left=60] (6+);
\draw [dashed] (2+) to [bend left=60] (5+);
\draw [dashed] (3+) to [bend left=60] (4+);
\draw [dashed] (1-) to [bend right=60] (4-);
\draw [dashed] (2-) to [bend right=60] (3-);

\node at (.5,-.5) {$D'$};
\node at (.5,1.75) {$A'$};
\node at (3.5,1.75) {$C'$};
\node at (6.5,1.75) {$B'$};
\node at (1,1) {$w_1^+$};
\node at (2,1) {$w_2^+$};
\node at (3,1) {$w_3^+$};
\node at (4,1) {$w_4^+$};
\node at (5,1) {$w_5^+$};
\node at (6,1) {$w_6^+$};
\node at (1,.35) {$w_1^-$};
\node at (2,.35) {$w_2^-$};
\node at (3,.35) {$w_3^-$};
\node at (4,.35) {$w_4^-$};

;\end{tikzpicture}
\]
\caption{Rotation Expansion applied to $\mathcal{M}$ with $a_1 = 5$ and $b_1=3$ (left) to obtain $\mathcal{M}'$ (right)}
\label{fig:InductionRotationExpansion}
\end{figure}
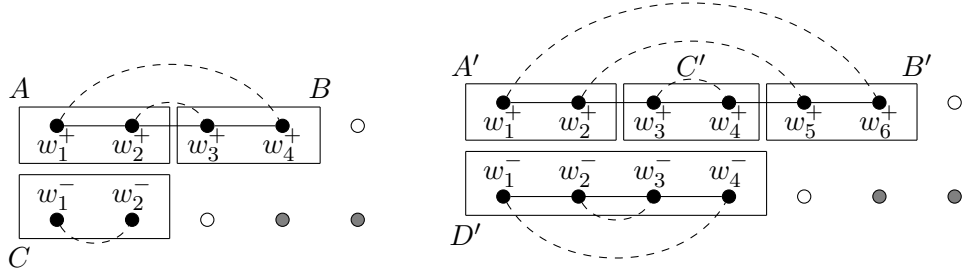

We have 
\begin{eqnarray}
\label{ind2.1} \mathcal{E}(D') &=& \mathcal{E}(A \cup B),\\
\label{ind2.2} \mathcal{E}(C') &=& -\mathcal{E}(C). 
\end{eqnarray}

\noindent 
By equation (\ref{ind2.1}), $\mathcal{E}(\sigma_2)$ is unbroken.  Without loss of generality, 
\begin{eqnarray}
\label{ind2.3} \mathcal{E}(A' \cup C') &=& -\mathcal{E}(A \cup B),
\end{eqnarray} 
so $\mathcal{E}(A' \cup C') \subseteq \mathcal{E}(\sigma_1)$ is unbroken.  All other eigenvalues in $\mathcal{E}(\sigma_1)$ are symmetric to eigenvalues in $\mathcal{E}(A' \cup C')$, and consequently, $\mathcal{E}(\sigma_1)$ is unbroken.  

\begin{case}\label{unbroken3}
Pure Expansion:
\end{case}

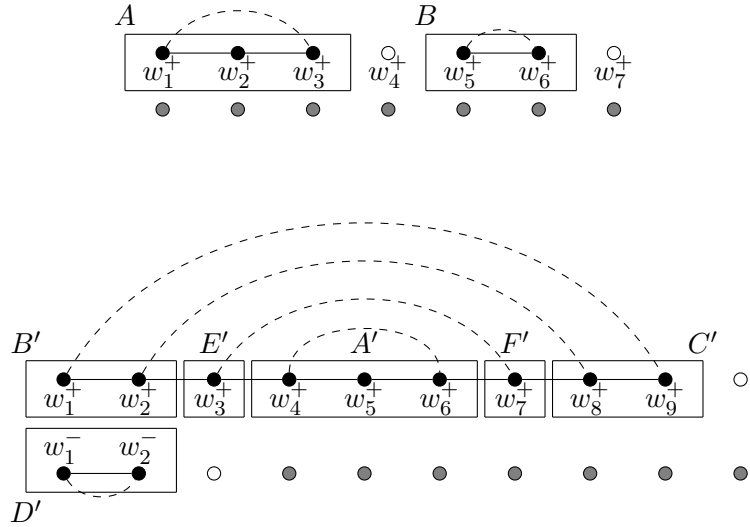
\begin{figure}[H]
\[\begin{tikzpicture}
[decoration={markings,mark=at position 0.6 with 
{\arrow{angle 90}{>}}}]

\draw (.5,1) -- (3.5,1);
\draw (.5,.25) -- (3.5,.25);
\draw (.5,.25) -- (.5,1);
\draw (3.5,.25) -- (3.5,1);

\draw (4.5,1) -- (6.5,1);
\draw (4.5,.25) -- (6.5,.25);
\draw (4.5,.25) -- (4.5,1);
\draw (6.5,.25) -- (6.5,1);

\draw (1,.75) node[draw,circle,fill=black,minimum size=5pt,inner sep=0pt] (1+) {};
\draw (2,.75) node[draw,circle,fill=black,minimum size=5pt,inner sep=0pt] (2+) {};
\draw (3,.75) node[draw,circle,fill=black,minimum size=5pt,inner sep=0pt] (3+) {};
\draw (4,.75) node[draw,circle,fill=white,minimum size=5pt,inner sep=0pt] (4+) {};
\draw (5,.75) node[draw,circle,fill=black,minimum size=5pt,inner sep=0pt] (5+) {};
\draw (6,.75) node[draw,circle,fill=black,minimum size=5pt,inner sep=0pt] (6+) {};
\draw (7,.75) node[draw,circle,fill=white,minimum size=5pt,inner sep=0pt] (7+) {};

\draw (1,0) node[draw,circle,fill=gray,minimum size=5pt,inner sep=0pt] (1-) {};
\draw (2,0) node[draw,circle,fill=gray,minimum size=5pt,inner sep=0pt] (2-) {};
\draw (3,0) node[draw,circle,fill=gray,minimum size=5pt,inner sep=0pt] (3-) {};
\draw (4,0) node[draw,circle,fill=gray,minimum size=5pt,inner sep=0pt] (4-) {};
\draw (5,0) node[draw,circle,fill=gray,minimum size=5pt,inner sep=0pt] (5-) {};
\draw (6,0) node[draw,circle,fill=gray,minimum size=5pt,inner sep=0pt] (6-) {};
\draw (7,0) node[draw,circle,fill=gray,minimum size=5pt,inner sep=0pt] (7-) {};

\draw (1+) to (2+);
\draw (2+) to (3+);
\draw (5+) to (6+);

\draw [dashed] (1+) to [bend left=60] (3+);
\draw [dashed] (5+) to [bend left=60] (6+);

\node at (.5,1.25) {$A$};
\node at (4.5,1.25) {$B$};
\node at (1,.5) {$w_1^+$};
\node at (2,.5) {$w_2^+$};
\node at (3,.5) {$w_3^+$};
\node at (4,.5) {$w_4^+$};
\node at (5,.5) {$w_5^+$};
\node at (6,.5) {$w_6^+$};
\node at (7,.5) {$w_7^+$};

;\end{tikzpicture}\]

\[\begin{tikzpicture}
[decoration={markings,mark=at position 0.6 with 
{\arrow{angle 90}{>}}}]

\draw (.5,1.5) -- (2.5,1.5);
\draw (.5,.75) -- (2.5,.75);
\draw (.5,.75) -- (.5,1.5);
\draw (2.5,.75) -- (2.5,1.5);

\draw (2.6,1.5) -- (3.4,1.5);
\draw (2.6,.75) -- (3.4,.75);
\draw (2.6,.75) -- (2.6,1.5);
\draw (3.4,.75) -- (3.4,1.5);

\draw (3.5,1.5) -- (6.5,1.5);
\draw (3.5,.75) -- (6.5,.75);
\draw (3.5,.75) -- (3.5,1.5);
\draw (6.5,.75) -- (6.5,1.5);

\draw (6.6,1.5) -- (7.4,1.5);
\draw (6.6,.75) -- (7.4,.75);
\draw (6.6,.75) -- (6.6,1.5);
\draw (7.4,.75) -- (7.4,1.5);

\draw (7.5,1.5) -- (9.5,1.5);
\draw (7.5,.75) -- (9.5,.75);
\draw (7.5,.75) -- (7.5,1.5);
\draw (9.5,.75) -- (9.5,1.5);

\draw (.5,-.25) -- (2.5,-.25);
\draw (.5,.6) -- (2.5,.6);
\draw (.5,.6) -- (.5,-.25);
\draw (2.5,.6) -- (2.5,-.25);

\draw (1,1.25) node[draw,circle,fill=black,minimum size=5pt,inner sep=0pt] (1+) {};
\draw (2,1.25) node[draw,circle,fill=black,minimum size=5pt,inner sep=0pt] (2+) {};
\draw (3,1.25) node[draw,circle,fill=black,minimum size=5pt,inner sep=0pt] (3+) {};
\draw (4,1.25) node[draw,circle,fill=black,minimum size=5pt,inner sep=0pt] (4+) {};
\draw (5,1.25) node[draw,circle,fill=black,minimum size=5pt,inner sep=0pt] (5+) {};
\draw (6,1.25) node[draw,circle,fill=black,minimum size=5pt,inner sep=0pt] (6+) {};
\draw (7,1.25) node[draw,circle,fill=black,minimum size=5pt,inner sep=0pt] (7+) {};
\draw (8,1.25) node[draw,circle,fill=black,minimum size=5pt,inner sep=0pt] (8+) {};
\draw (9,1.25) node[draw,circle,fill=black,minimum size=5pt,inner sep=0pt] (9+) {};
\draw (10,1.25) node[draw,circle,fill=white,minimum size=5pt,inner sep=0pt] (10+) {};

\draw (1,0) node[draw,circle,fill=black,minimum size=5pt,inner sep=0pt] (1-) {};
\draw (2,0) node[draw,circle,fill=black,minimum size=5pt,inner sep=0pt] (2-) {};
\draw (3,0) node[draw,circle,fill=white,minimum size=5pt,inner sep=0pt] (3-) {};
\draw (4,0) node[draw,circle,fill=gray,minimum size=5pt,inner sep=0pt] (4-) {};
\draw (5,0) node[draw,circle,fill=gray,minimum size=5pt,inner sep=0pt] (5-) {};
\draw (6,0) node[draw,circle,fill=gray,minimum size=5pt,inner sep=0pt] (6-) {};
\draw (7,0) node[draw,circle,fill=gray,minimum size=5pt,inner sep=0pt] (7-) {};
\draw (8,0) node[draw,circle,fill=gray,minimum size=5pt,inner sep=0pt] (8-) {};
\draw (9,0) node[draw,circle,fill=gray,minimum size=5pt,inner sep=0pt] (9-) {};
\draw (10,0) node[draw,circle,fill=gray,minimum size=5pt,inner sep=0pt] (10-) {};

\draw (1+) to (2+);
\draw (2+) to (3+);
\draw (3+) to (4+);
\draw (4+) to (5+);
\draw (5+) to (6+);
\draw (6+) to (7+);
\draw (7+) to (8+);
\draw (8+) to (9+);
\draw (1-) to (2-);

\draw [dashed] (1+) to [bend left=60] (9+);
\draw [dashed] (2+) to [bend left=60] (8+);
\draw [dashed] (3+) to [bend left=60] (7+);
\draw [dashed] (4+) to [bend left=90] (6+);
\draw [dashed] (1-) to [bend right=60] (2-);

\node at (.5,-.5) {$D'$};
\node at (.5,1.75) {$B'$};
\node at (3,1.75) {$E'$};
\node at (5,1.75) {$A'$};
\node at (7,1.75) {$F'$};
\node at (9.5,1.75) {$C'$};
\node at (1,1) {$w_1^+$};
\node at (2,1) {$w_2^+$};
\node at (3,1) {$w_3^+$};
\node at (4,1) {$w_4^+$};
\node at (5,1) {$w_5^+$};
\node at (6,1) {$w_6^+$};
\node at (7,1) {$w_7^+$};
\node at (8,1) {$w_8^+$};
\node at (9,1) {$w_9^+$};
\node at (1,.35) {$w_1^-$};
\node at (2,.35) {$w_2^-$};

;\end{tikzpicture}
\]
\caption{Pure Expansion applied to $\mathcal{M}$ with $a_1 = 4$ and $a_2=3$ (top) to obtain $\mathcal{M}'$ (bottom)}
\label{fig:InductionPureExpansion}
\end{figure}

We have 

\begin{eqnarray}
\label{ind2.4} \mathcal{E}(B') &=& -\mathcal{E}(B), \\
\label{ind2.5} \mathcal{E}(A') &=& \mathcal{E}(A), \\
\label{ind2.6} \mathcal{E}(C') &=& \mathcal{E}(B), \\
\label{ind2.7} \mathcal{E}(D') &=& \mathcal{E}(B).  
\end{eqnarray} 
\noindent 
By equation (\ref{ind2.7}), $\mathcal{E}(\sigma_2)$ is unbroken.  
Let $\gamma = \alpha_{a_2}(\widehat{F})$ for $\alpha_{a_2} \in \pi_1$.  By Theorem \ref{DSimple}, $\gamma = 1, 2,$ or $3$.  Since 
\[\alpha_{1}(\widehat{F})+ \alpha_{2}(\widehat{F}) + ... + \alpha_{a_2-1}(\widehat{F}) = -1,\] 
and $\gamma = 1, 2,$ or $3$, $\mathcal{E}(B' \cup E')$ contains $0, 1,$ or $2$.  In particular, 
\[\alpha_{1}(\widehat{F})+ \alpha_{2}(\widehat{F}) + ... + \alpha_{a_2-1}(\widehat{F}) + \alpha_{a_2}(\widehat{F})= \gamma-1.\] 
It is clear that if $a_2$ is even, then the eigenvalues in $\mathcal{E}(B' \cup E') \setminus \mathcal{E}(B')$ are 
\[\{ \gamma, 
\gamma + \alpha_{a_2-1}(\widehat{F}), 
\gamma + \alpha_{a_2-2}(\widehat{F}), ..., 
\gamma + \alpha_{\frac{a_2}{2}+1}(\widehat{F}), 
\gamma - 1, 
\gamma - 1 +  \alpha_{a_2-1}(\widehat{F}), 
\gamma - 1 + \alpha_{a_2-2}(\widehat{F}), ..., 
\gamma - 1 + \alpha_{\frac{a_2}{2}+1}(\widehat{F})
\}.\] 
By Theorem \ref{DSimple}, $|\alpha_i(\widehat{F})| = 0, 1, 2,$ or $3$ for all $i$.  

If $|\alpha_i(\widehat{F})| \neq 3$ for some $i \in \{a_2 - 1, a_2 - 2, ..., \frac{a_2}{2}+1\}$, then $\mathcal{E}(B' \cup E') \setminus \mathcal{E}(B')$ is unbroken.  Moreover, since 
\[\alpha_{a_2+1}(\widehat{F})+ \alpha_{a_2+2}(\widehat{F}) + ... + \alpha_{a_1+a_2-1}(\widehat{F}) = 1,\] 
the multiset
\[\mathcal{E}(B' \cup E' \cup A') \setminus \left[\mathcal{E}(B' \cup E') ~\cup~ \mathcal{E}(A')\right]\] 
is unbroken; but since $\{0, 1\} \subseteq \mathcal{E}(A')$, the multiset 
$\mathcal{E}(B' \cup E' \cup A')$ is unbroken and contains $0$.  Similarly, the multiset 
$\mathcal{E}(A' \cup F' \cup C')$ is unbroken and contains $0$.  Therefore,  
$\mathcal{E}(B' \cup E' \cup A') \cup \mathcal{E}(A' \cup F' \cup C')$ is unbroken and contains $0$.  Any remaining eigenvalue in $\mathcal{E}(\sigma_1)$ is symmetric to an eigenvalue in $\mathcal{E}(B' \cup E' \cup A') \cup \mathcal{E}(A' \cup F' \cup C')$.  Therefore, $\mathcal{E}(\sigma_1)$ is unbroken.  

If $|\alpha_i(\widehat{F})| = 3$ for all $i \in \{a_2 - 1, a_2 - 2, ..., \frac{a_2}{2}+1\}$, then $\mathcal{E}(B' \cup E') \setminus \mathcal{E}(B')$ is not necessarily unbroken.  Moreover, any numbers preventing $\mathcal{E}(B' \cup E') \setminus \mathcal{E}(B')$ from being unbroken must be congruent $(\bmod ~3)$.  
Let $x$ be any such number.  Observe that 
\[\mathcal{E}(B' \cup E') \setminus \mathcal{E}(B') = 
-[\mathcal{E}(C' \cup F') \setminus \mathcal{E}(C')].\]
Then $x$ is symmetric to some number in $\mathcal{E}(C' \cup F') \setminus \mathcal{E}(C')$ and exists somewhere in $\mathcal{E}(\sigma_1)$.  The rest of the argument follows similarly to the previous argument: that is,  when $|\alpha_i(\widehat{F})| \neq 3$ for each $i \in \{a_2 - 1, a_2 - 2, ..., \frac{a_2}{2}+1\}$.  

If $a_2$ is odd, then the eigenvalues in $\mathcal{E}(B' \cup E') \setminus \mathcal{E}(B')$ are the same as in the previous case in addition to $\gamma + \alpha_{\lf\frac{a_2}{2}\rf}(\widehat{F})$.  But since 
\[\alpha_{a_2+1}(\widehat{F})+ \alpha_{a_2+2}(\widehat{F}) + ... + \alpha_{a_1+a_2-1}(\widehat{F}) = 1,\] 
the multiset
\[\mathcal{E}(B' \cup E' \cup A') - \left[\mathcal{E}(B' \cup E') ~\cup~ \mathcal{E}(A')\right]\] 
is unbroken.  That $\mathcal{E}(\sigma_1)$ is unbroken follows from the argument where $a_2$ is even.  
\end{proof}


\begin{remark}
The proof of Theorem \ref{thm:unbroken}, in particular, the proof for Case \ref{unbroken3}, would have held even if the simple eigenvalues from Theorem \ref{DSimple} were bounded in absolute value by $4$.  The type-A proof by Coll et al. in \textbf{\cite{Coll typea}} requires the bound to be $3$.  
\end{remark}

\section{Exceptional cases}\label{Proof of the Structure Theorem - Exceptional Cases}

We now consider seaweeds of the exceptional Lie algebras:
$\E_6$, $\E_7$, $\E_8$, $\F_4$, and $\G_2$.  
We continue to use the Bourbaki ordering of the simple roots.  
We let $\mathfrak{p}^\X_n$ denote a seaweed of exceptional type, where $\X \in \{\E, \F, \G\}$ and $n$ is the rank.  
In this notation, $\mathfrak{p}^\G_2$ is a seaweed of type $\G_2$.  In the case where the pairs of included simple roots are explicit, we use the notation $\mathfrak{p}^\X_n(\pi_1 \dd \pi_2)$.

For components of all exceptional types except $\E_6$, the longest element of $W_{\pi_j}$ is given by $w_j = -id$.  For a component of type $\E_6$, 

\[w_j\alpha_i=
\begin{cases}
\alpha_6, & \text{ if  }i = 1;\\
\alpha_5, & \text{ if  }i = 3;\\
\alpha_i, & \text{ if }i =2, 4.
\end{cases}\]

We visualize these actions and construct an exceptional orbit meander as in the classical cases.  See Example \ref{FrobeniusE6Example}.  

\begin{example}\label{FrobeniusE6Example}
Define the seaweed $\mf{p}_{6}^\E(\Theta_1 \dd \Theta_2)$ by 
$$\Theta_1 = \{\alpha_5, \alpha_4, \alpha_3,\alpha_1\},$$ 
$$\Theta_2 = \{\alpha_6, \alpha_5, \alpha_4, \alpha_3, \alpha_2, \alpha_1\}.$$  
See Figure \ref{FrobeniusE6} for the orbit meander of $\mf{p}_{6}^\E(\Theta_1 \dd \Theta_2)$.  
 Note that this seaweed is Frobenius by Theorem \ref{thm:frobenius}.

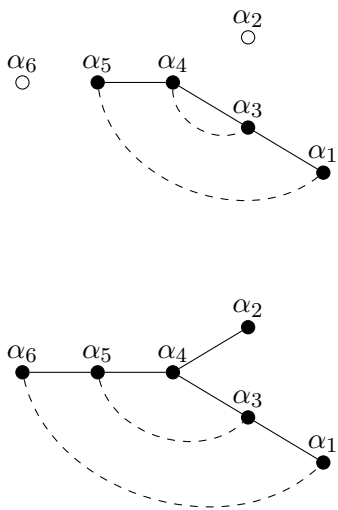
\begin{figure}[H]
\[\begin{tikzpicture}

\draw (1,0) node[draw,circle,fill=white,minimum size=5pt,inner sep=0pt] (1) {};
\draw (2,0) node[draw,circle,fill=black,minimum size=5pt,inner sep=0pt] (2) {};
\draw (3,0) node[draw,circle,fill=black,minimum size=5pt,inner sep=0pt] (3) {};
\draw (4,-.6) node[draw,circle,fill=black,minimum size=5pt,inner sep=0pt] (4) {};
\draw (5,-1.2) node[draw,circle,fill=black,minimum size=5pt,inner sep=0pt] (5) {};
\draw (4,.6) node[draw,circle,fill=white,minimum size=5pt,inner sep=0pt] (6) {};

\draw (2) to (3);
\draw (3) to (4);
\draw (4) to (5);

\draw [dashed] (2) to [bend right=60] (5);
\draw [dashed] (3) to [bend right=60] (4);

\node at (1,.25) {$\alpha_6$};
\node at (2,.25) {$\alpha_5$};
\node at (3,.25) {$\alpha_4$};
\node at (4,-.35) {$\alpha_3$};
\node at (5,-.95) {$\alpha_1$};
\node at (4,.85) {$\alpha_2$};

;\end{tikzpicture}\]

\[\begin{tikzpicture}

\draw (1,0) node[draw,circle,fill=black,minimum size=5pt,inner sep=0pt] (1) {};
\draw (2,0) node[draw,circle,fill=black,minimum size=5pt,inner sep=0pt] (2) {};
\draw (3,0) node[draw,circle,fill=black,minimum size=5pt,inner sep=0pt] (3) {};
\draw (4,-.6) node[draw,circle,fill=black,minimum size=5pt,inner sep=0pt] (4) {};
\draw (5,-1.2) node[draw,circle,fill=black,minimum size=5pt,inner sep=0pt] (5) {};
\draw (4,.6) node[draw,circle,fill=black,minimum size=5pt,inner sep=0pt] (6) {};

\draw (1) to (2);
\draw (2) to (3);
\draw (3) to (4);
\draw (4) to (5);
\draw (3) to (6);

\draw [dashed] (1) to [bend right=60] (5);
\draw [dashed] (2) to [bend right=60] (4);

\node at (1,.25) {$\alpha_6$};
\node at (2,.25) {$\alpha_5$};
\node at (3,.25) {$\alpha_4$};
\node at (4,-.35) {$\alpha_3$};
\node at (5,-.95) {$\alpha_1$};
\node at (4,.85) {$\alpha_2$};

;\end{tikzpicture}\]
\caption{The seaweed $\mf{p}_{6}^\E(\Theta_1 \dd \Theta_2)$}
\label{FrobeniusE6} 
\end{figure}

\end{example}

Again, as in Section~\ref{Principal Elements}, we have the exceptional analogues of Lemma \ref{eigenvalue table}
 and Lemma \ref{eigenvalue orbit}.

\begin{lemma}[Joseph \textbf{\cite{Joseph}}, Section 5]
\label{Exceptionaleignevalue table}
In Table \ref{tab:ExceptionalEigenvalue} below, the given value is 
$\alpha_i(\widehat{F})$ if $\sigma$ is a maximally connected component 
of $\pi_1$, and it is $-\alpha_i(\widehat{F})$ if 
$\sigma$ is a maximally connected component of $\pi_2$.
In either case, it is assumed that $\alpha_i\in\sigma$.
\end{lemma}

\begin{table}[H]
\[\begin{tabular}{|l|l|l|}
\hline
Type & $\pm\alpha_i(\widehat{F})$ & $\pm\alpha_i(\widehat{F})$ \\
\hline
\hline
$E_6$ & -1, if $i=2$ & 1, if $i=4$ \\
\hline
$E_7$ & -1, if $i=1,4,6$ & 1, if $i=2,3,5,7$ \\
\hline
$E_8$ & -1, if $i=1,4,6,8$ & 1, if $i=2,3,5,7$ \\
\hline
$F_4$ & $(-1)^{i}$, if $i=1,2$ & 0, if $i=3,4$ \\
\hline 
$G_2$ & $-1$, if $i=1$ & 1, if $i=2$\\
\hline
\end{tabular}\]
\caption{Values of $\pm\alpha_i(\widehat{F})$}
\label{tab:ExceptionalEigenvalue}
\end{table}

Table \ref{tab:ExceptionalEigenvalue} can be used to find all simple eigenvalues for comopnents not of type $\E_6$.  For components of type $\E_6$, the following lemma 
can be applied.

\begin{lemma}[Joseph \textbf{\cite{Joseph}}, Section 5]
\label{exceptional eigenvalue orbit}
For the equation below, $\sigma$ is assumed to be a component of $\pi_1$ of type $\E_6$.  
If $\sigma$ is a maximally connected component of $\pi_2$,
then replace $\alpha_i\mapsto -\alpha_i$ and $i_1\mapsto i_2$.
Then
\[\alpha_i(\widehat{F})+i_1\alpha_i(\widehat{F})=
0 \text{ if } i=1,3.\]
\end{lemma}

To establish Theorem~\ref{thm:main} for exceptional seaweeds, we first examine seaweeds with components of type $\E_6$ in Section~\ref{E6}. 
Because the Weyl action is non-trivial on an $\E_6$ component, this case requires more analysis. 
The other exceptional Lie algebras are treated in Section~\ref{NotE6}.   

\subsection{The exceptional Lie algebra $\E_6$}\label{E6}


We begin with an example computation. We use the seaweed in Example \ref{FrobeniusE6Example}.  We first compute the simple eigenvalues, then the eigenvalues, and finally the spectrum. 

\begin{example}\label{ExceptionalSimpleEigenvalues}
Figure \ref{SimpleEigenvaluesE6} shows the simple eigenvalues for $\mf{p}_{6}^\E(\Theta_1 \dd \Theta_2)$.  
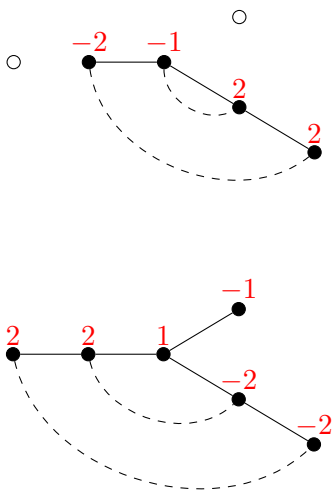
\begin{figure}[H]
\[\begin{tikzpicture}

\draw (1,0) node[draw,circle,fill=white,minimum size=5pt,inner sep=0pt] (1) {};
\draw (2,0) node[draw,circle,fill=black,minimum size=5pt,inner sep=0pt] (2) {};
\draw (3,0) node[draw,circle,fill=black,minimum size=5pt,inner sep=0pt] (3) {};
\draw (4,-.6) node[draw,circle,fill=black,minimum size=5pt,inner sep=0pt] (4) {};
\draw (5,-1.2) node[draw,circle,fill=black,minimum size=5pt,inner sep=0pt] (5) {};
\draw (4,.6) node[draw,circle,fill=white,minimum size=5pt,inner sep=0pt] (6) {};

\draw (2) to (3);
\draw (3) to (4);
\draw (4) to (5);

\draw [dashed] (2) to [bend right=60] (5);
\draw [dashed] (3) to [bend right=60] (4);

\node at (2,.25) [color=red] {$-2$};
\node at (3,.25) [color=red] {$-1$};
\node at (4,-.35) [color=red] {$2$};
\node at (5,-.95) [color=red] {$2$};

;\end{tikzpicture}\]

\[\begin{tikzpicture}

\draw (1,0) node[draw,circle,fill=black,minimum size=5pt,inner sep=0pt] (1) {};
\draw (2,0) node[draw,circle,fill=black,minimum size=5pt,inner sep=0pt] (2) {};
\draw (3,0) node[draw,circle,fill=black,minimum size=5pt,inner sep=0pt] (3) {};
\draw (4,-.6) node[draw,circle,fill=black,minimum size=5pt,inner sep=0pt] (4) {};
\draw (5,-1.2) node[draw,circle,fill=black,minimum size=5pt,inner sep=0pt] (5) {};
\draw (4,.6) node[draw,circle,fill=black,minimum size=5pt,inner sep=0pt] (6) {};

\draw (1) to (2);
\draw (2) to (3);
\draw (3) to (4);
\draw (4) to (5);
\draw (3) to (6);

\draw [dashed] (1) to [bend right=60] (5);
\draw [dashed] (2) to [bend right=60] (4);

\node at (1,.25) [color=red] {$2$};
\node at (2,.25) [color=red] {$2$};
\node at (3,.25) [color=red] {$1$};
\node at (4,-.35) [color=red] {$-2$};
\node at (5,-.95) [color=red] {$-2$};
\node at (4,.85) [color=red] {$-1$};

;\end{tikzpicture}\]
\caption{The simple eigenvalues of $\mf{p}_{6}^\E(\Theta_1 \dd \Theta_2)$}
\label{SimpleEigenvaluesE6} 
\end{figure}
\end{example}

We partition the multiset of eigenvalues according to the maximally connected components $\sigma$ of $\pi_1$ and $\pi_2$.  
Let $\sigma$ be a maximally connected
component of $\pi_1$.  If $\sigma$ is of Type $\E_6$, then
\begin{eqnarray*}\label{SE6}
\mathcal{E}(\sigma)=\{\beta(\widehat{F})\mid\beta\in\mathbb{N}\sigma\cap\Delta_+\}
\cup\{0^4\}. 
\end{eqnarray*}

To compute the eigenvalues of $\mf{p}_{6}^\E(\Theta_1 \dd \Theta_2)$, note that this seaweed has a single type-A component
$\sigma_1=\{\alpha_1,\alpha_3, \alpha_4, \alpha_5  \}$ on the top and a single type-$\E_6$ component on the bottom 
$$\sigma_2=\{\alpha_1,\alpha_2, \alpha_3, \alpha_4, \alpha_5, \alpha_6  \}.$$  We compute $\mathcal{E}(\sigma_1)$ as before to yield
$$\mathcal{E}(\sigma_1) = \{-3^1, -2^1, -1^2, 0^2, 1^2, 2^2, 3^1, 4^1\}.$$
We now compute $\mathcal{E}(\sigma_2)$.  
The positive roots for the computation of $\mathcal{E}(\sigma_2)$ form a thirty-six element set.  See Table \ref{tab:E6 eigenvalue}, where the left column lists the positive roots, and the right column lists the associated eigenvalue.

\newpage
\begin{table}[H]
\[\begin{tabular}{|c|c|}
\hline
Positive Root & Eigenvalue  \\
\hline
\hline
$2\alpha_1 + 2\alpha_2 + 2\alpha_3 + 3\alpha_4 + \alpha_5 + \alpha_6$ & 1 \\
\hline
$2\alpha_1 + \alpha_2 + 2\alpha_3 + 3\alpha_4 + \alpha_5 + \alpha_6$ & 2  \\
\hline
$2\alpha_1 + \alpha_2 + 2\alpha_3 + 2\alpha_4 + \alpha_5 + \alpha_6$  & 1 \\
\hline
$2\alpha_1 + \alpha_2 + \alpha_3 + 2\alpha_4 + \alpha_5 + \alpha_6$ & 3  \\
\hline
$\alpha_1 + \alpha_2 + 2\alpha_3 + 2\alpha_4 + \alpha_5 + \alpha_6$ &  -1 \\
\hline
$2\alpha_1 + \alpha_2 + \alpha_3 + 2\alpha_4 + \alpha_6$ &  5 \\
\hline
$\alpha_1 + \alpha_2 + \alpha_3 + 2\alpha_4 + \alpha_5 + \alpha_6$ & -3  \\
\hline
$\alpha_1 + \alpha_2 + 2\alpha_3 + 2\alpha_4 + \alpha_5$ &  1 \\
\hline
$\alpha_1 + \alpha_2 + \alpha_3 + 2\alpha_4 + \alpha_6$ & -1 \\
\hline
$\alpha_1 + \alpha_2 + \alpha_3 + 2\alpha_4 + \alpha_5$ & 3 \\
\hline
$\alpha_1 + \alpha_2 + \alpha_3 + \alpha_4 + \alpha_5 + \alpha_6$ & 0 \\
\hline
$\alpha_1 + \alpha_2 + \alpha_3 + 2\alpha_4$ & 1 \\
\hline
$\alpha_1 + \alpha_2 + \alpha_3 + \alpha_4 + \alpha_5$ & 1 \\
\hline
$\alpha_1 +\alpha_2 + \alpha_3 + \alpha_4 + \alpha_6$ & -2 \\
\hline
$\alpha_1 + \alpha_3 + \alpha_4 + \alpha_5 + \alpha_6$ &  2 \\
\hline
$\alpha_1 + \alpha_3 + \alpha_4 + \alpha_5$ & -4 \\
\hline
$\alpha_1 + \alpha_2 + \alpha_4 + \alpha_5$ & -1 \\
\hline
$\alpha_1 + \alpha_2 + \alpha_3 + \alpha_4$ & 0  \\
\hline
$\alpha_1 + \alpha_3 + \alpha_4 + \alpha_6$ & 4 \\
\hline
$\alpha_2 + \alpha_3 + \alpha_4 + \alpha_6$ & 3 \\
\hline
$\alpha_1 + \alpha_4 + \alpha_5$ & -3 \\
\hline
$\alpha_1 + \alpha_3 + \alpha_4$ & -2  \\
\hline
$\alpha_1 + \alpha_2 + \alpha_4$ & 2 \\
\hline
$\alpha_2 + \alpha_3 + \alpha_4$ & 1 \\
\hline
$\alpha_3 + \alpha_4 + \alpha_6$ & 5 \\
\hline
$\alpha_1 + \alpha_5$ & -4 \\
\hline
$\alpha_1 + \alpha_4$ & 0 \\
\hline
$\alpha_3 + \alpha_4$ & -1 \\
\hline
$\alpha_2 + \alpha_4$ & 3 \\
\hline
$\alpha_3 + \alpha_6$ & 4 \\
\hline
$\alpha_1$ & -2  \\
\hline
$\alpha_2$ & -1 \\
\hline
$\alpha_3$ & -2 \\
\hline
$\alpha_4$ & 1 \\
\hline
$\alpha_5$ & 2  \\
\hline
$\alpha_6$ & 2 \\
\hline
\end{tabular}\]
\caption{Eigenvalues associated with the component $\sigma_2=\{\alpha_1,\alpha_2, \alpha_3, \alpha_4, \alpha_5, \alpha_6  \}$}
\label{tab:E6 eigenvalue}
\end{table}

\noindent 
Consequently, 
$$\mathcal{E}(\sigma_2) = \{-4^2, -3^2, -2^4, -1^{5}, 0^{3}, 1^7, 2^5, 3^4, 4^2, 5^2\}\cup\{0^4\} = \{-4^2, -3^2, -2^4, -1^{5}, 0^{7}, 1^7, 2^5, 3^4, 4^2, 5^2\}.$$

  


The union of the sets $\mathcal{E}(\sigma_1)$ and $\mathcal{E}(\sigma_2)$ gives the spectrum.  See Table \ref{tab:E6ExampleEigenvalue}. 

\begin{table}[H]
\[\begin{tabular}{|l||l|l|l|l|l|l|l|l|l|l|}
\hline
Eigenvalue & -4 & -3 & -2 & -1 & 0 & 1 & 2 & 3 & 4 & 5 \\
\hline
Multiplicity & 2 & 3 & 5 & 7 & 9 & 9 & 7 & 5 & 3 & 2 \\
\hline
\end{tabular}\]
\caption{Spectrum of $\mf{p}_{6}^\E(\Theta_1 \dd \Theta_2)$ with multiplicities}
\label{tab:E6ExampleEigenvalue}
\end{table}



Using Theorem \ref{thm:frobenius}, it is straightforward to show, by exhaustion, that there are, up to isomorphism, seventy-four Frobenius seaweed subalgebras of $\E_6$.
\footnote{There are, up to isomorphism, two Frobenius seaweeds in $\G_2$, eight in $\F_4$, seventy-four in $\E_6$, one hundred forty-three in $\E_7$, and three hundred one in $\E_8$.}  
Of these, fourteen contain a component of type $\E_6$.  (See Appendix A.)  The simple eigenvalues associated with the $\E_6$ component take on one of nine possible configurations, which we note in Remark~\ref{remark} below.  

\begin{remark}\label{remark}
We list the simple eigenvalues according to the order of simple roots.  Observe that the first configuration occurs in the bottom of the orbit meander of $\mf{p}_{6}^\E(\Theta_1 \dd \Theta_2)$.
\end{remark}

\begin{table}[H]
\[\begin{tabular}{|c|c|}
\hline
Configuration & Simple eigenvalues  \\
\hline
\hline
1 & -2 -1 -2 1 2 2 \\
\hline
2 & -2 -1 1 1 -1 2  \\
\hline
3 & -1 -1 2 1 -2 1 \\
\hline
4 & 1 -1 -2 1 2 -1 \\
\hline
5 & 1 -1 -1 1 1 -1 \\
\hline
6 & 2 -1 -1 1 1 -2 \\
\hline
7 & -1 -1 -1 1 1 1 \\
\hline
8 & -1 -1 1 1 -1 1 \\
\hline
9 & 1 -1 1 1 -1 -1 \\
\hline
\end{tabular}\]
\caption{Configurations and simple eigenvalues for an $\E_6$ component}
\label{tab:E6simple eigenvalue}
\end{table}

For each of the configurations in Table \ref{tab:E6simple eigenvalue}, computations similar to those used in Example \ref{ExceptionalSimpleEigenvalues} can be used to show that 
the spectrum of any $\E_6$ component consists of an unbroken sequence of integers centered at one-half, and the associated eigenspace multiplicities form a symmetric distribution.

We next consider the 
remaining exceptional Lie algebras.

\subsection{The exceptional Lie algebras $\E_7$, $\E_8$, $\F_4$, and $\G_2$}\label{NotE6}

Because Table \ref{tab:ExceptionalEigenvalue} gives all simple eigenvalues in each of these cases, we need only consider the eigenvalues associated to a component of each type.  
We use Lemma \ref{Exceptionaleignevalue table} to find the simple eigenvalues for each component.  

\subsubsection{A component of type $\E_7$}

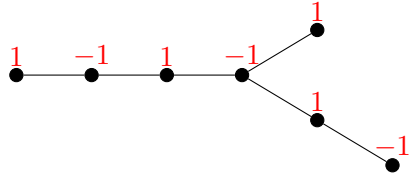
\begin{figure}[H]
\[\begin{tikzpicture}

\draw (0,0) node[draw,circle,fill=black,minimum size=5pt,inner sep=0pt] (0) {};
\draw (1,0) node[draw,circle,fill=black,minimum size=5pt,inner sep=0pt] (1) {};
\draw (2,0) node[draw,circle,fill=black,minimum size=5pt,inner sep=0pt] (2) {};
\draw (3,0) node[draw,circle,fill=black,minimum size=5pt,inner sep=0pt] (3) {};
\draw (4,-.6) node[draw,circle,fill=black,minimum size=5pt,inner sep=0pt] (4) {};
\draw (5,-1.2) node[draw,circle,fill=black,minimum size=5pt,inner sep=0pt] (5) {};
\draw (4,.6) node[draw,circle,fill=black,minimum size=5pt,inner sep=0pt] (6) {};

\draw (0) to (1);
\draw (1) to (2);
\draw (2) to (3);
\draw (3) to (4);
\draw (4) to (5);
\draw (3) to (6);

\node at (0,.25) [color=red] {$1$};
\node at (1,.25) [color=red] {$-1$};
\node at (2,.25) [color=red] {$1$};
\node at (3,.25) [color=red] {$-1$};
\node at (4,-.35) [color=red] {$1$};
\node at (5,-.95) [color=red] {$-1$};
\node at (4,.85) [color=red] {$1$};

;\end{tikzpicture}\]
\caption{The simple eigenvalues of a component $\sigma$ of type $\E_7$}
\label{FrobeniusE7} 
\end{figure}

\noindent
The eigenvalues of $\sigma$ are given by
\begin{eqnarray*}\label{SE7}
\mathcal{E}(\sigma)=\{\beta(\widehat{F})\mid\beta\in\mathbb{N}\sigma\cap\Delta_+\}
\cup\{0^7\}. 
\end{eqnarray*}

\noindent 
The spectrum of $\sigma$ and the multiplicities of associated eigenspaces are listed below (see Table \ref{tab:E7eigenvalue}).  

\begin{table}[H]
\[\begin{tabular}{|l||l|l|l|l|}
\hline
Eigenvalue& -1 & 0 & 1 & 2 \\
\hline
Multiplicity & 7 & 28 & 28 & 7 \\
\hline
\end{tabular}\]
\caption{Spectrum of $\sigma$ with multiplicities}
\label{tab:E7eigenvalue}
\end{table}

\subsubsection{A component of type $\E_8$}

\begin{figure}[H]
\[\begin{tikzpicture}

\draw (-1,0) node[draw,circle,fill=black,minimum size=5pt,inner sep=0pt] (-1) {};
\draw (0,0) node[draw,circle,fill=black,minimum size=5pt,inner sep=0pt] (0) {};
\draw (1,0) node[draw,circle,fill=black,minimum size=5pt,inner sep=0pt] (1) {};
\draw (2,0) node[draw,circle,fill=black,minimum size=5pt,inner sep=0pt] (2) {};
\draw (3,0) node[draw,circle,fill=black,minimum size=5pt,inner sep=0pt] (3) {};
\draw (4,-.6) node[draw,circle,fill=black,minimum size=5pt,inner sep=0pt] (4) {};
\draw (5,-1.2) node[draw,circle,fill=black,minimum size=5pt,inner sep=0pt] (5) {};
\draw (4,.6) node[draw,circle,fill=black,minimum size=5pt,inner sep=0pt] (6) {};

\draw (-1) to (0);
\draw (0) to (1);
\draw (1) to (2);
\draw (2) to (3);
\draw (3) to (4);
\draw (4) to (5);
\draw (3) to (6);

\node at (-1,.25) [color=red] {$-1$};
\node at (0,.25) [color=red] {$1$};
\node at (1,.25) [color=red] {$-1$};
\node at (2,.25) [color=red] {$1$};
\node at (3,.25) [color=red] {$-1$};
\node at (4,-.35) [color=red] {$1$};
\node at (5,-.95) [color=red] {$-1$};
\node at (4,.85) [color=red] {$1$};

;\end{tikzpicture}\]
\caption{The simple eigenvalues of a component $\sigma$ of type $\E_8$}
\label{FrobeniusE8} 
\end{figure}
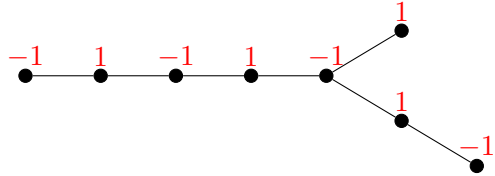

The eigenvalues of $\sigma$ are given by

\begin{eqnarray*}\label{SE8}
\mathcal{E}(\sigma)=\{\beta(\widehat{F})\mid\beta\in\mathbb{N}\sigma\cap\Delta_+\}
\cup\{0^8\}. 
\end{eqnarray*}

\noindent 
The spectrum of $\sigma$ and the multiplicities of associated eigenspaces are listed below (see Table \ref{tab:E8eigenvalue}).  

\begin{table}[H]
\[\begin{tabular}{|l||l|l|l|l|}
\hline
Eigenvalue& -1 & 0 & 1 & 2 \\
\hline
Multiplicity & 14 & 50 & 50 & 14 \\
\hline
\end{tabular}\]
\caption{Spectrum of $\sigma$ with multiplicities}
\label{tab:E8eigenvalue}
\end{table}

\subsubsection{A component of type $\F_4$}







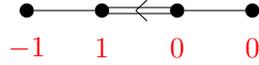
\begin{figure}[H]
\[\begin{tikzpicture}
[decoration={markings,mark=at position 0.6 with 
{\arrow{angle 90}{>}}}]

\draw (1,.75) node[draw,circle,fill=black,minimum size=5pt,inner sep=0pt] (1+) {};
\draw (2,.75) node[draw,circle,fill=black,minimum size=5pt,inner sep=0pt] (2+) {};
\draw (3,.75) node[draw,circle,fill=black,minimum size=5pt,inner sep=0pt] (3+) {};
\draw (4,.75) node[draw,circle,fill=black,minimum size=5pt,inner sep=0pt] (4+) {};

\draw (1+) to (2+);
\draw (3+) to (4+);

\draw (2,.79) -- (3,.79);
\draw (2,.71) -- (3,.71);
\draw (2.6,.89) -- (2.45,.75);
\draw (2.6,.61) -- (2.45,.75);

\node at (1,.25) [color=red] {$-1$};
\node at (2,.25) [color=red] {$1$};
\node at (3,.25) [color=red] {$0$};
\node at (4,.25) [color=red] {$0$};

;\end{tikzpicture}\]
\caption{Simple eigenvalues for a component $\sigma$ of type $\F_4$}
\end{figure}

\noindent
The eigenvalues of $\sigma$ are given by
\begin{eqnarray*}\label{SF}
\mathcal{E}(\sigma)=\{\beta(\widehat{F})\mid\beta\in\mathbb{N}\sigma\cap\Delta_+\}
\cup\{0^4\}. 
\end{eqnarray*}

\noindent
The spectrum of $\sigma$ and the multiplicities of associated eigenspaces are listed below (see Table \ref{tab:F4eigenvalue}).  

\begin{table}[H]
\[\begin{tabular}{|l||l|l|l|l|}
\hline
Eigenvalue& -1 & 0 & 1 & 2 \\
\hline 
Multiplicity & 1 & 13 & 13 & 1 \\
\hline
\end{tabular}\]
\caption{The spectrum of $\sigma$ with multiplicities}
\label{tab:F4eigenvalue}
\end{table}


\subsubsection{A component of type $\G_2$}






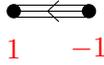
\begin{figure}[H]
\[\begin{tikzpicture}
[decoration={markings,mark=at position 0.6 with 
{\arrow{angle 90}{>}}}]

\draw (1,.75) node[draw,circle,fill=black,minimum size=5pt,inner sep=0pt] (1+) {};
\draw (2,.75) node[draw,circle,fill=black,minimum size=5pt,inner sep=0pt] (2+) {};

\draw (1,.83) -- (2,.83);
\draw (1,.75) -- (2,.75);
\draw (1,.67) -- (2,.67);
\draw (1.6,.89) -- (1.45,.75);
\draw (1.6,.61) -- (1.45,.75);

\node at (1,.25) [color=red] {$1$};
\node at (2,.25) [color=red] {$-1$};

;\end{tikzpicture}\]
\caption{Simple eigenvalues for a component $\sigma$ of type $\G_2$}
\end{figure}

\noindent
The eigenvalues of $\sigma$ are given by 
\begin{eqnarray*}\label{SG}
\mathcal{E}(\sigma)=\{\beta(\widehat{F})\mid\beta\in\mathbb{N}\sigma\cap\Delta_+\}
\cup\{0^2\}. 
\end{eqnarray*}

\noindent
The spectrum of $\sigma$ and the multiplicities of associated eigenspaces are listed below (see Table \ref{tab:G2eigenvalue}).

\begin{table}[H]
\[\begin{tabular}{|l||l|l|l|l|}
\hline
Eigenvalue& -1 & 0 & 1 & 2 \\
\hline
Multiplicity & 1 & 3 & 3 & 1 \\
\hline
\end{tabular}\]
\caption{Spectrum with multiplicities}
\label{tab:G2eigenvalue}
\end{table}








Since the spectrum of any component of exceptional type 
consists of an unbroken sequence of integers centered at one-half, and the associated eigenspace multiplicities form a symmetric distribution, we conclude Theorem \ref{thm:main} holds for the exceptional Lie algebras.

\section{Epilogue}
The unbroken spectrum property of a Frobenius seaweed subalgebra $\mathfrak{g}$ of a simple Lie algebra is a computable algebraic invariant of $\mathfrak{g}$ but it is not characteristic -- as the following example illustrates.

\begin{example}  Consider the poset $\cal P =$ $\{1,2,3,4\}$ with $1,2 \preceq 3 \preceq 4$ and no relations other than those following from these.  Letting 
$\mathbb{C}$ be the ground field, one may construct an 
associative matrix algebra $A(\cal P, \mathbb{C})$ 
which is the span over $\mathbb{C}$ of $e_{i,j}$, $i\preceq j$ with multiplication given by
$e_{i,j}e_{l,k}=e_{i,k}$ if $j=l$ and 0 otherwise.  The underlying vector space of $A(\cal P, \mathbb{C})$ becomes a Lie algebra $\mathfrak{g}(\cal P, \mathbb{C})$ under commutator multiplication and, if one considers only the elements of trace zero, may be regarded as a Lie subalgebra of 
$\A_4=\mathfrak{sl}(4)$:  in fact, a Frobenius Lie subalgebra with Frobenius functional $F= e_{1,4}^* + e_{2,4}^*+e_{2,3}^*$, principal element 
$\widehat{F}=\rm{diag}\left(\frac{1}{2},\frac{1}{2},-\frac{1}{2},-\frac{1}{2}\right)$, and unbroken spectrum $\{0^4,1^4\}$. The algebra $\mathfrak{g}(\cal P, \mathbb{C})$ has rank 3 and dimension 8. However, the only Frobenius seaweed subalgebra of $\mathfrak{sl}(4)$ with the same rank and dimension is
$\mf{p}_4^\A( \{\alpha_3, \alpha_1\} \dd \{\alpha_3, \alpha_2\})$,
  but the latter has spectrum given by the multiset $\{-1,0^3,1^3,2\}$.
  See \textbf{\cite{CM}}. 
\end{example}



\newpage
\section{Appendix A - Frobenius seaweeds in $\E_6$}

We provide a list of Frobenius seaweed subalgebras in $\E_6$.  To ease notation, we will denote, for example,  a seaweed $\mf{p}_6^\E(\{\alpha_5, \alpha_4, \alpha_3,\alpha_1\} \dd \{\alpha_6, \alpha_5, \alpha_4, \alpha_3, \alpha_2, \alpha_1\})$ 
by $\{5,4,3,1\}, \{6,5,4,3,2,1\}$.  Note that the seaweeds listed in 1-14 contain a component of type $\E_6$.  

\begin{center}
\begin{tabular}{ll}

\begin{tabular}{|c|c|}

\hline
1 & $ \{6,5,4,3,2,1\}, \{6,5,4,3\} $ \\
\hline
2 & $ \{6,5,4,3,2,1\}, \{6,5,4,2\}$ \\
\hline
3 & $ \{6,5,4,3,2,1\}, \{5,4,3,1\} $ \\
\hline
4 & $ \{6,5,4,3,2,1\}, \{4,3,2,1\} $ \\
\hline
5 & $ \{6,5,4,3,2,1\}, \{6,5,4\} $ \\
\hline
6 & $ \{6,5,4,3,2,1\}, \{6,5,3\} $ \\
\hline
7 & $ \{6,5,4,3,2,1\}, \{6,5,1\} $ \\
\hline
8 & $ \{6,5,4,3,2,1\}, \{6,4,3\} $ \\
\hline
9 & $  \{6,5,4,3,2,1\}, \{6,3,1\} $ \\
\hline
10 & $ \{6,5,4,3,2,1\}, \{5,4,1\} $ \\
\hline
11 & $ \{6,5,4,3,2,1\}, \{5,3,1\} $ \\
\hline
12 & $  \{6,5,4,3,2,1\}, \{4,3,1\} $ \\
\hline
13 & $ \{6,5,4,3,2,1\}, \{6,3\} $ \\
\hline
14 & $ \{6,5,4,3,2,1\}, \{5,1\} $ \\
\hline
15 & $ \{5,4,3,2,1\}, \{6,5,4\} $ \\
\hline
16 & $ \{6,4,3,2,1\}, \{6,5,3,2,1\}$ \\
\hline
17 & $ \{6,5,3,2,1\}, \{6,4,3,2\}$ \\
\hline
18 & $ \{6,5,4,2,1\}, \{6,5,3,2,1\}$ \\
\hline
19 & $ \{6,5,4,2,1\}, \{6,4,3,2\}$ \\
\hline
20 & $ \{6,5,4,3,1\}, \{6,5,4,2,1\}$ \\
\hline
21 & $ \{6,5,4,3,1\}, \{6,4,3,2,1\}$ \\
\hline
22 & $ \{6,5,4,3,1\}, \{6,5,2,1\}$ \\
\hline
23 & $ \{6,5,4,3,1\}, \{5,2,1\}$ \\
\hline
24 & $ \{6,5,4,3,2\}, \{4,3,2,1\}$ \\
\hline
25 & $ \{6,5,4,3,2\}, \{4,3,1\}$ \\
\hline
26 & $ \{6,5,4,3,2\}, \{4,2,1\}$ \\
\hline
27 & $ \{6,5,4,3,2\}, \{3,2,1\}$ \\
\hline
28 & $ \{6,5,4,3,2\}, \{2,1\}$ \\
\hline
29 & $ \{5,3,2,1\}, \{6,5,4,3,1\}$ \\
\hline
30 & $ \{5,3,2,1\}, \{6,5,4,2,1\}$ \\
\hline
31 & $ \{5,3,2,1\}, \{6,4,3,2,1\}$ \\
\hline
32 & $ \{5,3,2,1\}, \{6,4,3,2\}$ \\
\hline
33 & $ \{5,3,2,1\}, \{6,4,2,1\}$ \\
\hline
34 & $ \{5,3,2,1\}, \{6,4,1\}$ \\
\hline
35 & $ \{6,3,2,1\}, \{6,5,4,3,1\}$ \\
\hline
36 & $ \{6,3,2,1\}, \{5,4,2,1\}$ \\
\hline
37 & $ \{5,4,2,1\}, \{6,5,3,2,1\}$ \\
\hline

\end{tabular}


\hspace{1cm}

&

\begin{tabular}{|c|c|}

\hline
38 & $ \{5,4,2,1\}, \{6,4,3,2,1\} $ \\
\hline
39 & $ \{6,4,2,1\}, \{6,5,3,2,1\}$ \\
\hline
40 & $ \{6,5,2,1\}, \{6,4,3,2\}$ \\
\hline
41 & $ \{5,4,3,1\}, \{6,5,4,3,2\}$ \\
\hline
42 & $ \{6,4,3,1\}, \{6,5,4,2,1\}$ \\
\hline
43 & $ \{6,4,3,1\}, \{6,5,3,2,1\}$ \\
\hline
44 & $ \{6,4,3,1\}, \{6,5,2,1\}$ \\
\hline
45 & $ \{6,4,3,1\}, \{5,3,2,1\}$ \\
\hline
46 & $ \{6,4,3,1\}, \{5,2,1\}$ \\
\hline
47 & $ \{6,5,4,1\}, \{6,5,3,2,1\}$ \\
\hline
48 & $ \{6,5,4,1\}, \{6,4,3,2,1\}$ \\
\hline
49 & $ \{6,5,4,1\}, \{6,3,2,1\}$ \\
\hline
50 & $ \{6,5,3,2\}, \{6,5,4,3,1\}$ \\
\hline
51 & $ \{6,5,3,2\}, \{6,5,4,2,1\}$ \\
\hline
52 & $ \{6,5,3,2\}, \{6,4,3,2,1\}$ \\
\hline
53 & $ \{6,5,3,2\}, \{6,5,4,1\}$ \\
\hline
54 & $ \{6,5,3,2\}, \{6,4,2,1\}$ \\
\hline
55 & $ \{6,5,3,2\}, \{5,4,2,1\}$ \\
\hline
56 & $ \{6,5,3,2\}, \{6,4,1\}$ \\
\hline
57 & $ \{5,4,3,2\}, \{6,5,3,1\}$ \\
\hline
58 & $ \{5,4,3,2\}, \{6,5,1\}$ \\
\hline
59 & $ \{5,4,3,2\}, \{6,3,1\}$ \\
\hline
60 & $ \{6,5,4,2\}, \{5,4,3,2,1\}$ \\
\hline
61 & $ \{6,5,4,3\}, \{5,4,3,2,1\}$ \\
\hline
62 & $ \{5,2,1\}, \{6,4,3,2\}$ \\
\hline
63 & $ \{6,4,1\}, \{6,5,3,2,1\}$ \\
\hline
64 & $ \{5,3,2\}, \{6,5,4,2,1\}$ \\
\hline
65 & $ \{5,3,2\}, \{6,4,3,2,1\}$ \\
\hline
66 & $ \{5,3,2\}, \{6,4,2,1\}$ \\
\hline
67 & $ \{5,3,2\}, \{6,4,1\}$ \\
\hline
68 & $ \{6,3,2\}, \{6,5,4,3,1\}$ \\
\hline
69 & $ \{6,3,2\}, \{6,5,4,1\}$ \\
\hline
70 & $ \{6,3,2\}, \{5,4,2,1\}$ \\
\hline
71 & $ \{6,4,2\}, \{5,4,3,2,1\}$ \\
\hline
72 & $ \{6,5,2\}, \{5,4,3,2,1\}$ \\
\hline
73 & $ \{6,1\}, \{5,4,3,2\}$ \\
\hline
74 & $ \{6,2\}, \{5,4,3,2,1\}$ \\
\hline

\end{tabular}

\label{tab:E6Frobenius1}

\end{tabular}

\end{center}


\begin{thebibliography}{abcd}
\bibitem{BD}
A. Belavin and V. Drinfel'd, Solutions of the
classical Yang–Baxter equation for simple Lie
algebras, \textit{Funktsional. Anal. i Prilozhen} 16(3):1-29, 1982.

\bibitem{DIndex}
A. Cameron, V. Coll, and M. Hyatt, 
Combinatorial index formulas for Lie algebras of seaweed type, 
\textit{Communications in Algebra}, 48(12):5430-5454, 2020.  

\bibitem{Meanders3}
V. Coll, A. Dougherty, M. Hyatt, and N. Mayers, Meander Graphs and Frobenius Seaweed Lie Algebras III, \textit{Journal of Generalized Lie Theory and Applications}, 11: 266. doi: 10.4172/1736-4337, 2017.

\bibitem{UDF1}
V. Coll, M. Gerstenhaber, and A. Giaquinto, An Explicit Deformation Formula with Non-Commuting Derivations, \textit{Ring Theory-Weizmann Science Press}, 396-403, 1989.  

\bibitem{UDF2}
V. Coll, M. Gerstenhaber, and S. Schack, Universal Deformation Formulas and Breaking Symmetry, \textit{Journal of Pure and Applied Algebra}, 90:201-219, 1993.  


\bibitem{Coll typea}
V. Coll, M. Hyatt, and C. Magnant,
The unbroken spectrum of type-A Frobenius seaweeds,
\textit{Journal of Algebraic Combinatorics}, 48(2):289-305, 2017.



\bibitem{CM}
V. Coll and N. Mayers. ``The index of Lie poset algebras." \textit{J. Combin. Theory Ser. A}, 177, 2021.



\bibitem{dk}
V. Dergachev and A. Kirillov, Index of Lie algebras of seaweed type, \textit{J. Lie Theory} 10(2):331-343, 2000. 

\bibitem{DIATTA}
A. Diatta and B. Magna, On properties of principal elements of Frobenius Lie algebras, \textit{J. Lie Theory} 24:849-864, 2014. 


\bibitem{G1}
M. Gerstenhaber and A. Giaquinto, Boundary solutions of the classical Yang-Baxter equation, \textit{Letters Math. Physics} 40:337-353, 1997.

\bibitem{G2}
M. Gerstenhaber and A. Giaquinto, Graphs, Frobenius functionals, and the classical Yang-Baxter equation, arXiv:0808.2423v1, August 18, 2008.

\bibitem{G3}
M. Gerstenhaber and A. Giaquinto,
The Principal Element of a Frobenius Lie Algebra, \textit{Letters Math. Physics} 88(1):333-341, 2009.

\bibitem{Twist}
A. Giaquinto and J. Zhang, Bialgebra Actions, Twists, and Universal Deformation Formulas, \textit{Journal of Pure and Applied Algebra}, 128:133-151, 1998. 

\bibitem{Joseph5}
A. Joseph, 
The minimal orbit in a simple Lie algebra and its associated maximal ideal, \textit{Annales scientifiques de l'\'Ecole Normale Sup\'erieure}, 9(1):1-29, 1976.  

\bibitem{Joseph}
A. Joseph,
\newblock {On semi-invariants and index for biparabolic (seaweed) algebras, I}.
\newblock \textit{J. Algebra}, 305:487--515, 2006.

\bibitem{Joseph2}
A. Joseph, The integrality of an adapted pair.
\textit{Transformation Groups}, 20(3):771-816, 2015.

\bibitem{Joseph3}
A. Joseph and D Shafrir, Polynomiality of invariants, unimodularity and adapted pairs. \textit{Transformation Groups}, 15:851-882, 2010. 

\bibitem{Joseph4}
A. Joseph, The hidden semi-invariants generators of an almost-Frobenius biparabolic.  \textit{Transformation Groups}, 19(3):735-778, 2014.  

\bibitem{Kostant}
B. Kostant, The cascade of othogonal roots and the coadjoint structure of the nilradical of a Borel subgroup of a semisimple Lie group, \textit{Moscow Mathematical Journal}, 12(3):605-620, 2012.  

\bibitem{Ooms1}
A. Ooms, On Lie algebras having a primitive universal enveloping algebra, \textit{J. Algebra}, 32(3):488-500, 1974.

\bibitem{Ooms2}
A. Ooms, On Frobenius Lie algebras, \textit{Communications in Algebra}, 8:13-52, 1980.

\bibitem{Panyushev1}
D. Panyushev, Inductive formulas for the index of seaweed Lie algebras, \textit{Mosc. Math. J.} 1(2):221-241, 2001. 

\bibitem{Jacobson}
T. A. Springer and R. Steinberg, Conjugacy classes, Seminar on Algebraic Groups and Related Topics, Lecture Notes in Math., Springer, Berlin, Heidelberg, and New York.  131:167-266, 1970.

\end{thebibliography}
\end{document}